\documentclass[letterpaper, 10pt,reqno]{amsart}

\usepackage[usenames,dvipsnames]{xcolor}

\usepackage[portrait,margin=3cm]{geometry}
\usepackage{mathrsfs}

\usepackage{amssymb,amsthm,amsfonts,amsbsy,amstext,latexsym}
\usepackage[reqno]{amsmath}

\usepackage{color}
\usepackage{graphicx}
\usepackage{bbm}
\usepackage{comment}
\usepackage{a4wide,graphicx, listings,lscape, epstopdf, units, textcomp, fancyhdr, url, soul, color}
\usepackage{tikz}
\usetikzlibrary{shapes.geometric}
\usepackage[textsize=small]{todonotes}
\usepackage{enumitem}

\usepackage{mathtools}
\mathtoolsset{showonlyrefs}

\definecolor{MyDarkblue}{rgb}{0,0.08,0.50}
\definecolor{Brickred}{rgb}{0.65,0.08,0}

\usepackage{hyperref,upref}
\hypersetup{
	colorlinks=true,       % false: boxed links; true: colored links
	linkcolor=MyDarkblue,          % color of internal links
	citecolor=Brickred,        % color of links to bibliography
	filecolor=red,      % color of file links
	urlcolor=cyan           % color of external links
}

\newtheorem*{theorem*}{Theorem}
\newtheorem{theorem}{Theorem}[section]
\newtheorem{lemma}[theorem]{Lemma}
\newtheorem{proposition}[theorem]{Proposition}
\newtheorem{corollary}[theorem]{Corollary}

\theoremstyle{definition}
\newtheorem{definition}[theorem]{Definition}
\newtheorem{assumption}[theorem]{Assumption}
\newtheorem{remark}[theorem]{Remark}
\newtheorem{example}[theorem]{Example}

\renewcommand{\P}{\mathbb{P}}
\newcommand{\Pv}{\mathbb{P}}

\newcommand{\eps}{\varepsilon}

\newcommand{\s}[1]{\sum_{#1}}

\newcommand{\cA}{\mathcal{A}}

\newcommand{\cN}{\mathcal{N}}\newcommand{\cO}{\mathcal{O}}

\newcommand{\Var}{{\rm Var}}

\newcommand{\CE}{{\mathcal{E}}}

\newcommand{\e}{{\mathrm e}}

\newcommand{\R}{\mathbb{R}}
\newcommand{\N}{\mathbb{N}}
\newcommand{\Z}{\mathbb{Z}}

\renewcommand{\emptyset}{\varnothing}

\newcommand{\CA}{\mathcal {A}}
\newcommand{\CB}{\mathcal {B}}
\newcommand{\CD}{\mathcal {D}}

\newcommand{\CL}{\mathcal {L}}
\newcommand{\CM}{\mathcal {M}}
\newcommand{\CN}{\mathcal {N}}

\newcommand{\CP}{\mathcal {P}}

\newcommand{\CX}{\mathcal {X}}

\newcommand*{\wt}{\widetilde}

\newcommand*{\be}{\begin{equation}}
	\newcommand*{\ee}{\end{equation}}
\newcommand*{\ba}{\begin{aligned}}
	\newcommand*{\ea}{\end{aligned}}
\newcommand*{\barr}{\begin{array}{c}}
	\newcommand*{\earr}{\end{array}}
\def \toinp    {\buildrel {\Pv}\over{\longrightarrow}}
\def \toindis  {\buildrel {d}\over{\longrightarrow}}
\def \toas     {\buildrel {a.s.}\over{\longrightarrow}}

\newcommand*{\ind}{\mathbbm{1}}
\makeatletter
\def\namedlabel#1#2{\begingroup
	#2%
	\def\@currentlabel{#2}%
	\phantomsection\label{#1}\endgroup
}

\makeatother
\newcommand{\bes}{\begin{equation*}}
	\newcommand{\ees}{\end{equation*}}

\renewcommand{\P}[1]{\mathbb{P}\!\left(#1\right)}
\newcommand{\E}[1]{\mathbb{E}\left[#1\right]}

\newcommand{\F}{W}
\newcommand{\G}{\mathcal{G}}
\newcommand{\Zm}{\mathcal{Z}}
\renewcommand{\N}{\mathbb{N}}

\newcommand{\Ef}[2]{\mathbb{E}_\F#1[#2#1]}
\newcommand{\Pf}[1]{\mathbb{P}_\F\!\left(#1\right)}

\renewcommand{\d}{\mathrm{d}}

\newcommand{\zni}{\Zm_n(i)}

\newcommand{\inn}{i\in[n]}

\numberwithin{equation}{section}

\renewcommand{\e}{\mathrm{e}}

\makeatletter
\newcommand{\leqnomode}{\tagsleft@true\let\veqno\@@leqno}
\newcommand{\reqnomode}{\tagsleft@false\let\veqno\@@eqno}
\makeatother
\newlength{\tagmarginsep} % Distance required
\tikzstyle{vertex}=[circle,fill=orange!60,minimum size=10pt,inner sep=0pt]
\tikzstyle{tedge} = [draw,ultra thick,->,>=stealth, orange]
\tikzstyle{esq}=[circle,fill=white,minimum size=10pt,inner sep=0pt]
\tikzstyle{up}=[<-,>=stealth]

\begin{document}
	
	\title[Location of high-degree vertices in WRG with bounded weights]{The location of high-degree vertices in weighted recursive graphs with bounded random weights}

	\date{\today}
	
	\author[Lodewijks]{Bas Lodewijks}
	\address{Universit\"at Augsburg, Universit\"atsstr. 14, 86159 Augsburg, Germany}
	\email{bas.lodewijks@uni-a.de}

	\begin{abstract}
		We study the asymptotic growth rate of the labels of high-degree vertices in weighted recursive graphs (WRG) when the weights are independent, identically distributed, almost surely bounded random variables, and as a result confirm a conjecture by Lodewijks and Ortgiese~\cite{LodOrt21}. WRGs are a generalisation of the random recursive tree (RRT) and directed acyclic graph model (DAG), in which vertices are assigned vertex-weights and where new vertices attach to $m\in\N$ predecessors, each selected independently with a probability proportional to the vertex-weight of the predecessor.  Prior work established the asymptotic growth rate of the maximum degree of the WRG model and here we show that there exists a critical exponent $\mu_m$, such that the typical label size of the maximum degree vertex equals $n^{\mu_m(1+o(1))}$ almost surely as $n$, the size of the graph, tends to infinity. These results extend and improve on the asymptotic behaviour of the location of the maximum degree, formerly only known for the RRT model, to the more general weighted multigraph case of the WRG model. Moreover, for the Weighted Recursive Tree (WRT) model, that is, the WRG model with $m=1$, we prove the joint convergence of the rescaled degree and label of high-degree vertices under additional assumptions on the vertex-weight distribution, and also extend results on the growth rate of the maximum degree obtained by Eslava, Lodewijks and Ortgiese~\cite{EslLodOrt21}. 
	\end{abstract}
	
	\maketitle 
	
	\section{Introduction}\label{sec:intr}
	
	The Weighted Recursive Graph model (WRG) is a weighted multigraph generalisation of the random recursive tree model in which each vertex has a (random) weight and out-degree $m\in\N$. The graph process $(\G_n,n\in\N)$ is initialised with a single vertex $1$ with vertex-weight $W_1$, and at every step $n\geq 2$ vertex $n$ is assigned vertex-weight $W_n$ and $m$ half-edges and is added to the graph. Conditionally on the weights, each half-edge is then independently connected to a vertex $i$ in $\{1,\ldots,n-1\}$ with probability $W_i/\sum_{j=1}^{n-1}W_j$. The case $m=1$ yields the Weighted Recursive Tree model (WRT), first introduced by Borovkov and Vatutin~\cite{BorVat06,BorVat206}. In this paper we are interested in the \emph{asymptotic behaviour of the vertex labels} of vertices that attain the \emph{maximum degree} in the graph, when the vertex-weights are \emph{i.i.d.\ bounded random variables}. This was formerly only known for the random recursive tree model~\cite{BanBha20}, a special case of the WRT which is obtained when $W_i=1$ for all $i\in\N$. 
	
	After the introduction of the WRT model by Borovkov and Vatutin, Hiesmayr and I\c slak studied the height, depth and size of the tree branches of this model. Mailler and Uribe Bravo~\cite{MaiBra19}, as well as S\'enizergues~\cite{Sen19} and S\'enizergues and Pain~\cite{PainSen21,PainSen22} studied the weighted profile and height of the WRT model. Mailler and Uribe Bravo consider random vertex-weights with particular distributions, whereas S\'enizergues and Pain allow for a more general model with both sequences of deterministic as well as random weights.
	
	Iyer~\cite{Iyer20} and the more general work by Fountoulakis and Iyer~\cite{FouIyer21} study the degree distribution of a large class of evolving weighted random trees, of which the WRT model is a particular example, and Lodewijks and Ortgiese~\cite{LodOrt21} study the degree distribution of the WRG model. In both cases, an almost sure limiting degree distribution for the empirical degree distribution is identified. Lodewijks and Ortgiese~\cite{LodOrt21} also study the maximum degree and the labels of the maximum degree vertices of the WRG model for a large range of vertex-weight distributions. In particular, we distinguish two main cases in the behaviour of the maximum degree: when the vertex-weight distribution has unbounded support or bounded support. In the former case the behaviour and size of the label of maximum degree vertices is mainly controlled by a balance of vertices being old (i.e.\  having a small label) and having a large vertex-weight. In the latter case, due to the fact that the vertex-weights are bounded, the behaviour is instead controlled by a balance of vertices being old and having a degree which significantly exceeds their expected degree.
	
	Finally, Eslava, Lodewijks and Ortgiese~\cite{EslLodOrt21} describe the asymptotic behaviour of the maximum degree in the WRT model in more detail (compared to~\cite{LodOrt21}) when the vertex-weights are i.i.d.\ bounded random variables, under additional assumptions on the vertex-weight distribution. In particular, we outline several classes of vertex-weight distributions for which different higher-order behaviour is observed. 
	
	In this paper we identify the growth rate of the labels of vertices that attain the maximum degree, assuming only that the vertex-weights are almost surely bounded. If we set
	\be 
	\theta_m:=1+\E W/m \text{ and } \mu_m:=1-(\theta_m-1)/(\theta_m\log\theta_m),
	\ee 
	we show that the labels of vertices that attain the maximum degree are \emph{almost surely} of the order $n^{\mu_m(1+o(1))}$. This confirms a conjecture by Lodewijks and Ortgiese~\cite[Conjecture $2.11$]{LodOrt21}, improves a recent result of Banerjee and Bhamidi~\cite{BanBha20} for the location of the maximum degree in the random recursive tree model (which is obtained by setting $\E W=1,m=1$ so that $\mu_1=1-1/(2\log 2)$) from convergence in probability to almost sure convergence, and extends their result to the WRG model. Furthermore, for the WRT model, that is the case $m=1$, under an additional assumption on the vertex-weight distribution, we are able to provide a central limit theorem for the rescaled labels of uniform vertices $v_1,\ldots, v_k$ with $k\in\N$, conditionally on the event that the in-degree of vertex $v_i$ is at least $d_i$ for each $i\in[k]$, for a range of values of the $d_i$. Finally, for several specific cases of vertex-weight distribution, we prove the joint convergence of the rescaled degree and label of high-degree vertices to a marked point process. The points in this marked point process are defined in terms of a Poisson point process on $\R$ and the marks are Gaussian random variables. These additional assumptions on the vertex-weight distribution are similar to the assumptions made by Eslava, Lodewijks and Ortgiese in~\cite{EslLodOrt21} to provide higher-order asymptotic results for the growth rate of the maximum degree in the WRT model, but relax a particular technical condition used in~\cite{EslLodOrt21}, and our results allow for an extension of their results as well.
	
	\textbf{Notation.}	Throughout the paper we use the following notation: we let $\N:=\{1,2,\ldots\}$ denote the natural numbers, set $\N_0:=\{0,1,\ldots\}$ to include zero and let $[t]:=\{i\in\N: i\leq t\}$ for any $t\geq 1$. For $x\in\R$, we let $\lceil x\rceil:=\inf\{n\in\Z: n\geq x\}$ and $\lfloor x\rfloor:=\sup\{n\in\Z: n\leq x\}$. For $x\in \R, k\in\N$, we let $(x)_k:=x(x-1)\cdots (x-(k-1))$ and $(x)_0:=1$ and use the notation $\bar d$ to denote a $k$-tuple $d=(d_1,\ldots, d_k)$ (the size of the tuple will be clear from the context), where the $d_1,\ldots, d_k$ are either numbers or sets. For sequences $(a_n,b_n)_{n\in\N}$ such that $b_n$ is positive for all $n$ we say that $a_n=o(b_n), a_n=\omega(b_n), a_n\sim b_n, a_n=\mathcal{O}(b_n)$ if $\lim_{n\to\infty} a_n/b_n=0,\lim_{n\to\infty} |a_n|/b_n=\infty, \lim_{n\to\infty} a_n/b_n=1$ and if there exists a constant $C>0$ such that $|a_n|\leq Cb_n$ for all $n\in\N$, respectively. For random variables $X,(X_n)_{n\in\N}$ we let $X_n\toindis X, X_n\toinp X$ and $X_n\toas X$ denote convergence in distribution, probability and almost sure convergence of $X_n$ to $X$, respectively. We let $\Phi:\R\to(0,1)$ denote the cumulative density function of a standard normal random variable and for a set $B\subseteq \R$ we abuse this notation to also define $\Phi(B):=\int_B \phi(x)\,\d x$, where $\phi(x):= \Phi'(x)$ denotes the probability density function of a standard normal random variable. It will be clear from the context which of the two definitions is to be applied. Finally, we use the conditional probability measure $\Pf{\cdot}:=\mathbb{P}(\,\cdot\, |(\F_i)_{i\in\N})$ and conditional expectation $\Ef{}{\cdot}:=\E{\,\cdot\,|(\F_i)_{i\in\N}}$, where the $(W_i)_{i\in\N}$ are the i.i.d.\ vertex-weights of the WRG model.
	
	\newpage
	
	\section{Definitions and main results}\label{sec:def}
	
	We define the weighted recursive graph (WRG) as follows:
	
	\begin{definition}[Weighted Recursive Graph]\label{def:wrg}
		Let $(\F_i)_{i\geq 1}$ be a sequence of i.i.d.\ copies of a random variable $\F$ such that $\P{\F>0}=1$, let $m\in\N$, and set
		\be 
		S_n:=\sum_{i=1}^n\F_i.
		\ee
		We construct the \emph{Weighted Recursive Graph} as follows:
		\begin{enumerate}
			\item[1)] Initialise the graph with a single vertex $1$, the root, and assign to the root a vertex-weight $\F_1$. We let $\G_1$ denote this graph. .
			\item[2)] For $n\geq 1$, introduce a new vertex $n+1$ and assign to it the vertex-weight $\F_{n+1}$ and $m$ half-edges. Conditionally on $\G_n$, independently connect each half-edge to some vertex $\inn$ with probability $\F_i/S_n$. Let $\G_{n+1}$ denote this graph.
		\end{enumerate}
		We treat $\G_n$ as a directed graph, where edges are directed from new vertices towards old vertices. Moreover, we assume throughout this paper that the vertex-weights are bounded almost surely.
	\end{definition}
	
	\begin{remark}\label{remark:def} $(i)$ Note that the edge connection probabilities remain unchanged if we multiply each weight by the same constant. In particular, we assume without loss of generality (in the case of bounded vertex-weights) that $x_0:=\sup\{x\in\R\,|\, \P{\F\leq x}<1\}=1$.
		
		$(ii)$ It is possible to extend the definition of the WRG to the case of \emph{random out-degree}. Namely, we can allow that vertex $n+1$ connects to $\emph{every}$ vertex $\inn$ independently with probability $\F_i/S_n$, and the results presented in this paper still hold under this extension.
	\end{remark}
	
	Throughout, for any $n\in\N$ and $\inn$, we write 
	\be 
	\zni:=\text{in-degree of vertex $i$ in }\G_n.
	\ee 
	This paper presents the asymptotic behaviour of the labels of high-degree vertices, the maximum degree vertices in particular. To that end, we define
	\be \label{eq:in}
	I_n:=\inf\{\inn: \zni\geq \Zm_n(j)\text{ for all }j\in[n]\}.
	\ee  
	We now present our first result, which confirms~\cite[Conjecture $2.11$]{LodOrt21}:
	
	\begin{theorem}[Labels of the maximum degree vertices]\label{thrm:bddloc}
		Consider the WRG model as in Definition~\ref{def:wrg} with vertex-weights $(W_i)_{i\in\N}$, which are i.i.d.\ copies of a positive random variable $W$ such that $x_0:=\sup\{x>0:\P{W\leq x}<1\}=1$. Let $\theta_m:=1+\E W/m$ and recall $I_n$ from~\eqref{eq:in}. Then, 
		\be 
		\frac{\log I_n}{\log n}\toas 1-\frac{\theta_m-1}{\theta_m\log\theta_m}=:\mu_m.
		\ee 		
	\end{theorem}
	
	\begin{remark}\label{remark:thrmbddloc}
		$(i)$ The result also holds for $\wt I_n:=\sup\{i\in\N: \zni\geq \Zm_n(j)\text{ for all }j\in[n]\}$, so that \emph{all} vertices that attain the maximum degree have a label that is almost surely of the order $n^{\mu_m(1+o(1))}$. In fact, the result holds for vertices with `near-maximum' degree as well. That is, for vertices with degree $\log_{\theta_m}\!n-i_n$, where $i_n\to\infty$ and $i_n=o(\log n)$.
		
		$(ii)$ As discussed in Remark~\ref{remark:def}$(ii)$, the result presented in Theorem~\ref{thrm:bddloc} also holds, including the additional results discussed in point $(i)$ above, when considering the case of \emph{random out-degree}.
	\end{remark} 
	
	When we consider the Weighted Recursive Tree model (WRT), that is, the WRG model as in Definition~\ref{def:wrg} with $m=1$, we can provide higher-order results for the labels of high-degree vertices. Here, high-degree means that the degree diverges with $n$. These results are known for the random recursive tree model already, as proved by the author in~\cite{Lod22}. To extend this to the more general WRT model, additional assumptions on the vertex-weight distribution are required to prove these higher-order results, which are as follows.
	
		\begin{assumption}[Vertex-weight distribution]\label{ass:weights}
			The vertex-weights $W,(W_i)_{i\in\N}$ are i.i.d.\ strictly positive random variables, whose distribution satisfies the following condition:
			\begin{itemize}
				\item[\namedlabel{item:c1}{\textbf{C1}}] The essential supremum of the distribution is one; $x_0:=\sup\{x\in \R:\P{W\leq x}<1\}=1$.
			\end{itemize}
			Additionally, we may require the following conditions:
			\begin{itemize}
				\item[\namedlabel{item:c2}{\textbf{C2}}] There exist $a,c_1>0,\tau\in(0,1),$ and $x_0\geq 1$ such that $\P{W\geq 1-1/x}\geq a\e^{-c_1x^\tau}$ for all $x\geq x_0$.
				\item[\namedlabel{item:c3}{\textbf{C3}}]There exist $C,\rho>0,$ and $x_0\in(0,1)$ such that $\P{W\leq x}\leq Cx^\rho$ for all $x\in[0,x_0]$.
			\end{itemize}
			Finally, we may assume the vertex-weights satisfy one of the following cases:
			\begin{enumerate}[labelindent = 1cm, leftmargin = 2.2cm]
				\item[\namedlabel{ass:weightatom}{$($\textbf{Atom}$)$}] The vertex weights follow a distribution that has an atom at one, i.e.\ there exists $q_0\in(0,1]$ such that $\P{W=1}=q_0$. Note that $q_0=1$ recovers the random recursive tree model.
				\item[\namedlabel{ass:beta}{$($\textbf{Beta}$)$}] The vertex weights follow a beta distribution: for some $\alpha,\beta>0$ and with $\Gamma$ the gamma function, 
				\be \label{eq:betacdf}
				\P{W\geq x}=\int_x^1 \frac{\Gamma(\alpha+\beta)}{\Gamma(\alpha)\Gamma(\beta)}s^{\alpha-1}(1-s)^{\beta-1}\,\d s, \qquad x\in[0,1].
				\ee 
				\item[\namedlabel{ass:gamma}{$($\textbf{Gamma}$)$}] The vertex weights follow a distribution that satisfies, for some $b\in\R,c_1>0$, and $\tau\geq 1$ such that $b\leq 0$ if $\tau>1$ or $bc_1\leq 1$ when $\tau=1$, 
				\be \label{eq:gumbex}
				\P{W\geq x}=(1-x)^{-b}\e^{-(x/(c_1(1-x)))^\tau}, \qquad x\in[0,1).
				\ee 
			\end{enumerate}  
		\end{assumption}
		
		\begin{remark}\label{rem:ass}
			$(i)$ Condition~\ref{item:c1} naturally follows from the model definition, and is also stated in Remark~\ref{remark:def}(i). Condition~\ref{item:c2} provides a family of vertex-weight distributions for which we can prove a central limit theorem-type result for the labels of high-degree vertices. Informally, for vertex-weights with a tail distribution that decays at a sub-exponential rate as it approaches one, it holds that
			\be 
			\P{\Zm_n(v)\geq d,v\geq n\exp(-(1-\theta^{-1})d+x\sqrt{(1-\theta^{-1})^2d})}\approx \P{\Zm_n(v)\geq d}(1-\Phi(x)),
			\ee 
			where $\theta:=\theta_1=1+\E W$, $v$ is a vertex selected uniformly at random from $[n]$, $x\in\R$ is fixed, and $d=d(n)$ is an integer-valued sequence that diverges with $n$. This general result can be used to prove the desired result.
			
			Condition~\ref{item:c3} follows from~\cite{EslLodOrt21}. There, this condition is necessary to be able to precisely determine the asymptotic behaviour of $\P{\Zm_n(v)\geq d}$, where $d=d(n)\in\N$ is an integer-valued sequence and $v$ is a vertex selected uniformly at random from $[n]$. It is only needed here in a part of Theorem~\ref{thrm:conddegloc}. 
			
			$(ii)$ The~\ref{ass:gamma} case derives its name from the fact that $X:=(1-W)^{-1}$ is distributed as a gamma random variable, conditionally on $X\geq 1$. The condition on the parameters ensures that the probability density function is non-negative on $[0,1)$.
			
			$(iii)$ We observe that both the~\ref{ass:weightatom} and~\ref{ass:beta} cases satisfy Conditions~\ref{item:c1} and~\ref{item:c2}, whereas the~\ref{ass:gamma} case does not satisfy Condition~\ref{item:c2}. Indeed, the behaviour observed in the latter case is different from vertex-weight distributions that do satisfy Condition~\ref{item:c2}. More broadly speaking, from the perspective of extreme value theory, any distribution that falls within the Weibull maximum domain of attraction satisfies condition~\ref{item:c2} (e.g.\ the beta distribution), as well as a large range of distributions with bounded support that fall in the Gumbel maximum domain of attraction (e.g.\ $W=1-1/X$, with $X$ a log-normal random variable, conditionally on $X\geq 1$). An example of a vertex-weight whose distribution does not satisfy Condition~\ref{item:c2} is $W=1-1/X$, where $X$ is a standard normal, conditionally on $X\geq 1$, which is similar to the~\ref{ass:gamma} case with $\tau=2$. For a more precise classification of these domains, we refer to~\cite{Res13} for more details.
		\end{remark}

	The following result identifies the rescaling of the label of high-degree vertices (where high-degree denotes a degree that diverges to infinity with $n$). In particular, it outlines behaviour outside of the range of Theorem~\ref{thrm:bddloc}, both for degrees that are smaller as well as degrees that are \emph{larger} than the maximum degree.
	
	\begin{theorem}[Central limit theorem for high-degree vertex labels]\label{thrm:conddegloc}
		Consider the WRT model, that is, the WRG model as in Definition~\ref{def:wrg} with $m=1$, with vertex-weights $(W_i)_{i\in\N}$ which satisfy conditions~\ref{item:c1} and~\ref{item:c2} in Assumption~\ref{ass:weights}. Fix $k\in\N$, let $(d_i)_{i\in[k]}$ be $k$ integer-valued sequences that diverge as $n\to\infty$ and define
			\be \label{eq:ci}
			c_i:=\limsup_{n\to\infty}\frac{d_i}{\log n},\qquad i\in[k].
			\ee 
			First, assume $c_i\in[0,1/\log\theta)$ for all $i\in[k]$. Then, the tuple  
		\be 
		\Big(\frac{\log v_i -(\log n-(1-\theta^{-1})d_i)}{\sqrt{(1-\theta^{-1})^2d_i}}\Big)_{i\in[k]},
		\ee 
		conditionally on the event $\Zm_n(v_i)\geq d_i$ for all $i\in[k]$, converges in distribution to $(M_i)_{i\in[k]}$, which are $k$ independent standard normal random variables. If we additionally assume that Condition~\ref{item:c3} of Assumption~\ref{ass:weights} holds, then the result holds for $(c_i)_{i\in[k]}\in[1/\log\theta,\theta/(\theta-1))^k$ as well.
	\end{theorem}
	 
		\begin{remark} 
			$(i)$ Theorem~\ref{thrm:conddegloc} covers vertex-weight distributions that fall in the~\ref{ass:weightatom} and \ref{ass:beta} cases as well. As observed in Remark~\ref{rem:ass}$(iii)$, such distributions already satisfy conditions~\ref{item:c1} and~\ref{item:c2} (the other families of distributions outlined in $(iii)$ are also covered by Theorem~\ref{thrm:conddegloc}). 
			
			$(ii)$ Condition~\ref{item:c3} allows us to extend Theorem~\ref{thrm:conddegloc} to a wider range of degrees $d_i$, as it enables us to use~\cite[Proposition $5.1$]{EslLodOrt21} (Proposition~\ref{lemma:degprobasymp} here). This result provides an asymptotic expression for $\P{\Zm_n(v)\geq d}$, where $v$ is a vertex selected uniformly at random from $[n]$. This result can be avoided when the degrees $d_i$ are not too large (i.e.\ $\ll \log(n)/\log \theta$), so that Condition~\ref{item:c3} is not required in those cases. We observe that the~\ref{ass:beta} case satisfies condition~\ref{item:c3}.
		\end{remark}
	
	The following corollary is an immediate result from Theorem~\ref{thrm:conddegloc}
	
	\begin{corollary}\label{cor:conddegloc}
		With the same definitions and assumptions as in Theorem~\ref{thrm:conddegloc}, additionally assume that for each $i\in[k]$,
		\be 
		|d_i-c_i\log n|=o(\sqrt{\log n}).
		\ee 
		Then, the tuple
		\be 
		\Big(\frac{\log v_i -(1-c_i(1-\theta^{-1}))\log n}{\sqrt{c_i(1-\theta^{-1})^2\log n}}\Big)_{i\in[k]},
		\ee 
		conditionally on the event $\Zm_n(v_i)\geq d_i$ for all $i\in[k]$, converges in distribution to $(M_i)_{i\in[k]}$, which are $k$ independent standard normal random variables. Assuming Condition~\ref{item:c3} of Assumption~\ref{ass:weights} holds allows us to extend the result to $c_i\in[1/\log \theta,\theta/(\theta-1))$ for all $i\in[k]$ as well.
	\end{corollary}
	
	\begin{remark}\label{rem:degthrm}
		In both Theorem~\ref{thrm:conddegloc} and Corollary~\ref{cor:conddegloc}, the same results can be obtained when working with the conditional event $\{\Zm_n(v_i)= d_i, i\in[k]\}$ rather than $\{\Zm_n(v_i)\geq d_i, i\in[k]\}$, with an almost identical proof. 
	\end{remark} 
	
	Theorem~\ref{thrm:conddegloc} is very general, in the sense that  Condition~\ref{item:c2} is a mild condition satisfied by a wide range of distributions. In contrast, the behaviour of the maximum degree is much more dependent on the precise behaviour of the vertex-weight distribution (see, for example,~\cite[Theorems $2.6$ and $2.7$]{EslLodOrt21}). The labels of high-degree vertices are much less influenced by the underlying vertex-weight distribution. We provide an heuristic explanation of this fact in Section~\ref{sec:heur}.
	
	When more precise information regarding the vertex-weight distribution is available, as in the \ref{ass:weightatom}, \ref{ass:beta}, and~\ref{ass:gamma} cases, even more can be proved. We state a result for the~\ref{ass:weightatom} case here. It shows the distributional convergence of degrees and their labels in the WRT under proper rescaling. Let us set $\theta:=\theta_1,\mu:=\mu_1=1-(\theta-1)/(\theta\log\theta)$ and define $\sigma^2:=1-(\theta-1)^2/(\theta^2\log \theta)$.
	
	\begin{theorem}[Degrees and labels in the~\ref{ass:weightatom} case]\label{thrm:deglocwrt}
		Consider the WRT model, that is, the WRG model as in Definition~\ref{def:wrg} with $m=1$, with vertex-weights $(W_i)_{i\in\N}$ which satisfy the~\ref{ass:weightatom} case in Assumption~\ref{ass:weights}. Let $v^1,v^2,\ldots, v^n$ be the vertices in the tree in decreasing order of their in-degree $($\!where ties are split uniformly at random$)$, let $d_n^i$ and $\ell_n^i$ denote the in-degree and label of $v^i$, respectively, for $\inn$, and fix $\eps\in[0,1]$. Let $\eps_n:=\log_\theta n-\lfloor \log_\theta n\rfloor$, and let $(n_j)_{j\in\N}$ be a positive, diverging, integer sequence such that $\eps_{n_j}\to\eps$ as $j\to\infty$. Finally, let $(P_i)_{i\in\N}$ be the points of the Poisson point process $\CP$ on $\R$ with intensity measure $\lambda(x)=q_0\theta^{-x}\log \theta\,\d x$, ordered in decreasing order, and let $(M_i)_{i\in\N}$ be a sequence of i.i.d.\ standard normal random variables. Then, as $j\to\infty$, 
		\be
		\Big(d_{n_j}^i-\lfloor \log_\theta n_j\rfloor, \frac{\log(\ell_{n_j}^i)-\mu\log n_j}{\sqrt{(1-\sigma^2)\log n_j}}\Big)_{i\in[n_j]}\toindis (\lfloor P_i+\eps\rfloor , M_i)_{i\in\N}.
		\ee 
	\end{theorem}
	
	\begin{remark}\label{rem:highatom}
		We can view the convergence result in Theorem~\ref{thrm:deglocwrt} in terms of the weak convergence of marked point processes. Indeed, we can order the points in the marked point process 
		\be 
		\CM\CP^{(n)}:=\sum_{i=1}^n \delta_{(\zni-\lfloor \log_\theta n\rfloor, (\log i-\mu\log n)/\sqrt{(1-\sigma^2)\log n})}, 
		\ee 
		in decreasing order with respect to the first argument of the tuples, where $\delta$ is a Dirac measure. 
		
		Moreover, Theorem~\ref{thrm:deglocwrt} extends~\cite[Theorem $2.5$]{EslLodOrt21} to a wider range of vertex-weight distributions. Namely, let us define $\Z^*:=\Z\cup\{\infty\}$ and $\CM^\#_{\Z^*\times \R}, \CM^\#_{\Z^*},$ to be the spaces of boundedly finite measures on $\Z^*\times \R$ and $\Z^*$, respectively, and define  $T:\CM^\#_{\Z^*\times \R}\to \CM^\#_{\Z^*}$ for $\CM\CP\in \CM^\#_{\Z^*\times \R}$ by $T( \CM\CP):= \sum_{(x_1,x_2)\in \CM\CP}\delta_{x_1}$. $T( \CM\CP)$ is the restriction of marked processes $\CM\CP$ to its first coordinate, i.e.\ to the ground process $\CP:=T(\CM\CP)$. Since $T$ is continuous and $\CM\CP^{(n)}\in \CM^\#_{\Z^*\times \R}$, it follows from the continuous mapping theorem that Theorem~\ref{thrm:deglocwrt} implies Theorem $2.5$ in~\cite{EslLodOrt21} without the need of Condition~\ref{item:c3}.
	\end{remark} 
	
	Similar results hold in the~\ref{ass:beta} case as well. In the~\ref{ass:gamma} case slightly different behaviour is observed, where additional higher-order terms in the rescaling of the labels of high-degree vertices are required. We have deferred the results regarding these two cases to Section~\ref{sec:examples}, since they are similar in nature to Theorems~\ref{thrm:conddegloc} and~\ref{thrm:deglocwrt} but of independent interest. Moreover, the results in Theorems~\ref{thrm:conddegloc} and~\ref{thrm:deglocwrt}, as well as the results presented in Section~\ref{sec:examples}, also hold when we consider the WRT with \emph{random out-degree}, as discussed in Remark~\ref{remark:def}.

	\textbf{Discussion, open problems and outline of the paper}\\
	For the proof of Theorem~\ref{thrm:bddloc}, only the asymptotic growth-rate of the maximum degree of the WRG model, as proved by Lodewijks and Ortgiese in~\cite[Theorem $2.9$, Bounded case]{LodOrt21} (Theorem~\ref{thrm:bddmax} here), is required to prove the growth rate of the location of the maximum degree in the WRG model. It uses a slightly more careful approach compared to the proof of~\cite[Theorem $2.9$, Bounded case]{LodOrt21}, which allows us to determine the range of vertices which obtain the maximum degree.
	
	In recent work by Eslava, the author, and Ortgiese~\cite{EslLodOrt21}, more refined asymptotic behaviour of the maximum degree is presented for the weighted recursive tree model (WRT), that is, the WRG model with $m=1$, under additional assumptions on the vertex-weight distribution. We refine their proofs to allow for an extension of their results and to obtain higher-order results for the location of high-degree vertices. Whether either of these results can be extended to the case $m>1$ is an open problem to date.
	
	Finally, more involved results can be proved for the random recursive tree model. There, the joint convergence of the labels and depths of and the graph distance between high-degree vertices can be obtained, as shown by the author in~\cite[Theorems $2.2$ and $2.4$]{Lod22}. The analysis of the random recursive tree in that article heavily relies on a different construction of the tree compared to the WRG and WRT models, which can be viewed as a construction backward in time. This methodology can be applied to the random recursive tree only, and allows for a simplification of the dependence between degree, depth, and label of a vertex. Whether such results can be extended to the weighted tree case is unclear, but would surely need a different approach.
	
	The paper is organised as follows: In Section~\ref{sec:heur} we provide a short, non-rigorous and intuitive argument as to why the result presented in Theorem~\ref{thrm:bddloc} related to the WRG model holds and briefly discuss the approach to proving the other results stated in Section~\ref{sec:def}. Section~\ref{sec:mainproof} is devoted to proving Theorem~\ref{thrm:bddloc}. In Section~\ref{sec:highproof} we introduce some intermediate results related to the WRT model and use these to prove Theorems~\ref{thrm:conddegloc} and~\ref{thrm:deglocwrt}. We prove the intermediate results in Section~\ref{sec:degwrtproof}. We discuss the additional results (similar to Theorems~\ref{thrm:conddegloc} and~\ref{thrm:deglocwrt}) for the~\ref{ass:beta} and~\ref{ass:gamma} cases in Section~\ref{sec:examples}. Finally, the~\hyperref[sec:appendix]{Appendix} contains several technical lemmas that are used in some of the proofs.
	
	\section{Heuristic ideas behind the main results and preliminary results}\label{sec:heur}
	
	In this section we present some heuristic, non-rigorous ideas that underpin the main results, as presented in Theorems~\ref{thrm:bddloc},~\ref{thrm:conddegloc},~\ref{thrm:deglocwrt} (as well as the results presented in Section~\ref{sec:examples}), and also some preliminary results required throughout the paper.
	
	\subsection{Heuristic ideas}
	
	To understand why the maximum degree of WRG model is attained by vertices with labels of order $n^{\mu_m(1+o(1))}$, where $\mu_m:=1-(\theta_m-1)/(\theta_m\log\theta_m)$, we first state the following observation: for $m\in\N$, define $f_m:(0,1)\to \R_+$ by
	\be \label{eq:f}
	f_m(x):=\frac{1}{\log \theta_m}\Big(\frac{(1-x)\log \theta_m}{\theta_m-1}-1-\log\Big(\frac{(1-x)\log \theta_m}{\theta_m-1}\Big)\Big),\qquad x\in(0,1).
	\ee 
	It is readily checked that $f_m$ has a unique fixed point $x^*_m$ in $(0,1)$, namely $x^*_m=\mu_m$, and that $f_m(x)>x$ for all $x\in(0,1),x\neq \mu_m$. Then, using a Chernoff bound on $\zni$ (a Markov bound on $\exp(t\zni)$ for $t>0$ and determining the value of $t$ that minimises the upper bound) yields
	\be \label{eq:chern}
	\Pf{\zni\geq \log_{\theta_m}\!n}\leq \exp(-\log_{\theta_m}\!n(u_i-1-\log u_i)),
	\ee 
	where 
	\be
	u_i=\frac{mW_i}{\log_{\theta_m}\!n}\sum_{j=i}^{n-1}\frac{1}{S_j}.
	\ee 
	Here we use the quantity $\log_{\theta_m}\! n$, as this (asymptotically) is the size of the maximum degree. Let us now assume that $i\sim n^\beta$ for some $\beta\in(0,1)$. By Lemma~\ref{lemma:logconv}, almost surely
	\be 
	\sum_{j=i}^{n-1}1/S_j =(1+o(1))\log(n/i)/\E W=(1+o(1))(1-\beta)\log (n)/\E W,
	\ee 
	so that
	\be 
	u_i\leq \frac{m(1-\beta)\log \theta_m}{\E W}(1+o(1))=\frac{(1-\beta)\log \theta_m}{\theta_m-1}(1+o(1))<1,
	\ee  
	almost surely, where the final inequality holds for all $n$ sufficiently large as $\log(1+x)\leq x$ for all $x>-1$. Moreover, the $o(1)$ term is independent of $i$. As $x\mapsto x-1-\log x$ is decreasing on $(0,1)$, we can use the almost sure upper bound on $u_i$ in~\eqref{eq:chern}, combined with~\eqref{eq:f}, to obtain
	\be 
	\Pf{\zni\geq \log_{\theta_m}\! n}\leq\exp(-f_m(\beta)\log n(1+o(1)) ).
	\ee 
	Note that this upper bound depends on $i$ only via $i\sim n^\beta$. We perform a union bound over $\{\inn: i\leq n^{\mu_m-\eps}\text{ or }i\geq n^{\mu_m+\eps}\}$. As the sum obtained from the union bound can be well-approximated by an integral, we arrive at 
	\be 
	\mathbb P\Big(\max_{i\in[n]\backslash [n^{\mu_m-\eps}, n^{\mu_m+\eps}]}\zni \geq \log_{\theta_m}\! n\Big)\leq \int_{(0,1)\backslash (\mu_m-\eps,\mu_m+\eps)}\!\!\!\!\!\!\!\!\exp((\beta-f_m(\beta))\log n(1+o(1)))\,\d \beta.
	\ee 
	It follows from the properties of the function $f_m$ (as stated below~\eqref{eq:f}) that this integral converges to zero with $n$.
	
	To obtain the more precise behaviour of the labels of high-degree vertices, as in (among others) Theorem~\ref{thrm:conddegloc}, the precise evaluation of the union bound in the approach sketched above no longer suffices. Instead, for any $k\in\N$, we derive in Proposition~\ref{prop:deglocwrt} the asymptotic expression
		\be \label{eq:probapprox}
		\P{\Zm_n(v_i)\geq d_i, v_i>\ell_i, i\in[k]}\approx \prod_{i=1}^k \mathbb E\bigg[\Big(\frac{W}{\theta-1+W}\Big)^{d_i}\!\Pf{X_i\leq\Big(1+\frac{W}{\theta-1}\Big)\log(n/\ell_i)}\bigg],
		\ee 
		where $v_1, \ldots,v_k$ are $k$ vertices selected uniformly at random from $[n]$ without replacement, $\theta:=\theta_1=1+\E W$, and $X_i\sim\text{Gamma}(d_i+1,1)$ for each $i\in[k]$, under certain assumptions on the $d_i$ and $\ell_i$. Heuristically, this follows from the fact that $S_j\approx j\E W$, and
		\be 
		\Zm_n(j)=\sum_{i=j+1}^n \text{Ber}\Big(\frac{W_j}{S_{i-1}}\Big)\approx \text{Poi}\Big(\sum_{i=j+1}^n \frac{W_j}{i\E W}\Big)\approx \text{Poi}\Big(\frac{W_j}{\theta-1}\log(n/j)\Big).
		\ee 
		By conditioning on the value of $v$, we thus have (with $k=1$ for simplicity and dropping indices) 
		\be
		\P{\Zm_n(v)\geq d, v>\ell}\approx \P{\text{Poi}\Big(\frac{W}{\theta-1}\log(n/v)\Big)\geq d,v>\ell},
		\ee 
		where we can remove the index of the vertex-weight, as the weights are i.i.d.\ and hence it does not influence the probability. We first observe that $v/n\toindis U$, where $U$ is a uniform random variable on $(0,1)$. Second, we have that $T:=\log(1/U)$ is a rate one exponential random variable, independent of everything else. Finally, the duality between Poisson and gamma random variables via Poisson processes yields that we can approximate the right-hand side by
		\be
		\P{X\leq \frac{W}{\theta-1}T, T\leq \log(n/\ell)}=\P{X\leq T_W, T_W\leq \frac{W}{\theta-1}\log(n/\ell)}, 
		\ee 
		where $X\sim\text{Gamma}(d,1)$ and $T_W:=WT/(\theta-1)$. Note that $T_W$ is exponential with rate $(\theta-1)/W$, conditionally on $W$. Setting $x:=W\log(n/\ell)/(\theta-1)$ and conditioning on both $W $ and $X$, we obtain
		\be
		\Pf{X\leq T_W\leq x\,|\, X}=\ind_{\{X\leq x\}}\int_X^x \frac{\theta-1}{W}\e^{-(\theta-1)t/W}\,\d t=\ind_{\{X\leq x\}}\big(\e^{-(\theta-1) X/W}-\e^{-(\theta-1) x/W}\big).
		\ee
		Taking the expected value with respect to $X$ then yields
		\be 
		\Pf{X\leq T_W\leq x}=\Big(1+\frac{\theta-1}{W}\Big)^{-d}\Pf{X'\leq x}-\e^{-(\theta-1)x/W}\P{X\leq x},
		\ee 
		where $X'\sim\text{Gamma}(d,1+(\theta-1)/W)$, conditionally on $W$. As $X\overset d=(1+(\theta-1)/W)X' \sim \text{Gamma}(d,1)$, we obtain by substituting the definition of $x$, 
		\be 
		\Big(\frac{W}{\theta-1+W}\Big)^d\Pf{X\leq \Big(1+\frac{W}{\theta-1}\Big)\log(n/\ell)}-\frac\ell n\Pf{X\leq \frac{W}{\theta-1}\log(n/\ell)}.
		\ee 
		Conditions on $d$ and $\ell$ will allow us to show that the second term is negligible with respect to the first term and hence an error term. Taking the expected value with respect to $W$ then approximately yields~\eqref{eq:probapprox}. This result can then be used to obtain more precise statements regarding the label of high-degree vertices, as well as the size of the maximum degree in the tree.
		
		We finally comment on Condition~\ref{item:c2} and Theorem~\ref{thrm:conddegloc}. For vertex-weight distributions that satisfy this condition, we can show (as in Lemma~\ref{lemma:expasymp} in the~\hyperref[sec:appendix]{Appendix}) that the main contribution to 
		\be 
		\P{\Zm_n(v)\geq d}\approx\E{\Big(\frac{W}{\theta-1+W}\Big)^d}
		\ee 
		comes from values of $W$ close to one, namely at $W=1-Cd^{-\beta}$ for some constant $C>0$ and $\beta>1/2$ (or even closer to one). As such, one would expect this to be the same for the right-hand side of~\eqref{eq:probapprox}. Substituting this value of $W$ roughly yields (again dropping indices and setting $C=1$ for simplicity)
		\be 
		\E{\Big(\frac{W}{\theta-1+W}\Big)^d}\P{X\leq\Big(\frac{\theta}{\theta-1}+\frac{1}{d^\beta(\theta-1)}\Big)\log (n/\ell)}. 
		\ee 
		When we set, for $z\in\R$, 
		\be 
		\ell:=n\exp\big(-(1-\theta^{-1})(\E X - z\sqrt{\Var(X)})\big)\approx n\exp\big(-(1-\theta^{-1})(d-z\sqrt{d})), 
		\ee 
		this simplifies to 
		\be 
		\E{\Big(\frac{W}{\theta-1+W}\Big)^d}\P{\frac{X-\E X}{\sqrt{\Var(X)}}\leq -z+\frac{\E X(1+o(1))}{\theta d^\beta\sqrt{\Var(X)}}}.
		\ee 
		Now, the probability tends to $1-\Phi(z)$ by the central limit theorem when $d$ diverges with $n$, since $\E X=d+1$ and $d^\beta\sqrt{\Var(X)}\sim d^{1/2+\beta}\gg d$ as $\beta>1/2$. This thus shows that $\log v$ is approximately normal and provides the asymptotic mean and variance.  
		
		For tail distributions that decay at a faster rate near one, the main contribution to the expected value is made for $W=1-d^{-\beta}$ with $\beta\leq 1/2$, for which this argument does not hold. Here, we require additional higher-order terms in the rescaling of the labels of high-degree vertices. An example of such a family of distributions is presented in the~\ref{ass:gamma} case of Assumption~\ref{ass:weights}. Theorem~\ref{thrm:gammacond} provides, to some extent, the behaviour of the labels in this case.
		
		\subsection{Preliminaries}
		
		Here we present some known results that are needed throughout the paper. The first result states the almost sure convergence of the maximum degree in the WRG model, as in Theorem $2.9$ in~\cite{LodOrt21}:
		
		\begin{theorem}[Maximum degree in WRGs with bounded random weights, \cite{LodOrt21}]\label{thrm:bddmax}
			Consider the \\WRG model as in Definition~\ref{def:wrg} with almost surely bounded vertex-weights and $m\in\N$. Then, 
			\be 
			\max_{\inn}\frac{\zni}{\log_{\theta_m}\!n}\toas 1.
			\ee 
		\end{theorem}
		
		The following result concerns the asymptotic behaviour of the limiting (tail) degree distribution $p_{\geq d}$ and $p_d$, defined as 
		\be 
		p_{\geq d}:=\E{\Big(\frac{W}{\theta_m-1+W}\Big)^d}, \qquad p_d:=\E{\frac{\theta_m-1}{\theta_m-1+W}\Big(\frac{W}{\theta_m-1+W}\Big)^d},
		\ee 
		of the weighted recursive graph as $d$ from diverges, which combines (parts of) Theorem $2.7$ from~\cite{LodOrt21}, and Lemmas $5.5$, $7.1$ and $7.3$ from~\cite{EslLodOrt21}. For the purposes of this paper, we state the result for the case $m=1$ only. 
		
		\begin{theorem}[Asymptotic behaviour of $p_{\geq d}$,~\cite{EslLodOrt21,LodOrt21}]\label{thrm:pkasymp}
			Consider the WRT with vertex-weights \\$(W_i)_{i\in\N}$, i.i.d.\ copies of a non-negative random variable $W$ which satisfies Condition~\ref{item:c1}. Recall that $\theta:=\theta_1=1+\E W$. Then, for any $\xi>0$ and $d$ sufficiently large, 
			\be 
			(\theta+\xi)^{-d}\leq p_d\leq p_{\geq d}\leq \theta^{-d}.
			\ee 
			Moreover, consider the different cases in Assumption~\ref{ass:weights}: 
			\begin{itemize}
				\item If $W$ satisfies the~\ref{ass:weightatom} case for some $q_0\in(0,1]$, 
				\be 
				p_{\geq d}=q_0\theta^{-d}(1+o(1)).
				\ee 
				\item If $W$ satisfies the~\ref{ass:beta} case for some $\alpha,\beta>0$, 
				\be 
				p_{\geq d}=\frac{\Gamma(\alpha+\beta)}{\Gamma(\alpha)}(1-\theta^{-1})^{-\beta} d^{-\beta}\theta^{-d}\big(1+\cO(1/k)\big).
				\ee 
				\item If $W$ satisfies the~\ref{ass:gamma} case for some $b\in\R,c_1>0$, and $\tau= 1$ such that $bc_1\leq 1$, then 
				\be 
				p_{\geq d}=Cd^{b/2+1/4}\e^{-2\sqrt{c_1^{-1}(1-\theta^{-1})d}}\theta^{-d}\big(1+\cO(1/\sqrt d)\big),
				\ee 
				with $C:=\e^{c_1^{-1}(1-\theta^{-1})/2}\sqrt{\pi}c_1^{-1/4+b/2}(1-\theta^{-1})^{1/4+b/2}.$
			\end{itemize} 
		\end{theorem}
		
		\begin{remark}
			The final results which consider the different cases of Assumption~\ref{ass:weights} also hold for $p_d$ instead of $p_{\geq d}$ when one adds a multiplicative constant $1-\theta^{-1}$ to the right-hand side.
		\end{remark} 
		
		The following proposition provides an asymptotic expression of the tail degree distribution of $k$ typical vertices under certain conditions~\cite[Proposition $5.1$]{EslLodOrt21}: 
		
		\begin{proposition}[Distribution of typical vertex degrees~\cite{EslLodOrt21}]\label{lemma:degprobasymp}
			Consider the WRT model, that is, the WRG as in Definition~\ref{def:wrg} with $m=1$, with vertex-weights $(W_i)_{\inn}$ which are i.i.d.\ copies of a positive random variable $W$ that satisfies Conditions~\ref{item:c1} and~\ref{item:c3} of Assumption~\ref{ass:weights}. Fix $k\in\N,c\in(0,\theta/(\theta-1))$, and let $(v_i)_{i\in[k]}$ be $k$ vertices selected uniformly at random without replacement from $[n]$. Then, uniformly over $d_i\leq c\log n,i\in[k]$,
			\be
			\P{\Zm_n(v_i)\geq d_i,i\in[k]}= \prod_{i=1}^k \E{\Big(\frac{W}{\E{W}+W}\Big)^{d_i}}(1+o(1)).
			\ee
		\end{proposition}
		
		Finally, we have the following three technical lemmas. The first deals with concentration of sums of i.i.d.\ random variables, the second with particular multiple integrals that we use in one of the proofs.
		
		\begin{lemma}[Bounds on partial sums of vertex-weights~\cite{EslLodOrt21}, Lemma $A.2$]\label{lemma:weightsumbounds}
			Let $(W_i)_{i\in\N}$ be i.i.d.\\ copies of a random variable $W$ with mean $\E W\in(0,1]$. Let $\eta\in(0,1), \delta\in(0,1/2)$, $k\in\N$, and set $\zeta_n:=n^{-\delta \eta}/\E{W}$. Consider the events 
			\be\ba  \label{eq:events}
			E_n^{(1)}&:=\bigg\{ \sum_{\ell=1}^j W_\ell \in ((1-\zeta_n)\E{W}j,(1+\zeta_{n})\E{W}j),\text{ for all } n^{\eta}\leq j\leq n\bigg\},\\
			E_n^{(2)}&:=\Big\{\sum_{\ell=k+1 }^j W_\ell\in((1-\zeta_n)\E{W}j,(1+\zeta_n)\E{W}j),\text{ for all } n^{\eps}\leq j\leq n\Big\}.
			\ea \ee 
			Then, for any $\gamma>0$ and any $i\in\{1,2\}$, for all $n$ large,
			\be 
			\P{(E_n^{(i)})^c}\leq n^{-\gamma}.
			\ee 
		\end{lemma} 
		
		\begin{lemma}[\cite{EslLodOrt21}, Lemma $A.4$]\label{lemma:logints}
			For any $k\in\N$ and any $0< a\leq b<\infty$, 
			\be \label{eq:logint}
			\int_a^b \int_{x_1}^b\cdots \int_{x_{k-1}}^b \prod_{j=1}^kx_j^{-1}\,\d x_k\ldots\d x_1= \frac{(\log(b/a))^k}{k!}.
			\ee 
			Similarly, for any $k\in\N$ and any $0< a\leq b-k<\infty$,
			\be \label{eq:loglb}
			\int_{a+1}^b\int_{x_1+1}^b\cdots \int_{x_{k-1}+1}^b\prod_{j=1}^{k}x_{j}^{-1}\,\d x_k\ldots \d x_1\geq \frac{(\log(b/(a+k)))^k}{k!}.
			\ee 
		\end{lemma} 
		
		\begin{lemma}[\cite{LodOrt21}, Lemma $5.1$]\label{lemma:logconv}
			Let $(W_i)_{i\in\N}$ be a sequence of strictly positive i.i.d.\ random variables which are almost surely bounded. Then, there exists an almost surely finite random variable $Y$ such that
			\be 
			\sum_{j=1}^{n-1}\frac{1}{S_j}-\frac{1}{\E W}\log n\toas Y.		
			\ee 		
		\end{lemma}
		
		This lemma implies, in particular, that for any $i=i(n)$ such that $i$ diverges with an and $i=o(n)$ as $n\to\infty$, almost surely,
		\be\ba \label{eq:lemma}
		\sum_{j=i}^{n-1}\frac{1}{S_j}=\frac{1}{\E W}\log(n/i)(1+o(1)),\qquad
		\sum_{j=1}^{n-1}\frac{1}{S_j}=\frac{1}{\E W}\log(n)(1+o(1)).
		\ea \ee 
	
	\section{Location of the maximum degree vertices}\label{sec:mainproof}
	
	Let us, for ease of writing, set $\mu_m:=1-(\theta_m-1)/(\theta_m\log\theta_m)$, where we recall that $\theta_m:=1+\E W/m$. To make the intuitive idea presented in Section~\ref{sec:heur} precise, we use a careful union bound on the events $\{\max_{1\leq i\leq n^{\mu_m-\eps}}\zni\geq (1-\eta)\log_{\theta_m}\!n\}$ and $\{\max_{n^{\mu_m+\eps}\leq i\leq n}\zni\geq (1-\eta)\log_{\theta_m}\! n\}$ for arbitrary and fixed $\eps>0$ and some sufficiently small $\eta>0$.
	
	\begin{proof}[Proof of Theorem~\ref{thrm:bddloc}]
		As in the proofs of \cite[Theorem $2.9$]{LodOrt21} and \cite[Theorem $1$]{DevLu95}, we first prove the convergence holds in probability, and then discuss how to improve it to almost sure convergence.
		
		We take $\eta\in(0,1-\log(\theta_m)/(\theta_m-1))$ and write
		\be\ba \label{eq:locbound}
		\Pf{\Big|\frac{\log I_n}{\log n}-\mu_m\Big|\geq \eps}
		\leq {}&\Pf{\big\{I_n\leq n^{\mu_m-\eps}\}\cap \{\max_{\inn}\zni \geq (1-\eta)\log_{\theta_m}\!n\big\}}\\
		&+\Pf{\{I_n\geq n^{\mu_m+\eps}\big\}\cap \{\max_{\inn}\zni \geq (1-\eta)\log_{\theta_m}\!n\big\}}\\
		&+\Pf{\max_{\inn}\zni < (1-\eta)\log_{\theta_m}\!n}.
		\ea \ee 
		We deal with the first two terms on the right-hand side first and at the end use Theorem~\ref{thrm:bddmax} to deal with the final term. The first two probabilities can be bounded from above by
		\be \label{eq:locub}
		\Pf{ \max_{i\in[n^{\mu_m-\eps}]}\zni \geq (1-\eta)\log_{\theta_m}\!n}+\Pf{\max_{n^{\mu_m+\eps}\leq i\leq n}\zni \geq (1-\eta)\log_{\theta_m}\!n}.
		\ee 
		The aim is thus to show that vertices with a label `far away' from $n^{\mu_m}$ are unlikely to have a high degree. With $I_n^-:=n^{\mu_m-\eps}, I_n^+:=n^{\mu_m+\eps}$, we first apply a union bound to obtain the upper bound
		\be 
		\sum_{i\in[n]\backslash [I_n^-,I_n^+]}\Pf{\zni\geq (1-\eta)\log_{\theta_m}\!n}.
		\ee 
		With the same approach that leads to the upper bound in~\eqref{eq:chern}, that is, using a Chernoff bound with $t=\log((1-\eta)\log_{\theta_m}\!n)-\log\big(mW_i\sum_{j=i}^{n-1}1/S_j\big)$, we arrive at the upper bound
		\be\ba \label{eq:probbound}
		\sum_{\inn\backslash [I_n^-,I_n^+]}\!\!\!\!\!\!\!\e^{-t(1-\eta)\log_{\theta_m}\!\!n}\prod_{j=i}^{n-1}\Big(1+\big(\e^t-1\big)\frac{W_i}{S_j}\Big)^m\leq\sum_{\inn\backslash [I_n^-,I_n^+]}\!\!\!\!\!\!\e^{-(1-\eta)\log_{\theta_m}\!\!n(u_i-1-\log u_i)},
		\ea \ee 
		where
		\be \label{eq:ui}
		u_i:=\frac{mW_i}{(1-\eta)\log_{\theta_m}\!n}\sum_{j=i}^{n-1} \frac{1}{S_j}.
		\ee 
		We now set
		\be \ba
		\delta:=\min\bigg\{{}&\frac{1-\eta}{2\log\theta_m}\bigg(\frac{\log\theta_m}{(\theta_m-1)(1-\eta)}-1-\log\bigg(\frac{\log\theta_m}{(\theta_m-1)(1-\eta)}\bigg)\bigg),\\
		&-\frac{(\theta_m-1)(1-\eta)}{2\log \theta_m}W_0\big(-\theta_m^{-1/(1-\eta)}\e^{-1}\big)\bigg\},
		\ea \ee
		with $W_0$ the (main branch of the) $W$ Lambert function, the inverse of $f:[-1,\infty)\to [-1/\e,\infty)$, $f(x):=x\e^x$. Note that, when $\eps$ is sufficiently small, $\delta\in(0,\min\{\mu_m-\eps,1-\mu_m-\eps\})$. We use this $\delta$ to split the union bound in~\eqref{eq:probbound} into three parts: 
		\be\ba \label{eq:s123}
		R_1&:=\sum_{i=1}^{\lfloor n^{\delta}\rfloor} \e^{-(1-\eta)\log_{\theta_m}\!\!n(u_i-1-\log u_i)},\\
		R_2&:=\sum_{i=\lceil n^{1-\delta}\rceil}^n \e^{-(1-\eta)\log_{\theta_m}\!\!n(u_i-1-\log u_i)},\\
		R_3&:=\sum_{i\in [n^\delta, n^{1-\delta}]\backslash [I_n^-,I_n^+]}\e^{-(1-\eta)\log_{\theta_m}\!\!n(u_i-1-\log u_i)},
		\ea \ee 
		and we aim to show that each of these terms converges to zero with $n$ almost surely. For $R_1$ we use that uniformly in $i\leq n^\delta$, almost surely
		\be \label{eq:uibound1}
		u_i\leq \frac{m}{(1-\eta)\log_{\theta_m}\!n}\sum_{j=1}^{n-1}\frac{1}{S_j}=\frac{\log\theta_m}{(1-\eta)(\theta_m-1)}(1+o(1)),
		\ee 
		where the final step follows from Lemma~\ref{lemma:logconv}. Using that the upper bound is at most $1$ by the choice of $\eta$, that $x\mapsto x-1-\log x$ is decreasing on $(0,1)$ and using this in $R_1$ in~\eqref{eq:s123}, we bound $R_1$ from above by  
		\be \ba \label{eq:s1bound}
		\sum_{i=1}^{\lfloor n^\delta\rfloor}{}&\exp\Big(-\frac{(1-\eta)\log n}{\log\theta_m}\Big( \frac{\log\theta_m}{(1-\eta)(\theta_m-1)}-1-\log\Big(\frac{\log\theta_m}{(1-\eta)(\theta_m-1)}\Big)\Big)(1+o(1))\Big)\\
		={}&\exp\Big(\log n\Big(\delta-\frac{1-\eta}{\log\theta_m}\Big( \frac{\log\theta_m}{(1-\eta)(\theta_m-1)}-1-\log\Big(\frac{\log\theta_m}{(1-\eta)(\theta_m-1)}\Big)\Big)(1+o(1))\Big),
		\ea \ee 
		which converges to zero by the choice of $\delta$. In a similar way, uniformly in $n^{1-\delta}\leq i\leq n$, almost surely
		\be \label{eq:uibound2}
		u_i\leq \frac{m}{(1-\eta)\log_{\theta_m}\!n}\sum_{j=\lceil n^{1-\delta}\rceil }^{n-1}\frac{1}{S_j}=\frac{\delta\log\theta_m}{(1-\eta)(\theta_m-1)}(1+o(1)),
		\ee 
		so that we can bound $R_2$ from above by
		\be\ba \label{eq:s2bound}
		\sum_{i= \lceil n^{1-\delta}\rceil}^n\!\!\!\!\!\!{}&\exp\Big(-(1-\eta)\log_{\theta_m}\! n\Big( \frac{\delta\log\theta_m}{(1-\eta)(\theta_m-1)}-1-\log\Big(\frac{\delta\log\theta_m}{(1-\eta)(\theta_m-1)}\Big)\Big)(1+o(1))\Big)\\
		={}&\exp\Big(\log n\Big(1-\frac{1-\eta}{\log\theta_m}\Big( \frac{\delta\log\theta_m}{(1-\eta)(\theta_m-1)}-1-\log\Big(\frac{\delta\log\theta_m}{(1-\eta)(\theta_m-1)}\Big)\Big)(1+o(1))\Big).
		\ea\ee 
		Again, by the choice of $\delta$, the exponent is strictly negative, so that the upper bound converges to zero with $n$. It remains to bound $R_3$. We aim to approximate the sum by an integral, using the same approach as in the proof of~\cite[Theorem $2.9$, Bounded case]{LodOrt21}. We first bound $u_i\leq m(H_n-H_i)/((1-\eta)\log_{\theta_m}\! n)=:\wt u_i$ almost surely for any $\inn$, where $H_n:=\sum_{j=1}^{n-1}1/S_j$. Then, define $u:(0,\infty)\to \R$ and $\phi:(0,\infty)\to\R$ by 
		\be 
		u(x):=\Big(1-\frac{\log x}{\log n}\Big)\frac{\log\theta_m}{(1-\eta)(\theta_m-1)},\quad \text{and}\quad \phi(x):=x-1-\log x, \qquad x>0. 
		\ee 
		For $i$ in $[n^\delta, n^{1-\delta}]\backslash [I_n^-,I_n^+]$ such that $i=n^{\beta+o(1)}$ for some $\beta\in[\delta,1-\delta]$ (where the $o(1)$ is independent of $\beta$) and $x\in[i,i+1)$, 
		\be\ba\label{eq:phidiff}
		|\phi(\wt u_i)-\phi(u(x))|\leq{}& |\wt u_i-u(x)|+|\log(\wt u_i/u(x))|\\
		={}&\bigg|\frac{\log\theta_m}{(1-\eta)(\theta_m-1)}\Big(1-\frac{\log x}{\log n}\Big)-\frac{\log\theta_m}{(1-\eta)(\theta_m-1)\log n}\sum_{j=i}^{n-1}\frac{1}{S_j}\bigg|\\
		&+\bigg|\log\bigg(\frac{\E W}{\log n-\log x}\sum_{j=i}^{n-1}\frac{1}{S_j}\bigg)\bigg|.
		\ea\ee 
		By~\eqref{eq:lemma} and since $i$ diverges with $n$, $\sum_{j=i}^{n-1}1/S_j-\log(n/i)/\E W=o(1)$ almost surely as $n\to\infty$. Applying this to the right-hand side of~\eqref{eq:phidiff} yields
		\be \label{eq:phidiff2}
		|\phi(\wt u_i)-\phi(u(x))|\leq \frac{\log\theta_m}{(1-\eta)(\theta_m-1)}\Big|\frac{\log x-\log i}{\log n}\Big|+\Big|\log\Big(1+\frac{\log x-\log i+o(1)}{\log n-\log x}\Big)\Big|.
		\ee 
		Since $x\geq i\geq n^\delta$ and $|x-i|\leq 1$, we thus obtain that, uniformly in $[n^\delta, n^{1-\delta}]\backslash [I_n^-,I_n^+]$ and $x\in[i,i+1)$, we have $|\phi(\wt u_i)-\phi(u(x))|=o(1/(n^\eps\log n))$ almost surely as $n\to\infty$. Applying this to $R_3$ in~\eqref{eq:s123} yields the upper bound
		\be\ba \label{eq:intbound}
		\sum_{i\in [n^\delta, n^{1-\delta}]\backslash [I_n^-,I_n^+]}\!\!\!\!\!\!\!\!\!\!{}&\e^{-(1-\eta)\phi(\wt u_i)\log_{\theta_m}\! n}\\
		\leq{}& \sum_{i\in [n^\delta, n^{1-\delta}]\backslash [I_n^-,I_n^+]}\int_i^{i+1}\e^{-(1-\eta)\log_{\theta_m}\! n( \phi(u(x))+|\phi(\wt u_i)-\phi(u(x))|)}\,\d x\\
		\leq {}&(1+o(1))\int_{[n^\delta,n^{1-\delta}]\backslash [I_n^-I_n^+]}\e^{-(1-\eta)\phi(u(x))\log_{\theta_m}\! n}\,\d x.
		\ea \ee  
		Using the variable transformation $w=\log x/\log n$ and setting $U:=[\delta,1-\delta]\backslash[\mu_m-\eps,\mu_m+\eps]$ yields 
		\be \ba\label{eq:uexp}
		(1{}&+o(1))\int_{U}\exp\Big(-\log n\frac{1-\eta}{\log \theta_m}\phi\Big(\frac{(1-w)\log \theta_m}{(1-\eta)(\theta_m-1)}\Big)\Big)n^w\log n \, \d w\\
		&=(1+o(1))\int_{U}\exp\Big(-\log n\Big(\frac{1-\eta}{\log \theta_m}\phi\Big(\frac{(1-w)\log \theta_m}{(1-\eta)(\theta_m-1)}\Big)-w\Big)+\log\log n\Big) \, \d w.
		\ea\ee 
		We now observe that the mapping
		\be 
		w\mapsto \frac{1-\eta}{\log \theta_m}\phi\Big(\frac{(1-w)\log\theta_m}{(1-\eta)(\theta_m-1)}\Big)
		\ee 
		has two fixed points, namely
		\be \ba\label{eq:wfix}
		w^{(1)}&:=1+\frac{(1-\eta)(\theta_m-1)}{\theta_m\log\theta_m}W_0\big(-\theta_m^{-\eta/(1-\eta)}\e^{-1}\big), \\ w^{(2)}&:=1+\frac{(1-\eta)(\theta_m-1)}{\theta_m\log\theta_m}W_{-1}\big(-\theta_m^{-\eta/(1-\eta)}\e^{-1}\big),
		\ea \ee 
		where we recall that $W_0$ is the inverse of $f:[-1,\infty)\to [-1/\e,\infty)$, $f(x)=x\e^x$, also known as the main branch of the Lambert $W$ function, and where $W_{-1}$ is the inverse of $g:(-\infty,-1]\to (-\infty,-1/\e]$, $g(x)=x\e^x$, also known as the negative branch of the Lambert $W$ function. Moreover, the following inequalities hold as well:
		\be \ba\label{eq:wineq}
		w&<\frac{1-\eta}{\log \theta_m}\phi\Big(\frac{(1-w)\log\theta_m}{(1-\eta)(\theta_m-1)}\Big), \qquad w\in(0,w^{(2)}),\ w\in(w^{(1)},1), \\
		w&>\frac{1-\eta}{\log \theta_m}\phi\Big(\frac{(1-w)\log\theta_m}{(1-\eta)(\theta_m-1)}\Big), \qquad w\in(w^{(2)},w^{(1)}),
		\ea \ee
		and we claim that the following statements hold:
		\be\label{eq:wclaims}
		\forall\  \eta>0\text{ sufficiently small},\ w^{(2)}<\mu_m<w^{(1)},\quad\text{ and }\quad\;  \lim_{\eta\downarrow 0}w^{(1)}=\lim_{\eta\downarrow 0}w^{(2)}=\mu_m.
		\ee
		We defer the proof of these inequalities and claims to the end. For now, let us use these properties and set $\eta$ sufficiently small so that $\mu_m-\eps<w^{(2)}<\mu_m<w^{(1)}<\mu_m+\eps$, so that $U\subset [\delta,w^{(2)})\cup (w^{(1)},1-\delta]$. If we define 
		\be \label{eq:wu}
		\phi'_U:=\inf_{w\in U}\Big[\frac{1-\eta}{\log \theta_m}\phi\Big(\frac{(1-w)\log \theta_m}{(1-\eta)(\theta_m-1)}\Big)-w\Big],
		\ee
		then it follows from the choice of $\eta$, from~\eqref{eq:wineq} and the definition of $U$ that $\phi'_U>0$, so that the integral in~\eqref{eq:uexp} can be bounded from above by  
		\be \label{eq:finbound}
		(1+o(1))\exp\Big(-\phi'_U\log n+\log\log n\Big),
		\ee 
		which converges to zero with $n$. We have thus established that $R_1, R_2, R_3$ converge to zero almost surely as $n$ tends to infinity. Combined, this yields that the upper bound in~\eqref{eq:probbound} converges to zero almost surely, so that together with~\eqref{eq:locub} we thus find that 
		\be\ba  \label{eq:condconv}
		\mathbb P_W{}&\Big(\big\{I_n\leq n^{\mu_m-\eps}\}\cap \{\max_{\inn}\zni \geq (1-\eta)\log_{\theta_m}\!n\big\}\Big)\\
		&+\Pf{\{I_n\geq n^{\mu_m+\eps}\big\}\cap \{\max_{\inn}\zni \geq (1-\eta)\log_{\theta_m}\!n\big\}}\toas 0, 
		\ea \ee 
		We now return to~\eqref{eq:locbound}. Taking the mean yields
		\be \ba 
		\limsup_{n\to\infty}\P{\Big|\frac{\log I_n}{\log n}-\mu_m\Big|\geq \eps}\leq{}& \limsup_{n\to\infty}\mathbb E\bigg[\Pf{\big\{I_n\leq n^{\mu_m-\eps}\}\cap \{\max_{\inn}\zni \geq (1-\eta)\log_{\theta_m}\!n\big\}}\\
		{}&+\Pf{\{I_n\geq n^{\mu_m+\eps}\big\}\cap \{\max_{\inn}\zni \geq (1-\eta)\log_{\theta_m}\!n\big\}}\bigg]\\
		&+\limsup_{n\to\infty}\P{\max_{\inn}\zni<(1-\eta)\log_{\theta_m}\! n}.
		\ea\ee 
		Using the uniform integrability of the conditional probability (this is clearly the case as the conditional probability is bounded from above by one) combined with~\eqref{eq:condconv} implies that the first limsup on the right-hand side equals zero. The second limsup also equals zero by Theorem~\ref{thrm:bddmax}. Since $\eps>0$ is arbitrary, this proves that $\log I_n/\log n\toinp \mu_m$.
		
		Now that we have obtained the convergence in probability of $\log I_n/\log n$ to $\mu_m$, we strengthen it to almost sure convergence. We obtain this by constructing the following inequalities: First, for any $\eps\in(0,\mu_m)$, using the monotonicity of $\max_{i\in[n^{\mu_m-\eps}]}\zni$ and $\log_{\theta_m}\!n$, 
		\be \ba
		\sup_{2^N\leq n}\!\frac{\max_{i\in[ n^{\mu_m-\eps}]}\zni}{\log_{\theta_m}\!n}&=\sup_{k\in\N}\sup_{2^{N+(k-1)}\leq n<2^{N+k}}\!\!\!\frac{\max_{i\in[ n^{\mu_m-\eps}]}\zni}{\log_{\theta_m}\!n}\\
		&\leq \sup_{N\leq n}\frac{\max_{i\in[2^{(n+1)(\mu_m-\eps)}]}\Zm_{2^{n+1}}(i)}{n\log_{\theta_m}\!2}.
		\ea \ee 
		With only a minor modification, we can obtain a similar result for $\max_{n^{\mu_m+\eps}\leq i\leq n}\zni$, where now $\eps\in(0,1-\mu_m)$. Here, we can no longer use that this maximum is monotone. Rather, we write
		\be \ba
		\sup_{2^N\leq n}\!\frac{\max_{ n^{\mu_m+\eps}\leq i\leq n}\zni}{\log_{\theta_m}\!n}&=\sup_{k\in\N}\sup_{2^{N+(k-1)}\leq n<2^{N+k}}\!\!\!\frac{\max_{n^{\mu_m+\eps}\leq i\leq n}\zni}{\log_{\theta_m}\!n}\\
		&\leq \sup_{k\in\N}\frac{\max_{2^{(N+(k-1))(\mu_m+\eps)}\leq i\leq 2^{N+k}}\Zm_{2^{N+k}}(i)}{(N+(k-1))\log_{\theta_m}\!2}\\
		&=\sup_{N\leq n}\frac{\max_{2^{n(\mu_m+\eps)}\leq i\leq 2^{n+1}}\Zm_{2^{n+1}}(i)}{n\log_{\theta_m}\!2}.
		\ea \ee 
		It thus follows that, for any $\eta>0$,
		\be \label{eq:normbounds}
		\limsup_{n\to\infty} \frac{\max_{i\in[n^{\mu_m-\eps}]}\zni}{(1-\eta)\log_{\theta_m}\!n}\leq 1,\quad\limsup_{n\to\infty}\frac{\max_{n^{\mu_m+\eps}\leq i\leq n}\zni}{(1-\eta)\log_{\theta_m}\!n}\leq 1, \quad \mathbb P_W\text{-a.s.,}
		\ee 
		are implied by 
		\be\ba \label{eq:expbounds}
		\limsup_{n\to\infty}\frac{\max_{i\in[2^{(n+1)(\mu_m-\eps)}]}\Zm_{2^{n+1}}(i)}{(1-\eta) n\log_{\theta_m}\!2}&\leq 1,\quad &\mathbb P_W-\text{a.s.},&\\
		\limsup_{n\to\infty} \frac{\max_{2^{n(\mu_m+\eps)}\leq i\leq 2^{n+1}}\Zm_{2^{n+1}}(i)}{(1-\eta) n\log_{\theta_m}\!2}&\leq 1,\quad &\mathbb P_W-\text{a.s.},&
		\ea \ee 
		respectively. We start by proving the first inequality in~\eqref{eq:expbounds}. Define 
		\be \ba
		\mathcal E_n^1&:=\big\{\max_{i\in[2^{(n+1)(\mu_m-\eps)}]}\Zm_{2^{n+1}}(i)> (1-\eta)n\log_{\theta_m}\!2\big\}, \\ \mathcal E_n^2&:=\big\{\max_{2^{n(\mu_m+\eps)}\leq i\leq 2^{n+1}}\Zm_{2^{n+1}}(i)>(1-\eta)n\log_{\theta_m}\! 2\big\}.
		\ea\ee 
		Let us abuse notation to write $I_n^-=2^{(n+1)(\mu_m-\eps)}, I_n^+=2^{n(\mu_m+\eps)}$. By a union bound, we again find
		\be \ba\label{eq:expsplit}
		\P{\mathcal E_n^1\cup \mathcal E_n^2}\leq{}& \sum_{i=1}^{\lfloor 2^{(n+1)\delta}\rfloor }\P{\Zm_{2^{n+1}}(i)>(1-\eta)n\log_{\theta_m}\!2}\\
		&+\sum_{i=\lceil 2^{(n+1)(1-\delta)}\rceil }^{2^{n+1}}\P{\Zm_{2^{n+1}}(i)>(1-\eta)n\log_{\theta_m}\!2}\\
		&+\sum_{i\in[ 2^{(n+1)\delta},2^{(n+1)(1-\delta)}]\backslash [I_n^-,I_n^+]}\P{\Zm_{2^{n+1}}(i)>(1-\eta)n \log_{\theta_m}\! 2},
		\ea \ee 
		and these tree sums are the equivalence of $R_1,R_2,R_3$ in~\eqref{eq:s123}. We again take $\eta$ small enough so that $\mu_m-\eps<w^{(2)}<\mu_m<w^{(1)}<\mu_m+\eps$, where we recall $w^{(1)}, w^{(2)}$ from~\eqref{eq:wfix}. With the same steps as in~\eqref{eq:probbound}, \eqref{eq:uibound1}, and~\eqref{eq:s1bound}, we obtain that we can almost surely bound the first sum on the right-hand side from above by 
		\be\ba
		\sum_{i=1}^{\lfloor 2^{(n+1)\delta}\rfloor}{}&\exp\Big(-\frac{(1-\eta)n\log 2}{\log\theta_m}\Big( \frac{\log\theta_m}{(1-\eta)(\theta_m-1)}-1-\log\Big(\frac{\log\theta_m}{(1-\eta)(\theta_m-1)}\Big)\Big)(1+o(1))\Big)\\
		={}&\exp\Big(n\log 2\Big(\delta-\frac{1-\eta}{\log\theta_m}\Big( \frac{\log\theta_m}{(1-\eta)(\theta_m-1)}-1-\log\Big(\frac{\log\theta_m}{(1-\eta)(\theta_m-1)}\Big)\Big)(1+o(1))\Big),
		\ea \ee 
		which is summable by the choice of $\delta$. Similarly, using the same steps as in~\eqref{eq:uibound2} and~\eqref{eq:s2bound}, we can almost surely bound the second sum on the right-hand side of~\eqref{eq:expsplit} from above by 
		\be\ba 
		\sum_{i= \lceil 2^{(n+1)(1-\delta)}\rceil}^{2^n}\!\!\!\!\!\!{}&\exp\Big(-\frac{(1-\eta)n\log 2}{\log \theta_m}\Big( \frac{\delta\log\theta_m}{(1-\eta)(\theta_m-1)}-1-\log\Big(\frac{\delta\log\theta_m}{(1-\eta)(\theta_m-1)}\Big)\Big)(1+o(1))\Big)\\
		={}&\exp\Big(n\log 2\Big(1-\frac{1-\eta}{\log\theta_m}\Big( \frac{\delta\log\theta_m}{(1-\eta)(\theta_m-1)}-1-\log\Big(\frac{\delta\log\theta_m}{(1-\eta)(\theta_m-1)}\Big)\Big)(1+o(1))\Big),
		\ea\ee 
		which again is summable by the choice of $\delta$. Finally, the last sum on the right-hand side of~\eqref{eq:expsplit} can be approximated by an integral, as in~\eqref{eq:intbound}. By the choice of $\eta$, we can then use the same steps as in~\eqref{eq:uexp} through~\eqref{eq:finbound} to obtain the almost sure upper bound
		\be 
		(1+o(1))\exp\Big(-n\phi'_U\log 2(1+o(1))+\log n+\mathcal O(1)\Big),
		\ee 
		which again is summable. As a result, $\mathbb P_W$-almost surely $\mathcal E_n^1\cup \mathcal E_n^2$ occurs only finitely often by the Borel-Cantelli lemma. This implies that both bounds in~\eqref{eq:expbounds} hold, which imply the bounds in~\eqref{eq:normbounds}.  Defining the events 
		\be \ba
		\mathcal C^1_n&:=\{|\log I_n/\log n-\mu_m|\geq \eps\}, \quad \mathcal C^2_n:=\big\{I_n\leq n^{\mu_m-\eps}\big\}, \quad \mathcal C^3_n:=\big\{I_n\geq n^{\mu_m+\eps}\big\},\\ \mathcal C^4_n&:=\big\{\max_{\inn}\zni> (1-\eta)\log_{\theta_m}\!n\big\},
		\ea\ee 
		we can use the same approach as in~\eqref{eq:locbound} to bound
		\be 
		\sum_{n=1}^\infty \ind_{\mathcal C^1_n}\leq \sum_{n=1}^\infty \ind_{\mathcal C^2_n\cap \mathcal C^4_n}+\ind_{\mathcal C^3_n\cap\mathcal  C^4_n}+\ind_{(\mathcal C^4_n)^c}.
		\ee 
		By the proof of Theorem~\ref{thrm:bddmax} in~\cite{LodOrt21},  $(\mathcal C^4_n)^c$ occurs for finitely many $n$ $\mathbb P_W$-almost surely (not just $\mathbb P$-almost surely as follows directly from Theorem~\ref{thrm:bddmax}). The bounds in~\eqref{eq:normbounds} imply that $\mathbb P_W$-almost surely the events $\mathcal C^2_n\cap\mathcal  C^4_n$ and $\mathcal C^3_n\cap \mathcal C^4_n$ occur for only finitely many $n$, via a similar reasoning as in~\eqref{eq:locub}. Combined, we obtain that $\mathcal C^1_n$ occurs only finitely many times $\mathbb P_W$-almost surely. As a final step we write
		\be \ba
		\mathbb P({}&\forall\, \eps>0\ \exists\, N\in \N: \forall\, n\geq N\ |\log I_n/\log n-\mu_m|<\eps)\\
		&=\E{\Pf{\forall\, \eps>0\, \exists\, N\in \N: \forall\, n\geq N |\log I_n/\log n-\mu_m|<\eps}}=1,
		\ea \ee 
		so that $\log I_n/\log n\xrightarrow{\mathbb P\text{-a.s.}} \mu_m$.
		
		It remains to prove the inequalities in~\eqref{eq:wineq} and the claims in~\eqref{eq:wclaims}. Let us start with the inequalities in~\eqref{eq:wineq}.  We compute 
		\be \label{eq:der}
		\frac{\d }{\d w}\Big(w-\frac{1-\eta}{\log\theta_m}\phi\Big(\frac{(1-w)\log\theta_m}{(1-\eta)(\theta_m-1)}\Big)\Big)=1+\frac{1}{\theta_m-1}-\frac{1-\eta}{\log \theta_m}\frac{1}{1-w},
		\ee  
		which equals zero when $w=w^*:=1-(1-\eta)(\theta_m-1)/(\theta_m\log\theta_m)$, is positive when $w\in(0,w^*)$ and is negative when $w\in(w^*,1)$. Moreover, as $W_0(x)\geq -1$ for all $x\in[-1/\e,\infty)$ and $ W_{-1}(x)\leq -1$ for all $x\in[-1/\e,0)$, it follows from the definition of $w^{(1)}$ and $w^{(2)}$ in~\eqref{eq:wfix} that $w^{(2)}<w^*<w^{(1)}$ for any choice of $\eta>0$. This implies both inequalities in~\eqref{eq:wineq}.
		
		We now prove the claims in~\eqref{eq:wclaims}. Again using that $W_0(x)\geq -1$ for all $x\in[-1/\e,\infty)$ directly yields $w^{(1)}>\mu_m$. The inequality $w^{(2)}<\mu_m$ is implied by    
		\be 
		W_{-1}\big(-\theta_m^{-\eta/(1-\eta)}\e^{-1}\big)<-\frac{1}{1-\eta},
		\ee
		or, equivalently, 
		\be 
		-\theta_m^{-\eta/(1-\eta)}\e^{-1}>-\frac{1}{1-\eta}\e^{-1/(1-\eta)}.
		\ee 
		Setting $\beta:=1/(1-\eta)$ yields
		\be 
		\frac{\theta_m}{\e}<\beta \Big(\frac{\theta_m}{\e}\Big)^{\beta}.
		\ee 
		This inequality is then satisfied when $\beta\in(1,W_{-1}(\log(\theta_m/\e)\theta_m/\e)/\log(\theta_m/\e))$, or, equivalently, when $\eta\in(0,1-\log(\theta_m/\e)/W_{-1}(\log(\theta_m/\e)\theta_m/\e))$, as required. 
		By the definition of $w^{(1)}$ and $w^{(2)}$ in~\eqref{eq:wfix} and since $\mu_m:=1-(\theta_m-1)/(\theta_m\log \theta_m)$, the second claim in~\eqref{eq:wclaims} directly follows from the continuity of $W_0$ and $W_{-1}$ and since $W_0(-1/\e)=W_{-1}(-1/\e)=-1$, which concludes the proof.
	\end{proof}
	
	\section{Higher-order behaviour of the location of high-degree vertices}\label{sec:highproof}
	
	In this section we provide a more detailed insight into the behaviour of the degree and location of high-degree vertices when considering the Weighted Recursive Tree (WRT) model; the WRG model with out-degree $m=1$. Under additional assumptions on the vertex-weights, as in Assumption~\ref{ass:weights}, we are able to extend the result of Theorem~\ref{thrm:bddloc} to higher-order results for the location as well as to \emph{all} high-degree vertices, rather than just the maximum-degree vertices.
	
	The approach taken here is an improvement and extension of the methodology used by Eslava, the author and Ortgiese in~\cite{EslLodOrt21}. In that paper, we study the maximum degree of the WRT model with bounded vertex-weights, and we improve and extend those results in this section.
	
	The approach used in~\cite{EslLodOrt21} is to obtain a precise asymptotic estimate for the probability that $k$ vertices $v_1,\ldots, v_k$, selected uniformly at random without replacement from $[n]$, have degrees at least  $d_1, \ldots, d_k$, respectively, for any $k\in\N$. One of the difficulties in proving this estimate is to show that the probability of this event, conditionally on $\CE_n:=\cup_{i=1}^k \{v_i\leq n^{\eta}\}$ for some arbitrarily small $\eta>0$, is sufficiently small. On $\CE_n$ it is harder to control sums of the first $v_i$ many vertex-weights, as one cannot apply the law of large numbers easily, as opposed to when conditioning on $\CE_n^c$. This is eventually overcome by assuming that the vertex-weight distribution satisfies Condition~\ref{item:c3}, which limits the range of vertex-weight distributions for which the results discussed in~\cite{EslLodOrt21} hold.
	
	Here, we compute an asymptotic estimate for the probability that the degree of $v_i$ is at least $d_i$ \emph{and} that $v_i$ is at least $\ell_i$ for all $i\in[k]$, where the $(\ell_i)_{i\in[k]}$ satisfy $\ell_i\geq n^\eta$ for all $i\in[k]$ and some $\eta\in(0,1)$. The two main advantages of considering this event are that the issues described in the previous paragraph are circumvented, and that for a correct parametrisation of the $\ell_i$ we obtain some precise results on the location of high-degree vertices.
	
	\subsection{Convergence of marked point processes via finite dimensional distributions}\label{sec:pppfdd}
	
	$\,$\\ We first discuss some theoretical preparations to prove Theorem~\ref{thrm:deglocwrt}, after which we state the required intermediate results that we use in the proofs of Theorems~\ref{thrm:conddegloc} and~\ref{thrm:deglocwrt}. Recall the following notation: $d_n^i$ and $\ell_n^i$ denote the degree  and label of the vertex with the $i^{\text{th}}$ largest degree, respectively, $\inn$, where ties are split uniformly at random, and let us write $\theta=\theta_1:=1+\E W, \mu=\mu_1:=1-(\theta-1)/(\theta\log \theta)$ and define $\sigma^2:=1-(\theta-1)^2/(\theta^2\log \theta)$. To prove Theorem~\ref{thrm:deglocwrt} we view the tuples
	\be 
	\Big(d_n^i-\lfloor \log_\theta n\rfloor, \frac{\log \ell_n^i-\mu\log n}{\sqrt{(1-\sigma^2)\log n}}\Big)_{ \inn} 
	\ee
	as a marked point process, where the rescaled degrees form the points and the rescaled labels form the marks of the points. Let $\Z^*:=\Z \cup \{\infty\}$ and endow $\Z^*$ with the metric $d(i,j)=|2^{-i}-2^{-j}|$ and $d(i,\infty)=2^{-i}$ for $ i,j\in\Z$. We work with $\Z^*$ rather than $\Z$, as sets $[i,\infty]$ for $i\in\Z$ are now compact, which provides an advantage later on. Let $\CP$ be a Poisson point process on $\R$ with intensity $\lambda( x):=q_0\theta^{-x}\log \theta\,\d x$ and let $(\xi_x)_{x\in\CP}$ be independent standard normal random variables. For $\eps\in[0,1]$, we define the ground process $\CP^\eps$ on $\Z^*$ and the marked processes $\CM\CP^\eps$ on $\Z^*\times \R$ by 
	\be\label{eq:limppp}
	\CP^\eps:=\sum_{x\in\CP}\delta_{\lfloor x+\eps\rfloor}, \qquad \CM\CP^\eps:=\sum_{x\in \CP}\delta_{(\lfloor x+\eps\rfloor,  \xi_x)},
	\ee  
	where $\delta$ is a Dirac measure. Similarly, we can define
	\be
	\CP^{(n)}:=\sum_{i=1}^n \delta_{\zni-\lfloor\log_\theta n\rfloor}, \qquad \CM\CP^{(n)}:=\sum_{i=1}^n \delta_{(\zni-\lfloor\log_\theta  n\rfloor, (\log i-\mu\log n)/\sqrt{(1-\sigma^2)\log n})}.
	\ee 
	We then let $\CM_{\Z^*}^\#$ and $ \CM_{\Z^*\times\R}^\#$ be the spaces of boundedly finite measures on $\Z^*$ and $ \Z^*\times \R$ (which, in this case, corresponds to locally finite  measures) equipped with the vague topology, respectively. We observe that $\CP^{(n)}$ and $\CP^\eps$ are random elements of $\CM_{\Z^*}^\#$, and $\CM\CP^{(n)}$ and $\CM\CP^\eps$ are random elements of $ \CM_{\Z^*\times \R}^\#$, respectively. Theorem~\ref{thrm:deglocwrt} is then equivalent to the weak convergence of $\CM\CP^{(n_j)}$ to $\CM\CP^\eps$ in $\CM_{\Z^*\times \R}^\#$ along suitable subsequences $(n_j)_{j\in\N}$, as we can order the points in the definition of $\CM\CP^{(n)}$ (and $\CM\CP^\eps$) in decreasing order of their degrees (of the points $x\in \CP$). We remark that the weak convergence of $\CP^{(n_j)}$ to $\CP^{\eps}$ in $\CM_{\Z^*}^\#$ along subsequences when the vertex-weights of the WRT belong to the~\ref{ass:weightatom} case has been established by Eslava, the author, and Ortgiese in~\cite{EslLodOrt21} (and for the particular case of the random recursive tree by Addario-Berry and Eslava in~\cite{AddEsl18}). We extend these results, among others, to the tuple of degree and label.
	
	The approach we shall use to prove the weak convergence of $\CM\CP^{(n_j)}$ is to show that its finite dimensional distributions (FDDs) converge along subsequences. The FDDs of a random measure $\CP$ are defined as the joint distributions, for all finite families of bounded Borel sets $(B_1, \ldots, B_k)$, of the random variables $(\CP(B_1), \ldots, \CP(B_k))$, see~\cite[Definition $9.2.$II]{DalVer08}. Moreover, by~\cite[Proposition $9.2.$III]{DalVer08}, the distribution of a random measure $\CP$ on $\CX$ is completely determined by the FDDs for all finite families $(B_1, \ldots, B_k)$ of \emph{disjoint} sets from a semiring $\CA$ that generates $\CB(\CX)$. In our case, we consider the marked point process $\CM\CP^{(n)}$ on $\CX:=\Z^*\times \R$, see~\eqref{eq:limppp}. Hence, we let
	\be \label{eq:Awrt}
	\CA:=\{\{j\}\times (a,b]: j\in \Z, a,b\in \R\}\cup \{[j,\infty]\times (a,b]: j\in \Z, a,b\in \R\}
	\ee
	be the semiring that generates $\CB(\Z^*\times \R)$. The choice of the metric on $\Z^*$ is convenient, since now weak convergence in $\Z^*\times \R$ is equivalent to the convergence of the finite dimensional distributions by~\cite[Theorem $11.1.$VII]{DalVer08}. So, the weak convergence of the measure $\CM\CP^{(n_j)}$ to $\CM\CP^\eps$ in $\CM^\#_{\Z^*\times \R}$ is equivalent to the convergence of the FDDs of $\CM\CP^{(n_j)}$ to the FDDs of $\CM\CP^\eps$. It thus suffices to prove the joint convergence of the counting measures of finite collections of disjoint subsets of $\CA$. In particular, the weak convergence of $\CM\CP^{n_j}$ implies the distributional convergence of $X_{\geq j}^{(n_j)}(B)=\CM\CP^{(n_\ell)}([j,\infty))$ for any $\{j\}\times B\in \cA$.
	
	Recall the Poisson point process $\CP$ used in the definition of $\CP^\eps$ in~\eqref{eq:limppp} and enumerate its points in decreasing order. That is, $P_i$ denotes the $i^{\text{th}}$ largest point of $\CP$ (ties broken uniformly at random). We observe that this is well-defined, since $\CP([x,\infty))<\infty$ almost surely for any $x\in \R$. Let $(M_i)_{i\in\N}$ be a sequence of i.i.d.\ standard normal random variables. For $\{j\}\times B\in \CA$, we then define 
	\be \ba \label{eq:xnwrt}
	X^{(n)}_j(B)&:=\Big|\Big\{\inn: \zni=\lfloor \log_\theta n\rfloor +j,  \frac{\log i-(\log n-(1-\theta^{-1})(\lfloor \log_\theta n\rfloor +j))}{\sqrt{(1-\theta^{-1})^2(\lfloor \log_\theta n\rfloor+j)}}\in B\Big\}\Big|,\\
	X^{(n)}_{\geq j}(B)&:=\Big|\Big\{\inn: \zni\geq\lfloor \log_\theta n\rfloor +j,  \frac{\log i-(\log n-(1-\theta^{-1})(\lfloor \log_\theta n\rfloor +j))}{\sqrt{(1-\theta^{-1})^2(\lfloor \log_\theta n\rfloor+j)}}\in B\Big\}\Big|,\\
	\wt X^{(n)}_j(B)&:=\Big|\Big\{\inn: \zni=\lfloor \log_\theta n\rfloor +j,  \frac{\log i-\mu\log n}{\sqrt{(1-\sigma^2)\log n}}\in B\Big\}\Big|,\\
	\wt X^{(n)}_{\geq j}(B)&:=\Big|\Big\{\inn: \zni\geq\lfloor \log_\theta n\rfloor +j,  \frac{\log i-\mu\log n}{\sqrt{(1-\sigma^2)\log n}}\in B\Big\}\Big|,\\
	X_j(B)&:=\Big|\Big\{i\in\N: \lfloor P_i+\eps\rfloor = j,  M_i\in B\Big\}\Big|,\\
	X_{\geq j}(B)&:=\Big|\Big\{i\in\N: \lfloor P_i+\eps\rfloor \geq j, M_i\in B\Big\}\Big|.
	\ea \ee 
	Using these random variables is justified, as $\wt X_j^{(n)}(B)=\CM\CP^{(n)}(\{j\}\times B)$, $\wt X_{\geq j}^{(n)}(B)=\CM\CP^{(n)}([j,\infty]\times B)$, and $X_j(B)=\CM\CP^\eps(\{j\}\times B)$ and $X_{\geq j}(B)=\CM\CP^\eps([j,\infty]\times B)$. Furthermore, when $j=o(\sqrt{\log n})$, $X_j^{(n)}(B)\approx \wt X^{(n)}_j(B), X_{\geq j}^{(n)}(B)\approx \wt X_{\geq j}^{(n)}(B)$. For any $K\in\N$, take any (fixed) increasing integer sequence $(j_k)_{k\in[K]}$ with $0\leq K':=\min\{k: j_{k+1}=j_K\}$ and any sequence $(B_k)_{k\in[K]}$ with $B_k=(a_k,b_k]\in\CB(\R)$ for some $a_k,b_k\in\R$ and such that $B_k\cap B_\ell=\emptyset$ when $j_k=j_\ell$ and $k\neq \ell$. The conditions on the sets $B_k$ ensure that the elements $\{j_1\}\times B_1, \ldots, \{j_K'\}\times B_{K'}, \{j_{K'+1},\ldots\}\times B_{K'+1}, \ldots, \{j_K,\ldots\}\times B_K$ of $\CA$ are disjoint. We are thus required to prove the joint distributional convergence of the random variables 
	\be\label{eq:xnseqwrt}
	(\wt X_{j_1}^{(n)}(B_1),\ldots,\wt  X^{(n)}_{j_{K'}}(B_{K'}),\wt X^{(n)}_{\geq j_{K'+1}}(B_{K'+1}),\ldots,\wt X_{\geq j_K}^{(n)}(B_K)),
	\ee
	to prove Theorem~\ref{thrm:deglocwrt}.  
	
	\subsection{Intermediate results}
	
	We first state some intermediate results which are required to prove Theorems~\ref{thrm:conddegloc} and~\ref{thrm:deglocwrt} and prove these theorems afterwards. We defer the proof of the intermediate results to Section~\ref{sec:degwrtproof}.
	
	The first result provides precise and general asymptotic bounds for the joint distribution of the degree and label of vertices selected uniformly at random from $[n]$. We recall $\theta=\theta_1:=1+\E W$. We then formulate the following result.
	
	\begin{proposition}[Degree and label of typical vertices]\label{prop:deglocwrt}
		Consider the WRT model, that is, the WRG as in Definition~\ref{def:wrg} with $m=1$, with vertex-weights $(W_i)_{\inn}$ which are i.i.d.\ copies of a positive random variable $W$ that satisfies Condition~\ref{item:c1} of Assumption~\ref{ass:weights}. Fix $k\in\N,c\in(0,\theta/(\theta-1)),\eta\in(0,1)$, and let $(v_i)_{i\in[k]}$ be $k$ vertices selected uniformly at random without replacement from $[n]$. For non-negative integers $(d_i)_{i\in[k]}$ such that $d_i\leq c\log n, i\in[k]$, let $(\ell_i)_{i\in[k]}\in \R_+^k$ be such that they satisfy $\ell_i\leq n\exp(-(1-\zeta)(1-\theta^{-1})(d_i+1))$ and $\ell_i\geq n^\eta$ for all $n$ large, for any $\zeta>0$ and each $i\in[k]$, and let $X_i\sim\mathrm{Gamma}(d_i+1,1),i\in[k]$. Then, uniformly over $d_i\leq c\log n,i\in[k]$,
		\be\ba \label{eq:degdistr}
		\mathbb P({}&\Zm_n(v_i)=d_i, v_i> \ell_i, i\in[k])\\
		&=(1+o(1))\prod_{i=1}^k\E{\frac{\theta-1}{\theta-1+W}\Big(\frac{W}{\theta-1+W}\Big)^{d_i}\Pf{X_i<\Big(1+\frac{W}{\theta-1}\Big)\log(n/\ell_i)}}.
		\ea\ee  
		Moreover, when $d_i=d_i(n)$ diverges with $n$ and with $\wt X_i\sim\text{Gamma}(d_i+\lfloor d_i^{1/4}\rfloor+1,1),i\in[k]$,
		\be \ba \label{eq:degdistrtail}
		\mathbb P({}&\Zm_n(v_i)\geq d_i, v_i> \ell_i, i\in[k])&\\
		&\leq(1+o(1))\prod_{i=1}^k \E{\Big(\frac{W}{\theta-1+W}\Big)^{d_i}\Pf{X_i<\Big(1+\frac{W}{\theta-1}\Big)\log(n/\ell_i)}},\\
		\mathbb P({}&\Zm_n(v_i)\geq d_i, v_i> \ell_i, i\in[k])\\
		&\geq(1+o(1))\prod_{i=1}^k \E{\Big(\frac{W}{\theta-1+W}\Big)^{d_i}\Pf{\wt X_i<\Big(1+\frac{W}{\theta-1}\Big)\log(n/\ell_i)}}.
		\ea \ee 
	\end{proposition}
	\begin{remark}
		$(i)$ We conjecture that the additional condition that $d_i$ diverges with $n$ for all $i\in[k]$ is sufficient but not  necessary for the result in~\eqref{eq:degdistrtail} to hold, and that a sharper lower bound, using $X_i$ instead of $\wt X_i$, can be achieved. These minor differences arise only due to the nature of our proof. However, the results in Proposition~\ref{prop:deglocwrt} are sufficiently strong for the purpose of this paper.
		
		$(ii)$ Lemma~\ref{lemma:expasymp} and Corollary~\ref{cor:expasymp} in the~\hyperref[sec:appendix]{Appendix} provide asymptotic estimates for the probability in~\eqref{eq:degdistrtail} when the vertex-weight distribution satisfy Condition~\ref{item:c2} or satisfies the~\ref{ass:weightatom}, \ref{ass:beta}, or~\ref{ass:gamma} case from Assumption~\ref{ass:weights}, for a  particular parametrisation of $d_i, \ell_i,i\in[k]$.
		
		$(iii)$ Proposition~\ref{prop:deglocwrt} also holds when we consider the definition of the WRT model with \emph{random out-degree}, as discussed in Remark~\ref{remark:def}$(ii)$. For the interested reader, we refer to the discussion after the proof of Lemma $5.10$ in~\cite{EslLodOrt21} for the (minor) adaptations required, which also suffice for the proof of Proposition~\ref{prop:deglocwrt}.
	\end{remark}
	
	With Proposition~\ref{prop:deglocwrt} we can make the heuristic that the maximum degree is of the order $d_n$ when $p_{\geq d_n}\approx 1/n$ rigorous, where 
	\be \label{eq:pk}
	p_{\geq d}:=\E{\Big(\frac{W}{\theta-1+W}\Big)^{d}}, \qquad d\in\N_0,
	\ee
	is the limiting tail degree distribution of the WRT model. This follows from the following lemma.
	
	\begin{lemma}\label{lemma:maxdeg}
		Consider the WRT model, that is, the WRG as in Definition~\ref{def:wrg} with $m=1$, with vertex-weights $(W_i)_{\inn}$ which are i.i.d.\ copies of a positive random variable $W$ that satisfies Condition~\ref{item:c1} of Assumption~\ref{ass:weights}, and recall $\theta=\theta_1=1+\E W$. Fix $c\in(0,\theta/(\theta-1))$ and let $(d_n)_{n\in\N}$ be a positive integer sequence that diverges with $n$ such that $d_n\leq c\log n$. Then,
		\be 
		\lim_{n\to\infty}n\E{\Big(\frac{W}{\theta-1+W}\Big)^{d_n}}=0\quad \Rightarrow \quad 	\lim_{n\to\infty}\P{\max_{\inn}\zni\geq d_n}=0.
		\ee
		Similarly, 
		\be 
		\lim_{n\to\infty}n\E{\Big(\frac{W}{\theta-1+W}\Big)^{d_n}}=\infty\quad \Rightarrow \quad 	\lim_{n\to\infty}\P{\max_{\inn}\zni\geq d_n}=1.
		\ee 
	\end{lemma}

		\begin{remark}
			Lemma~\ref{lemma:maxdeg} can be used to provide precise asymptotic values for the maximum degree in the WRT model. Under assumptions on the distribution of the vertex-weights, it is possible to determine values of $d_n$ for which either $\lim_{n\to\infty}np_{\geq d_n}=0$ or $\lim_{n\to\infty}np_{\geq d_n}=\infty$ is met. In particular, Lemma~\ref{lemma:maxdeg} can be used to extend Theorems $2.6,2.7,$ and Equation $(4.6)$ in Theorem $4.6$ of~\cite{EslLodOrt21} to a wider range of vertex-weight distributions. Namely, in~\cite{EslLodOrt21}, condition~\ref{item:c3} is required for a result equivalent to Lemma~\ref{lemma:maxdeg} to hold. This result is used to prove the aforementioned theorems. Here, however, we do not need Condition~\ref{item:c3} for Lemma~\ref{lemma:maxdeg}, so that these Theorems can be extended to a wider range of vertex-weight distributions. 
		\end{remark} 
	
	We now present a proposition which asymptotically determines the joint factorial moments of the random variables $X_j^{(n)}(B)$ and $X_{\geq j}^{(n)}(B)$, as in~\eqref{eq:xnwrt}, when the vertex-weight distribution satisfies the~\ref{ass:weightatom} case. It is instrumental for the proof of Theorem~\ref{thrm:deglocwrt}.
	
	\begin{proposition}\label{prop:momentconvwrt}
		Consider the WRT model, that is, the WRG model as in Definition~\ref{def:wrg} with $m=1$, with vertex-weights $(W_i)_{\inn}$ that satisfy the~\ref{ass:weightatom} case in Assumption~\ref{ass:weights} for some $q_0\in(0,1]$. Recall that $\theta:=1+\E W$ and that $(x)_k:=x(x-1)\cdots (x-(k-1))$ for $x\in\R,k\in\N$, and $(x)_0:=1$. Fix $c\in(0,\theta/(\theta-1))$ and $K\in\N$, let $(j_k)_{k\in[K]}$ be  a non-decreasing integer sequence with $0\leq K':=\min\{k: j_{k+1}=j_K\}$ such that $j_1+\log_\theta n=\omega(1)$ and $j_K+\log_\theta n<c\log n$, let $(B_k)_{k\in[K]}$ be a sequence of sets $B_k\in \CB(\R)$ such that $B_k\cap B_\ell=\emptyset $ when $j_k=j_\ell$ and $k\neq \ell$, and let $(c_k)_{k\in[K]}\in \N_0^K$. Recall the random variables $X^{(n)}_j(B), X_{\geq j}^{(n)}(B)$ and $\wt X^{(n)}_j(B), \wt X_{\geq j}^{(n)}(B)$ from~\eqref{eq:xnwrt} and define $\eps_n:=\log_\theta n-\lfloor \log_\theta n\rfloor$. Then, 
		\be \ba
		\E{\prod_{k=1}^{K'}\Big(X_{j_k}^{(n)}(B_k)\Big)_{c_k}\prod_{k=K'+1}^K \Big(X_{\geq j_k}^{(n)}(B_k)\Big)_{c_k}}={}&(1+o(1))\prod_{k=1}^{K'}\Big(q_0(1-\theta^{-1})\theta^{-j_k+\eps_n}\Phi(B_k)\Big)^{c_k}\\
		&\times \prod_{k=K'+1}^{K}\Big(q_0\theta^{-j_K+\eps_n}\Phi(B_k)\Big)^{c_k}.
		\ea\ee 
		Moreover, when $j_1,\ldots, j_K=o(\sqrt{\log n})$, 
		\be \ba
		\E{\prod_{k=1}^{K'}\Big(\wt X_{j_k}^{(n)}(B_k)\Big)_{c_k}\prod_{k=K'+1}^K \Big(\wt X_{\geq j_k}^{(n)}(B_k)\Big)_{c_k}}={}&(1+o(1))\prod_{k=1}^{K'}\Big(q_0(1-\theta^{-1})\theta^{-j_k+\eps_n}\Phi(B_k)\Big)^{c_k}\\
		&\times \prod_{k=K'+1}^{K}\Big(q_0\theta^{-j_K+\eps_n}\Phi(B_k)\Big)^{c_k}.
		\ea\ee 
	\end{proposition}
	
	We can interpret the results in Proposition~\ref{prop:momentconvwrt} as follows. Fix some $(j_k)_{k\in[K]}$ and $(B_k)_{k\in[K]}$ as in the proposition (we note that the $j_k$ are allowed to be a function of $n$, but for simplicity we do not discuss this case here). Then the result of the proposition tells us that the joint factorial moments of the random variables $X^{(n)}_{j_k}(B_k)$ and $X^{(n)}_{\geq j_k}(B_k)$ are asymptotically equal to a product of terms $(q_0(1-\theta^{-1})\theta^{-j_k+\eps_n}\theta(B_k))^{c_k}$ and $(q_0\theta^{-j_K+\eps_n}\theta(B_k))^{c_k}$, respectively. Since $\eps_n$ is bounded, it converges along subsequences to some value $\eps\in[0,1]$. Hence, the method of moments yields that the random variables of interest are asymptotically independent and that their limits, along certain subsequences, are Poisson random variables. Thus, the number of vertices with a degree equal to, or at least, $j_k+\log_\theta n$ and a label $i$ such that 
		\be 
		\frac{\log i-(\log n-(1-\theta^{-1})(\lfloor \log_\theta n\rfloor +j_k))}{\sqrt{(1-\theta^{-1})^2(\lfloor \log_\theta n\rfloor +j_k)}}\in B_k,  
		\ee   
		is asymptotically Poisson distributed. A similar statement can be made for the random variables $\wt X^{(n)}_{j_k}(B_k)$ and $\wt X^{(n)}_{\geq j_k}(B_k)$.
		
		A similar result can be proved for the~\ref{ass:beta} and~\ref{ass:gamma} cases, which we defer to Section~\ref{sec:examples}.
	
	\subsection{Proofs of main results}
	
	With the intermediate results at hand, we can prove Theorems~\ref{thrm:conddegloc} and~\ref{thrm:deglocwrt}.
	
	\begin{proof}[Proof of Theorem~\ref{thrm:conddegloc} subject to Proposition~\ref{prop:deglocwrt}]
		We recall that $d_i$ diverges as $n\to\infty$ for all $i\in[k]$ such that $c_i:=\limsup_{n\to\infty} d_i/\log n$ is strictly smaller than $\theta/(\theta-1)$ for all $i\in[k]$ and define, for $(x_i)_{i\in[k]}\in \R^k$ fixed,
		\be \label{eq:ell}
		\ell_i:=n\exp(-(1-\theta^{-1})d_i+x_i\sqrt{(1-\theta^{-1})^2d_i}), \qquad i\in[k].
		\ee 
		We first observe that by this definition, 
		\be 
		\Big\{\frac{\log v_i-(\log n-(1-\theta^{-1})d_i)}{\sqrt{(1-\theta^{-1})^2d_i}}\geq x_i, i\in[k]\Big\}=\{v_i>\ell_i, i\in[k]\}.
		\ee 
		Furthermore, we note that there exists an $\eta>0$ such that for all $i\in[k]$, we have $\ell_i\geq n^\eta$ and $\ell_i\leq n\exp(-(1-\zeta)(1-\theta^{-1})(d_i+1))$ for all $\zeta>0$ and all $n$ sufficiently large. Hence, the conditions in Proposition~\ref{prop:deglocwrt} are satisfied. We then write
			\be 
			\P{v_i\geq \ell_i, i\in[k]\,|\, \Zm_n(v_i)\geq d_i, i\in[k]}=\frac{\P{\Zm_n(v_i)\geq d_i, v_i\geq \ell_i, i\in[k]}}{\P{\Zm_n(v_i)\geq d_i,i\in[k]}}.
			\ee 
			We now combine Proposition~\ref{prop:deglocwrt} with Lemma~\ref{lemma:expasymp} in the~\hyperref[sec:appendix]{Appendix}. As we assume that the vertex-weight distribution satisfies Conditions~\ref{item:c1} and~\ref{item:c2} of Assumption~\ref{ass:weights}, it follows that
			\be \label{eq:noncondprob}
			\P{\Zm_n(v_i)\geq d_i, v_i>\ell_i, i\in[k]}=(1+o(1))\prod_{i=1}^k p_{\geq d_i}(1-\Phi(x_i)),
			\ee 		
			where we recall $p_{\geq d}$ from~\eqref{eq:pk}. It thus remains to show that 
			\be \label{eq:degasymp}
			\P{\Zm_n(v_i)\geq d_i, i\in[k]}=(1+o(1))\prod_{i=1}^k p_{\geq d_i}.
			\ee 
			We first assume that $c_i<1/\log \theta$ for all $i\in[k]$. We can then take any $\eps\in(0,\mu)$, and for all $n$ sufficiently large $n^{\mu-\eps}\leq n\exp(-(1-\theta^{-1})(d_i+1))$	holds for all $i\in[k]$. It then follows from Proposition~\ref{prop:deglocwrt} (with $\ell_i=n^{\mu-\eps}$ for all $i\in[k]$) and Lemma~\ref{lemma:asymp} that 
			\be \label{eq:crudeasymp}
			\P{\Zm_n(v_i)\geq d_i, i\in[k]}\geq\P{\Zm_n(v_i)\geq d_i,v_i\geq n^{\mu-\eps},  i\in[k]}=(1+o(1))\prod_{i=1}^k p_{\geq d_i}.
			\ee 
			It remains to prove a matching upper bound, for which we use that for any $\eta>0$ small,
			\be \ba\label{eq:split}
			\P{\Zm_n(v_i)\geq d_i,i\in[k]}\leq{}& \P{\Zm_n(v_i)\geq d_i,v_i\geq n^\eta,i\in[k]}\\
			&+\P{\Big(\cap_{i=1}^k \{\Zm_n(v_i)\geq d_i\}\Big)\cap \Big(\cup_{i=1}^k \{v_i<n^\eta\}\Big)}.
			\ea \ee 
			The first term on the right-hand side can be dealt with in the same manner as~\eqref{eq:crudeasymp} by setting $\eta=\mu-\eps$ with $\eps$ sufficiently close to $\mu$. We write the second term as
			\be \ba\label{eq:Sbound}
			\sum_{j=1}^k{}& \sum_{\substack{S\subseteq[k]\\|S|=j}}\P{\Zm_n(v_i)\geq d_i, i\in[k], v_j<n^\eta, j\in S, v_m>n^\eta, m\in S^c}\\
			&\leq \sum_{j=1}^k \sum_{\substack{S\subseteq[k]\\|S|=j}}\P{\Zm_n(v_i)\geq d_i, v_i>n^\eta, i\in S^c}\P{v_j<n^\eta, j\in S}\\
			&\leq \sum_{j=1}^k \sum_{\substack{S\subseteq[k]\\|S|=j}} (1+o(1)) n^{-j(1-\eta)}\prod_{i\in S^c}p_{\geq d_i},
			\ea \ee 
			where we use that the uniform vertices $(v_i)_{i\in S}$ are independent of everything else, and where we take care of the other probability in the second line in the same manner as the first term on the right-hand side of~\eqref{eq:split}. We now use Theorem~\ref{thrm:pkasymp} to bound $p_{\geq d}\geq (\theta+\xi)^{-d}=\exp(-d\log(\theta+\xi))$ for any $\xi>0$ and $d$ sufficiently large. Since $c_i<1/\log\theta$ for all $i\in[k]$, it thus follows that for $\xi$ and $\eta$ sufficiently small, $n^{-(1-\eta)}=o(p_{\geq d_i})$ for all $i\in[k]$. Hence, the final line of~\eqref{eq:Sbound} is $o(\prod_{i=1}^k p_{\geq d_i})$. In~\eqref{eq:split}, we thus find that
			\be 
			\P{\Zm_n(v_i)\geq d_i,i\in[k]}\leq(1+o(1))\prod_{i=1}^k p_{\geq d_i}.
			\ee 
			Combined with~\eqref{eq:crudeasymp}, this proves~\eqref{eq:degasymp} and thus the desired result.
		
		To extend the proof to $c_i\in[1/\log \theta,\theta/(\theta-1))$, the lower bound in~\eqref{eq:crudeasymp} is still valid when we choose $\eps$ sufficiently close to $\mu$ so that $n^{\mu-\eps}\leq n\exp(-(1-\theta^{-1})(d_i+1))$ still holds for all $i\in[k]$. To be more precise, when we let $\eps\in(c(1-\theta^{-1})-(1-\mu),\mu)$, where $c\in(\max_{i\in[k]}c_i,\theta/(\theta-1))$. The upper bound, however, no longer suffices, since the error terms on the right-hand side of~\eqref{eq:Sbound} no longer decay sufficiently fast. Instead, we require Condition~\ref{item:c3} of Assumption~\ref{ass:weights}. With this condition and since $c_i<\theta/(\theta-1)$ for all $i\in[k]$,  we can apply Proposition~\ref{lemma:degprobasymp}. This yields 
		\be 
		\P{\Zm_n(v_i)\geq d_i, i\in[k]}=(1+o(1))\prod_{i=1}^k p_{\geq d_i}.
		\ee  
		Together with~\eqref{eq:noncondprob} this implies the same result.
		
		Using Remark~\ref{rem:expasymp}$(i)$ and $(ii)$ (together with Proposition~\ref{lemma:degprobasymp} for the case $c_i\in[1/\log \theta,\theta/(\theta-1))$ for all $i\in[k]$), a similar result can be proved when conditioning on the event $\{\Zm_n(v_i)= d_i, i\in[k]\}$, as claimed in Remark~\ref{rem:degthrm}.
	\end{proof}	
	
	\begin{proof}[Proof of Theorem~\ref{thrm:deglocwrt} subject to Proposition~\ref{prop:momentconvwrt}]
		As discussed prior to~\eqref{eq:xnwrt}, it suffices to prove the weak convergence of $\CM\CP^{(n_j)}$ to $\CM\CP^\eps$ along subsequences $(n_j)_{j\in\N}$ such that $\eps_{n_j}\to \eps\in[0,1]$ as $j\to\infty$. In turn, this is implied by the convergence of the FDDs, i.e., by the joint convergence of the counting measures in~\eqref{eq:xnseqwrt}. 
		
		We recall that the points $P_i$ in the definition of the variables $X_j(B),X_{\geq j}(B)$ in~\eqref{eq:xnwrt} are the points of the Poisson point process $\CP$ with intensity measure $\lambda(x):= q_0\theta^{-x}\log \theta\,\d x$ in decreasing order. As a result, as the random variables $(M_i)_{i\in\N}$ are i.i.d.\ and also independent of $\CP$, $X_j(B)\sim \text{Poi}(\lambda_j(B)), X_{\geq j}(B)\sim \text{Poi}((1-\theta^{-1})^{-1}\lambda_j(B))$, where 
		\be 
		\lambda_j(B)=q_0(1-\theta^{-1})\theta^{-j+\eps}\Phi(B)=q_0(1-\theta^{-1})\theta^{-j+\eps}\P{M_1\in B}.
		\ee 
		We also recall that $(n_\ell)_{\ell\in\N}$ is a subsequence such that $\eps_{n_\ell}\to \eps$ as $\ell\to\infty$. We now take $c\in(1/\log \theta,\theta/(\theta-1))$ and for any $K\in\N$ consider any \emph{fixed} non-decreasing integer sequence $(j_k)_{k\in[K]}$. It follows from the choice of $c$ and the fact that the $j_k$ are fixed with respect to $n$ that $j_1+\log_\theta n=\omega(1)$ and that $j_K+\log_\theta n <c\log n$ for all large $n$. Moreover, let $K':=\min\{k: j_{k+1}=j_K\}$ and let $(B_k)_{k\in[K]}$ be a sequence of sets in $\CB(\R)$ such that $B_k\cap B_\ell=\emptyset$ when $j_k=j_\ell$ and $k\neq \ell$.
		
		We obtain from Proposition~\ref{prop:momentconvwrt} that, for any $(c_k)_{k\in[K]}\in \N_0^K$, and since $j_1,\ldots, j_K$ are fixed, 
		\be  \ba
		\lim_{n\to\infty}\mathbb E\bigg[\prod_{k=1}^{K'}\Big(\wt X_{j_k}^{(n_\ell)}(B_k)\Big)_{c_k}\prod_{k=K'+1}^K\!\!\Big(\wt X_{\geq j_k}^{(n_\ell)}(B_k)\Big)_{c_k} \bigg]
		&=\prod_{k=1}^{K'}\lambda_{j_k}^{c_k}\prod_{k=K'+1}^K((1-\theta^{-1})^{-1}\lambda_{j_k})^{c_k}\\
		&=\mathbb E\bigg[\prod_{k=1}^{K'}\Big(X_{j_k}(B_k)\Big)_{c_k}\prod_{k=K'+1}^K\!\!\Big(X_{\geq j_k}(B_k)\Big)_{c_k} \bigg],
		\ea \ee
		where the last step follows from the independence property of (marked) Poisson point processes and the choice of the sequences $(j_k, B_k)_{k\in[K]}$. The method of moments~\cite[Section $6.1$]{JanLucRuc00} then concludes the proof.
	\end{proof}
	
	\section{Proof of intermediate results}\label{sec:degwrtproof}
	
	In this section we prove the intermediate results introduced in Section~\ref{sec:highproof} that were used to prove some of the main results presented in Section~\ref{sec:def}. We start by proving Lemmas~\ref{lemma:maxdeg} and~\ref{prop:momentconvwrt} (subject to Proposition~\ref{prop:deglocwrt}) and finally prove Proposition~\ref{prop:deglocwrt}, which requires the most work and hence is deferred until the end of the section.
	
	\begin{proof}[Proof of Lemma~\ref{lemma:maxdeg} subject to Proposition~\ref{prop:deglocwrt}]
		Fix $\eps\in(0\vee (c(1-\theta^{-1})-(1-\mu)),\mu)$. We note that such an $\eps$ exists, since $c<\theta/(\theta-1)$. We start with the first implication. By Theorem~\ref{thrm:bddloc} and a union bound we have
		\be \ba \label{eq:maxdegub}
		\P{\max_{\inn}\zni \geq d_n}&\leq \P{\max_{\inn}\zni\geq d_n, I_n> n^{\mu-\eps}}+\P{I_n\leq n^{\mu-\eps}}\\
		&\leq \P{\max_{n^{\mu-\eps}<i \leq n}\zni\geq d_n}+o(1)\\
		&\leq \sum_{i=\lceil n^{\mu-\eps}\rceil}^n \P{\zni \geq d_n}+o(1)\\
		&=n\P{\Zm_n(v_1)\geq d_n, v_1> n^{\gamma-\eps}}+o(1), 
		\ea \ee  
		where $v_1$ is a vertex selected uniformly at random from $[n]$. We now apply Proposition~\ref{prop:deglocwrt} with $k=1, d_1=d_n, \ell_1=n^{\mu-\eps}$ (we observe that, since $\eps<\mu$ and by the bound on $d_n$, the conditions in Proposition~\ref{prop:deglocwrt} for $\ell_1$ and $d_1$ are satisfied) to obtain the upper bound
		\be 
		\P{\max_{\inn}\zni\geq d_n}\leq n\E{\Big(\frac{W}{\theta-1+W}\Big)^{d_n}\Pf{X\leq \Big(1+\frac{W}{\theta-1}\Big)\log (n^{1-\mu+\eps})}}(1+o(1))+o(1),
		\ee 
		where $X\sim\text{Gamma}(d+1,1)$. We can simply bound the conditional probability from above by one, so that the assumption yields the desired implication. 
		
		For the second implication, we use the Chung-Erd{\H o}s inequality. If we let $v_1,v_2$ be two vertices selected uniformly at random without replacement from $[n]$ and set $A_{i,n}:=\{\zni\geq d_n\}$, then
		\be \label{eq:chungerd}
		\P{\max_{\inn}\zni\geq d_n}=\P{\cup_{i=1}^n A_{i,n}}\geq \P{\cup_{i=\lceil n^{\mu-\eps}\rceil}^n A_{i,n}}\geq \frac{\big(\sum_{i=\lceil n^{\mu-\eps}\rceil}^n \P{A_{i,n}}\big)^2}{\sum_{i,j= \lceil n^{\mu-\eps}\rceil}^n \P{A_{i,n}\cap A_{j,n}}}.
		\ee 
		As in~\eqref{eq:maxdegub}, we can write the numerator as $(n\P{\Zm_n(v_1)\geq d_n, v_1\geq n^{\mu-\eps}})^2$. The denominator can be written as 
		\be \ba
		\sum_{\substack{i,j= \lceil n^{\mu-\eps}\rceil\\i\neq j}}^n \P{A_{i,n}\cap A_{j,n}}+\sum_{i=\lceil n^{\mu-\eps}\rceil}^n \P{A_{i,n}}={}&n(n-1)\P{\Zm_n(v_i)\geq d_n, v_i\geq n^{\mu-\eps}, i\in\{1,2\}}\\
		&+n\P{\Zm_n(v_1)\geq d_n, v_1\geq n^{\mu-\eps}}.
		\ea \ee 
		By applying Proposition~\ref{prop:deglocwrt} to the right-hand side, we find that it equals
		\be 
		(n\P{\Zm_n(v_1)\geq d_n, v_1\geq n^{\mu-\eps}})^2(1+o(1))+n\P{\Zm_n(v_1)\geq d_n, v_1\geq n^{\mu-\eps}}.
		\ee
		It follows that the right-hand side of~\eqref{eq:chungerd} equals
		\be \label{eq:degloclb}
		\frac{n\P{\Zm_n(v_1)\geq d_n, v_1\geq n^{\mu-\eps}}}{n\P{\Zm_n(v_1)\geq d_n, v_1\geq n^{\mu-\eps}}(1+o(1))+1}.
		\ee 
		It thus suffices to prove that the implication
		\be \label{eq:logic}
		\lim_{n\to\infty}n\E{\Big(\frac{W}{\theta-1+W}\Big)^{d_n}}=\infty\quad \Rightarrow \quad 	\lim_{n\to\infty}n\P{\Zm_n(v_1)\geq d_n, v_1\geq n^{\mu-\eps}}=\infty
		\ee 
		holds to conclude the proof. Again using Proposition~\ref{prop:deglocwrt}, we have that 
		\be 
		\P{\Zm_n(v_1)\geq d_n, v_1\geq n^{\mu-\eps}}\geq \E{\Big(\frac{W}{\theta-1+W}\Big)^{d_n}\Pf{\wt X\leq \Big(1+\frac{W}{\theta-1}\Big)\log (n^{1-\mu+\eps})}}(1+o(1)),
		\ee 
		where $\wt X\sim\text{Gamma}(d+\lfloor d^{1/4}\rfloor+1,1)$. Hence, it follows from Lemma~\ref{lemma:asymp} in the~\hyperref[sec:appendix]{Appendix} and the choice of $\eps$ that 
		\be 
		n\P{\Zm_n(v_1)\geq d_n, v_1\geq n^{\mu-\eps}}\geq n\E{\Big(\frac{W}{\theta-1+W}\Big)^{d_n}}(1-o(1)),
		\ee 
		which implies~\eqref{eq:logic} as desired and concludes the proof.
	\end{proof}
	
	\begin{proof}[Proof of Proposition~\ref{prop:momentconvwrt} subject to Proposition~\ref{prop:deglocwrt}]
		Recall that $c\in(0,\theta/(\theta-1))$, that $\mu=1-(\theta-1)/(\theta\log \theta)$, $\sigma^2=1-(\theta-1)^2/(\theta^2\log \theta)$, and that we have a non-decreasing integer sequence $(j_k)_{k\in[K]}$ with $K'=\min\{k: j_{k+1}=j_K\}$ such that $j_1+\log_\theta n =\omega(1), j_K+\log_\theta n<c\log n$ and a sequence $(B_k)_{k\in[K]}$ such that $B_k\in \CB(\R)$ and $B_k\cap B_\ell=\emptyset$ when $j_k=j_\ell$ and $k\neq \ell$. Then, let $(c_k)_{k\in[K]}\in \N_0^K$ and set  $M:=\sum_{k=1}^K c_k$ and $M':=\sum_{k=1}^{K'}c_k$.
		
		We define $\bar d=(d_i)_{i\in[M]}\in \Z^M$ and $\bar A=(A_i)_{i\in[M]}\subset \CB(\R)^M$ as follows. For each $i\in[M]$, find the unique $k\in[K]$ such that $\sum_{\ell=1}^{k-1}c_\ell<i\leq \sum_{\ell=1}^k c_\ell$ and set $d_i:=\lfloor \log_\theta n\rfloor +j_k, A_i:=B_k$. We note that this construction implies that the first $c_1$ many $d_i$ and $A_i$ equal $\lfloor \log_\theta n\rfloor +j_1$ and $B_1$, respectively, that the next $c_2$ many $d_i$ and $A_i$ equal $\lfloor \log_\theta n\rfloor +j_2$ and $B_2$, respectively, etcetera. Moreover, we let $(v_i)_{i\in[M]}$ be $M$ vertices selected uniformly at random without replacement from $[n]$. We then define the events
		\be \ba 
		\CL_{\bar A,\bar d}&:=\Big\{ \frac{\log v_i-(\log n-(1-\theta^{-1})d_i)}{\sqrt{(1-\theta^{-1})^2 d_i}}\in A_i, i\in[M]\Big\},\\
		\CD_{\bar d}(M',M)&:=\{\Zm_n(v_i)=d_i, i\in[M'], \Zm_n(v_j)\geq d_j, M'<j\leq M\},\\
		\CE_{\bar d}(S)&:=\{\Zm_n(v_i)\geq d_i+\ind_{\{i\in S\}}, i\in [M]\}.
		\ea \ee 
		We know from~\cite[Lemma $5.1$]{AddEsl18} that by the inclusion-exclusion principle, 
		\be\label{eq:inex}
		\P{\CD_{\bar d}(M',M)}=\sum_{j=0}^{M'}\sum_{\substack{ S\subseteq [M']:\\ |S|=j}}(-1)^j\P{\CE_{\bar d}(S)},		
		\ee
		so that intersecting the event $\CL_{\bar A.\bar d}$ in the probabilities on both sides yields
		\be \label{eq:probdhl}
		\P{\CD_{\bar d}(M',M)\cap \CL_{\bar A,\bar d}}=\sum_{j=0}^{M'}\sum_{\substack{ S\subseteq [M']:\\ |S|=j}}(-1)^j\P{\CE_{\bar d}(S)\cap \CL_{\bar A,\bar d}}.
		\ee 
		We define $\ell_d:\R\to (0,\infty)$ by $\ell_d(x):= \exp\big(\log n-(1-\theta^{-1})d+x\sqrt{(1-\theta^{-1})^2d}\big), x\in \R$, abuse this notation to also write $\ell_d( A):=\{\ell_d(x): x\in A\}$ for $A\subseteq \R$, and note that $\CL_{\bar A,\bar d}=\{v_i\in \ell_{d_i}(A_i), i\in[M]\}$. We also observe that, since $d_i$ diverges with $n$ for all $i\in[M]$, that $\ell_{d_i+\ind_{\{i\in S\}}}(x)=\ell_{d_i}(x(1+o(1)))$ for any $i\in[M]$ and $x\in\R$. This can be extended to the sets $(A_i)_{i\in[M]}$ rather than $x\in\R$ as well. As a result, we can use Corollary~\ref{cor:expasymp} in the~\hyperref[sec:appendix]{Appendix} (with the observations made in Remark~\ref{rem:expasymp}) to then obtain
		\be 
		\P{\CE_{\bar d}(S)\cap \CL_{\bar A,\bar d}(S)}=(1+o(1))\prod_{i=1}^M q_0\theta^{-(d_i+\ind_{\{i\in S\}})}\Phi(A_i)=(1+o(1))q_0^M\theta^{-|S|-\sum_{i=1}^M d_i}\prod_{i=1}^M \Phi(A_i).
		\ee 
		Using this in~\eqref{eq:probdhl} we arrive at 
		\be\ba \label{eq:condlimit}
		\P{\CD_{\bar d}(M',M)\cap \CL_{\bar A,\bar d}}&=(1+o(1))q_0^M \theta^{-\sum_{i=1}^M d_i}\prod_{i=1}^M\Phi(A_i)\sum_{j=0}^{M'}\sum_{\substack{ S\subseteq [M']:\\ |S|=j}}(-1)^j \theta^{-j}\\
		&=(1+o(1))q_0^M \theta^{-\sum_{i=1}^M d_i}(1-\theta^{-1})^{M'}\prod_{i=1}^M\Phi(A_i),
		\ea \ee 
		where the $1+o(1)$ and the product  on the right-hand side are independent of $S$ and $j$ and can therefore be taken out of the double sum. Now, recall the definition of the variables $X_j^{(n)}(B),X_{\geq j}^{(n)}(B)$ as in~\eqref{eq:xnwrt}. Combining~\eqref{eq:probdhl} and \eqref{eq:condlimit}, we arrive at
		\be\ba \label{eq:bigprod}
		\mathbb E\Bigg[\prod_{k=1}^{K'}\Big(X_{j_k}^{(n)}(B_k)\Big)_{c_k}\prod_{k=K'+1}^K \Big(X_{\geq j_k}^{(n)}(B_k)\Big)_{c_k}\Bigg]
		&=(n)_M \P{ \CD_{\bar d}(M',M)\cap \CL_{\bar A.\bar d} }\\
		&\sim q_0^M \theta^{M\log_\theta n-\sum_{i=1}^M d_i}(1-\theta^{-1})^{M'}\prod_{i=1}^M\Phi(A_i),
		\ea\ee 
		since $(n)_M:=n(n-1)\cdots (n-(M-1))=(1+o(1))n^M$ and where we recall that $a_n\sim b_n$ denotes $\lim_{n\to\infty}a_n/b_n=1$. We now recall that there are exactly $c_k$ many $d_i$ and $A_i$ that equal $\lfloor \log_2n\rfloor +j_k$ and $B_k$, respectively, for each $k\in[K]$ and that $j_{K'+1}=\ldots =j_K$, so that 
		\be \ba \label{eq:mult}
		\prod_{i=1}^M \Phi(A_i)&=\prod_{k=1}^K \Phi(B_k)^{c_k},\\
		M\log_\theta n-M'-\sum_{i=1}^M d_i &= -\sum_{k=1}^{K'}(j_k+1-\eps_n)c_k-\sum_{k=K'+1}^K (j_K-\eps_n)c_k,
		\ea \ee 
		which, combined with~\eqref{eq:bigprod}, finally yields
		\be \ba\label{eq:fin}
		\mathbb E\Bigg[\!\prod_{k=1}^{K'}\!\bigg(\!X_{j_k}^{(n)}(B_k)\!\bigg)_{c_k}\prod_{k=K'+1}^K\!\! \bigg(\!X_{\geq j_k}^{(n)}(B_k)\!\bigg)_{c_k}\!\Bigg]\!={}&(1+o(1))\prod_{k=1}^{K'} \big(q_0(1-\theta^{-1})\theta^{-j_k+\eps_n}\Phi(B_k)\big)^{c_k}\\
		&\times \prod_{k=K'+1}^K \big(q_0\theta^{-j_K+\eps_n}\Phi(B_k)\big)^{c_k}.
		\ea \ee 
		To prove the second result, we observe that for $j_1,\ldots, j_K=o(\sqrt{\log n})$, 
		\be 
		\frac{\log v_i-(\log n-(1-\theta^{-1})d_i)}{\sqrt{(1-\theta^{-1})^2 d_i}}=\frac{\log v_i-\mu\log n}{\sqrt{(1-\sigma^2)\log n}}(1+o(1))+o(1).
		\ee 
		Hence, the same steps as above can be applied to the random variables $\wt X^{(n)}_j(B),\wt X^{(n)}_{\geq j}(B)$ to obtain the desired result.		
	\end{proof}
	
	We finally prove Proposition~\ref{prop:deglocwrt}. This result extends and improves Proposition~\ref{lemma:degprobasymp} and~\cite[Lemma $5.10$]{EslLodOrt21}, which one could think of analogous result with $\ell_i=n^\eps$ for all $i\in[k]$ and some $\eps>0$ small. We split the proof of the proposition into three main parts. We first prove an upper bound for~\eqref{eq:degdistr}, then prove a matching lower bound for~\eqref{eq:degdistr} (up to error terms) and finally prove~\eqref{eq:degdistrtail}.
	
	\begin{proof}[Proof of Proposition~\ref{prop:deglocwrt}, Equation~\eqref{eq:degdistr}, upper bound]
		We assume without loss of generality that $\ell_1, \ldots, \ell_k$ are integer-valued. If they would not be, we would use $\lceil \ell_1\rceil, \ldots, \lceil \ell_k\rceil$ which yields the same result. By first conditioning on the value of $v_1, \ldots, v_k$, we obtain 
		\be 
		\P{\Zm_n(v_i)= d_i, v_i>\ell_i, i\in[k]}=\frac{1}{(n)_k}\sum_{j_1=\ell_1+1}^n \sum_{\substack{j_2=\ell_2+1\\ j_2\neq j_1}}^n \cdots \sum_{\substack{j_k=\ell_k+1\\ j_k\neq j_{k-1}, \ldots, j_1}}^n \P{\Zm_n(j_i)=d_i, i\in[k]}.
		\ee 
		If we let $\CP_k$ be the set of all permutations on $[k]$, we can write the sums on the right-hand side as 
		\be \label{eq:aimbound}
		\frac{1}{(n)_k}\sum_{\pi\in\CP_k}\sum_{j_{\pi(1)}=\ell_{\pi(1)}}^n \sum_{j_{\pi(2)}=(\ell_{\pi(2)}\vee j_{\pi(1)})+1}^n \cdots \sum_{j_{\pi(k)}=(\ell_{\pi(k)}\vee j_{\pi(k-1)})+1}^n \P{\Zm_n(j_i)=d_i, i\in[k]}.
		\ee 
		To prove an upper bound of this expression, we first consider the identity permutation, i.e.\ $\pi(i)=i$ for all $i\in[k]$, and take
		\be\label{eq:aimub}
		\frac{1}{(n)_k}\sum_{j_1=\ell_1}^n \sum_{j_2=(\ell_2\vee j_1)+1}^n \cdots \sum_{j_k=(\ell_k\vee j_{k-1})+1}^n \P{\Zm_n(j_i)=d_i, i\in[k]}.
		\ee
		One can think of this as all realisations $v_i=j_i, i\in[k]$ where $j_1<j_2<\ldots <j_k$ and $j_i>\ell_i$ for all $i\in[k]$. We discuss what changes when using other $\pi\in\CP_k$ in~\eqref{eq:aimbound} later on. Let us introduce the event
		\be \label{eq:en}
		E_n^{(1)}:=\bigg\{ \sum_{\ell=1}^j W_\ell \in ((1-\zeta_n)\E{W}j,(1+\zeta_{n})\E{W}j),\ \forall \ n^\eta\leq j\leq n\bigg\},
		\ee 
		where $\zeta_n=n^{-\delta\eta}/\E{W}$ for some $\delta\in(0,1/2)$ and where we recall $n^\eta$ is a lower bound for all $\ell_i,i\in[k]$, with $\eta\in(0,1)$. It follows from Lemma~\ref{lemma:weightsumbounds} that $\mathbb P((E_n^{(1)})^c)=o(n^{-\gamma})$ for any $\gamma>0$.  We can hence bound~\eqref{eq:aimub} from above, for any $\gamma>0$, by
		\be \ba \label{eq:enbound}
		\frac{1}{(n)_k}\sum_{j_1=\ell_1}^n\ldots \!\!\!\!\sum_{j_k=(\ell_k\vee j_{k-1})+1}^n\!\!\!\!\!\!\!\!\mathbb E[\Pf{\Zm_n(j_\ell)=m_\ell,\ell\in[k]}\ind_{E_n^{(1)}}]+o(n^{-\gamma}),
		\ea \ee 
		Now, to express the first term in~\eqref{eq:enbound} we introduce the ordered indices $j_i<m_{1,i}<\ldots<m_{d_i,i}\leq n,i\in[k]$, which denote the steps at which vertex $j_i$ increases its degree by one. Note that for every $i\in[k]$ these indices are distinct by definition, but we also require that $m_{s,i}\neq m_{t,h}$ for any distinct $i,h\in[k],s\in[d_i],t\in[d_h]$ (equality is allowed only when $i=h$ and $s=t$). We denote this constraint by adding a $*$ on the summation symbol. If we also define $j_{k+1}:=n$, we can write the first term in~\eqref{eq:enbound} as
		\be\ba \label{eq:ubfirst}
		\frac{1}{(n)_k}{}&\sum_{j_1=\ell_1}^n\ldots \!\!\!\!\sum_{j_k=(\ell_k\vee j_{k-1})+1}^n\ \, \sideset{}{^*}\sum_{\substack{j_i<m_{1,i}<\ldots<m_{d_i,i}\leq n,\\i\in[k]}}\mathbb E\Bigg[\prod_{t=1}^k\prod_{s=1}^{d_t}\frac{W_{j_t}}{\sum_{\ell=1}^{m_{s,t}-1}W_\ell}\\ &\times \prod_{u=1}^k\!\!\prod_{\substack{s=j_u+1\\s\neq m_{i,t},t\in[d_i],i\in[k]}}^{j_{u+1}}\!\!\!\!\bigg(1-\frac{\sum_{\ell=1}^u W_{j_\ell}}{\sum_{\ell=1}^{s-1}W_{\ell}}\bigg)\ind_{E_n^{(1)}}\Bigg].
		\ea \ee   
		We then include the terms where $s=m_{i,t}$ for $i\in[d_t],t\in[k]$ in the second double product. To do this, we need to change the first double product to
		\be \label{eq:fracbound}
		\prod_{t=1}^k \prod_{s=1}^{d_t}\frac{W_{j_t}}{\sum_{\ell=1}^{m_{s,t}-1}W_\ell-\sum_{\ell=1}^k W_{j_\ell}\ind_{\{m_{s,t}>j_\ell\}}}\leq\prod_{t=1}^k \prod_{s=1}^{d_t} \frac{W_{j_t}}{\sum_{\ell=1}^{m_{s,t}-1}W_\ell-k},
		\ee 
		that is, we subtract the vertex-weight $W_{j_\ell}$ in the numerator when the vertex  $j_\ell$ has already been introduced by step $m_{s,t}$. In the upper bound we use that the weights are bounded from above by one. We thus arrive at the upper bound		
		\be\ba \label{eq:tailprobmiddle}
		\frac{1}{(n)_k}\sum_{j_1=\ell_1}^n\ldots \!\!\!\!\sum_{j_k=(\ell_k\vee j_{k-1})+1}^n\ \, \sideset{}{^*}\sum_{\substack{j_i<m_{1,i}<\ldots<m_{d_i,i}\leq n,\\i\in[k]}}\!\!\!\mathbb E\Bigg[{}&\prod_{t=1}^k\prod_{s=1}^{d_t}\frac{W_{j_t}}{\sum_{\ell=1}^{m_{s,t}-1}W_\ell-k}\\
		&\times \prod_{u=1}^k\prod_{s=j_u+1}^{j_{u+1}}\bigg(1-\frac{\sum_{\ell=1}^u W_{j_\ell}}{\sum_{\ell=1}^{s-1}W_\ell}\bigg)\ind_{E_n^{(1)}}\Bigg].
		\ea \ee  
		For ease of writing, for now we only consider the inner sum until we actually intend to sum over the indices $j_1,\ldots,j_k$ later on in~\eqref{eq:intstep1}. We use the bounds from the event $E_n^{(1)}$ defined in~\eqref{eq:en} to bound 
		\be
		\sum_{\ell=1}^{m_{s,t}-1}W_\ell\geq (m_{s,t}-1)\E W(1-\zeta_n),\qquad \sum_{\ell=1}^{s-1}W_\ell \leq s\E W (1+\zeta_n).
		\ee
		For $n$ sufficiently large, we observe that $(m_{s,t}-1)\E W(1-\zeta_n)-k\geq m_{s,t}\E W(1-2\zeta_n)$, which yields the upper bound
		\be \ba 
		\frac{1}{(n)_k}\ \, \sideset{}{^*}\sum_{\substack{j_i<m_{1,i}<\ldots<m_{d_i,i}\leq n,\\i\in[k]}}\!\!\!\!\mathbb E\Bigg[\prod_{t=1}^k\prod_{s=1}^{d_t}\frac{W_{j_t}}{m_{s,t}\E W (1-2\zeta_n)} \prod_{u=1}^k\prod_{s=j_u+1}^{j_{u+1}}\!\!\bigg(1-\frac{\sum_{\ell=1}^u W_{j_\ell}}{s\E W(1+\zeta_n)}\bigg)\ind_{E_n^{(1)}}\Bigg].
		\ea \ee
		We can now bound the indicator from above by one. Moreover, relabelling the vertex-weights $W_{j_t}$ to $W_t$ for $t\in[k]$ does not change the distribution of the terms within the expected value, so that the expected value remains unchanged. We thus arrive at the upper bound
		\be \label{eq:probub}
		\frac{1}{(n)_k}\ \, \sideset{}{^*}\sum_{\substack{j_i<m_{1,i}<\ldots<m_{d_i,i}\leq n,\\i\in[k]}}\!\!\!\!\mathbb E\Bigg[\prod_{t=1}^k\prod_{s=1}^{d_t}\frac{W_t}{m_{s,t}\E W (1-2\zeta_n)} \prod_{u=1}^k\prod_{s=j_u+1}^{j_{u+1}}\!\!\bigg(1-\frac{\sum_{\ell=1}^u W_\ell}{s\E W(1+\zeta_n)}\bigg)\Bigg].
		\ee 
		We bound the final product from above by
		\be \ba \label{eq:expbound}
		\prod_{s=j_u+1}^{j_{u+1}}\bigg(1-\frac{\sum_{\ell=1}^u W_\ell}{s\E W(1+\zeta_n)}\bigg)&\leq \exp\bigg(-\frac{1}{\E W(1+\zeta_n)} \sum_{s=j_u+1}^{j_{u+1}}\frac{\sum_{\ell=1}^u W_\ell}{s}\bigg)\\
		&\leq \exp\bigg(-\frac{1}{\E W(1+\zeta_n)} \sum_{\ell=1}^u W_\ell \log\Big(\frac{j_{u+1}}{j_u+1}\Big)\bigg)\\
		&= \Big(\frac{j_{u+1}}{j_u+1}\Big)^{-\sum_{\ell=1}^u W_\ell/(\E W(1+\zeta_n))}.
		\ea\ee 
		As the weights are almost surely bounded by one, we thus find
		\be\ba
		\prod_{s=j_u+1}^{j_{u+1}}\bigg(1-\frac{\sum_{\ell=1}^u W_\ell}{s\E W(1+\zeta_n)}\bigg)&\leq \Big(\frac{j_{u+1}}{j_u}\Big)^{-\sum_{\ell=1}^u W_\ell/(\E W(1+\zeta_n))}\Big(1+\frac{1}{j_u}\Big)^{k/(\E W(1+\zeta_n))}\\
		&=\Big(\frac{j_{u+1}}{j_u}\Big)^{-\sum_{\ell=1}^u W_\ell/(\E W(1+\zeta_n))}(1+o(1)).
		\ea \ee 
		Using this upper bound in~\eqref{eq:probub} and setting 
		\be \label{eq:a}
		a_t':=\frac{W_t}{\E W(1+\zeta_n)},\qquad t\in[k],
		\ee 
		we obtain
		\be\ba \label{eq:expomit}
		\frac{1}{(n)_k}{}&\ \,\sideset{}{^*}\sum_{\substack{j_i<m_{1,i}<\ldots<m_{d_i,i}\leq n,\\i\in[k]}}\!\!\!\!\!\mathbb E\Bigg[\prod_{t=1}^k \bigg(a_t'^{d_t}\prod_{s=1}^{d_t}\frac{1+\zeta_n}{m_{s,t}(1-2\zeta_n)}\bigg)\prod_{u=1}^k\Big(\frac{j_{u+1}}{j_u}\Big)^{-\sum_{\ell=1}^u a_\ell'}\Bigg] (1+o(1))\\
		={}&\frac{1}{(n)_k}\ \,\sideset{}{^*}\sum_{\substack{j_i<m_{1,i}<\ldots<m_{d_i,i}\leq n,\\i\in[k]}}\!\!\!\Big(\frac{1+\zeta_n}{1-2\zeta_n}\Big)^{-\sum_{t=1}^k d_t}\mathbb E\Bigg[\prod_{t=1}^k\Big( a_t'^{d_t} (j_t/n)^{a_t'}\prod_{s=1}^{d_t}\frac{1}{m_{s,t}}\Big) \Bigg](1+o(1)),
		\ea \ee 		
		where in the last step we recall that $j_{k+1}=n$. Since $d_t\leq c\log n$ for all $t\in[k]$, $j_t>\ell_t>n^\eta$ for all $t\in[k]$, and $\zeta_n=n^{-\delta\eta}/\E W$, it readily follows that 
		\be \label{eq:fracs}
		\Big(\frac{1+\zeta_n}{1-2\zeta_n}\Big)^{-\sum_{t=1}^k d_t}=1+o(1),\quad \text{and}\quad a_t'^{d_t}\Big(\frac{j_t}{n}\Big)^{a_t'}=\Big(\frac{W_t}{\E W}\Big)^{d_t}\Big(\frac{j_t}{n}\Big)^{W_t/\E W}(1+o(1)).
		\ee
		We can thus omit the first term from~\eqref{eq:expomit} as well as use $a_t:=W_t/\E W$ instead of $a_t'$ at the cost of an additional $1+o(1)$ term. So, we obtain
		\be \label{eq:simplify}
		\frac{1}{(n)_k}\ \,\sideset{}{^*}\sum_{\substack{j_i<m_{1,i}<\ldots<m_{d_i,i}\leq n,\\i\in[k]}}\mathbb E\Bigg[\prod_{t=1}^k\Big( a_t^{d_t} (j_t/n)^{a_t}\prod_{s=1}^{d_t}\frac{1}{m_{s,t}}\Big) \Bigg](1+o(1))
		\ee 
		We then bound this from above even further by no longer constraining the indices $m_{s,t}$ to be distinct (so that the $*$ in the sum is omitted). That is, for different $t_1,t_2\in[k]$, we allow $m_{s_1,t_1}=m_{s_2,t_2}$ to hold for any $s_1\in[d_{t_1}],s_2\in[d_{t_2}]$. This also allows us to interchange the sum and the first product. We bound the sums from above by multiple integrals, which yields 
		\be \ba\label{eq:intbound2}
		\frac{1}{(n)_k}\E{ \prod_{t=1}^k a_t^{d_t}(j_t/n)^{a_t} \int_{j_t}^n \int _{x_{1,t}}^n\cdots \int_{x_{d_t-1,t}}^n \prod_{s=1}^{d_t}x_{s,t}^{-1}\,\d x_{d_t,t}\ldots\d x_{1,t}}(1+o(1)).
		\ea \ee 
		Applying Lemma~\ref{lemma:logints} with $a=j_t, b=n$, we then obtain
		\be 
		\frac{1}{(n)_k}\E{ \prod_{t=1}^k (n/j_t)^{-a_t} \frac{(a_t\log(n/j_t))^{d_t}}{d_t!}}(1+o(1)).
		\ee 
		Reintroducing the sums over the indices $j_1,\ldots, j_k$ (which were omitted after~\eqref{eq:tailprobmiddle}), we arrive at
		\be \label{eq:intstep1}
		\frac{1}{(n)_k} \sum_{j_1=\ell_1}^n\ldots \!\!\!\!\sum_{j_k=(\ell_k\vee j_{k-1})+1}^n\E{ \prod_{t=1}^k(n/j_t)^{-a_t}\frac{(a_t\log(n/j_t))^{d_t}}{d_t!}}(1+o(1)).
		\ee 
		We observe that switching the order of the indices $j_1,\ldots,j_k$ (and their respective bounds $\ell_1,\ldots, \ell_k$) achieves the same result as permuting the $d_1,\ldots,d_k$ and $a_1,\ldots, a_k$. Hence, if we take any $\pi\in\CP_k$, then as in~\eqref{eq:aimbound} and~\eqref{eq:enbound},
		\be \ba 
		\frac{1}{(n)_k}{}&\sum_{j_{\pi(1)}=\ell_{\pi(1)}}^n \sum_{j_{\pi(2)}=(\ell_{\pi(2)}\vee j_{\pi(1)})+1}^n\!\!\! \cdots\!\!\! \sum_{j_{\pi(k)}=(\ell_{\pi(k)}\vee j_{\pi(k-1)})+1}^n \!\!\!\!\!\!\E{\Pf{\Zm_n(j_i)=d_i, i\in[k]}\ind_{E_n^{(1)}}}\\
		&\leq \frac{1}{(n)_k}\E{\sum_{j_{\pi(1)}=\ell_{\pi(1)}}^n\!\!\!\cdots \!\!\!\!\sum_{j_{\pi(k)}=(\ell_{\pi(k)}\vee j_{\pi(k-1)})+1}^n\prod_{t=1}^k(n/j_t)^{-a_t}\frac{(a_t\log(n/j_t))^{d_t}}{d_t!}}.
		\ea \ee 
		As a result, reintroducing the sum over all $\pi\in\CP_k$, we arrive at 
		\be \ba 
		\frac{1}{(n)_k}{}&\sum_{\pi\in\CP_k}\sum_{j_{\pi(1)}=\ell_{\pi(1)}}^n \sum_{j_{\pi(2)}=(\ell_{\pi(2)}\vee j_{\pi(1)})+1}^n\!\!\! \cdots\!\!\! \sum_{j_{\pi(k)}=(\ell_{\pi(k)}\vee j_{\pi(k-1)})+1}^n\!\!\!\!\!\! \E{\Pf{\Zm_n(j_i)=d_i, i\in[k]}\ind_{E_n^{(1)}}}\\
		&\leq\frac{1}{(n)_k}\E{ \sum_{\pi\in\CP_k}\sum_{j_{\pi(1)}=\ell_{\pi(1)}}^n\!\!\!\cdots\!\!\! \!\!\!\!\sum_{j_{\pi(k)}=(\ell_{\pi(k)}\vee j_{\pi(k-1)})+1}^n\prod_{t=1}^k(n/j_t)^{-a_t}\frac{(a_t\log(n/j_t))^{d_t}}{d_t!}}(1+o(1))\\
		&=\frac{1}{(n)_k}\E{ \sum_{j_1=\ell_1+1}^n\sum_{\substack{j_2=\ell_2+1\\j_2\neq j_1}}\ldots \!\!\!\!\sum_{\substack{j_k=\ell_k+1\\ j_k\neq j_{k-1,\ldots, j_1}}}^n\prod_{t=1}^k(n/j_t)^{-a_t}\frac{(a_t\log(n/j_t))^{d_t}}{d_t!}}(1+o(1)).		
		\ea \ee 
		We now bound these sums from above by allowing each index $j_i$ to take \emph{any} value in $\{\ell_i+1,\ldots, n\}$ for all $i\in[k]$, independent of the values of the other indices. Moreover, since the weights $W_1,\ldots, W_k$, and hence $a_1, \ldots, a_k,$ are independent,  this yields the upper bound
		\be \label{eq:prodsumub}
		\prod_{t=1}^k\E{ \frac1n\sum_{j_t=\ell_t+1}^n(n/j_t)^{-a_t}\frac{(a_t\log(n/j_t))^{d_t}}{d_t!}}(1+o(1)),
		\ee 
		so that we can now deal with each sum independently instead of $k$ sums at the same time. First, we note that $(n/j_t)^{a_t}(\log(n/j_t))^{d_t}$ is increasing on $(0,n\exp(-d_t/a_t))$, maximised at $n\exp(-d_t/a_t)$, and decreasing on $(n\exp(-d_t/a_t),n]$ for all $t\in[k]$. To provide the optimal bound, we want to know whether this maximum is attained in $[\ell_t+1,n]$ or not. That is, whether $n\exp(-d_t/a_t)\in[\ell_t+1,n]$ or not. To this end, we let 
		\be \label{eq:ct}
		c_t:=\limsup_{n\to\infty}\frac{d_t}{\log n}, \qquad t\in[k],
		\ee 
		and consider two cases:
		\begin{enumerate}
			\item[\namedlabel{item:1}{$(1)$}] $c_t\in[0,1/(\theta-1)]$, $t\in[k]$.
			\item[\namedlabel{item:2}{$(2)$}] $c_t\in(1/(\theta-1),c)$, $t\in[k]$.
		\end{enumerate}
		Clearly, when $c\leq 1/(\theta-1)$ the second case can be omitted, so that without loss of generality we can assume $c>1/(\theta-1)$. In the second case, it directly follows that the maximum is almost surely attained at  
		\be
		n\exp(-d_t/a_t)\leq n\exp(-c_t\log n (\theta-1)(1+o(1)))=n^{1-c_t(\theta-1)(1+o(1))}=o(1),
		\ee
		so that the summand $(n/j_t)^{-a_t}(a_t\log(n/j_t))^{d_t}$ is almost surely decreasing in $j_t$ when $\ell_t< j_t\leq n$. In the first case, such a conclusion cannot be made in general and depends on the precise value of $W_t$. Therefore, the first case requires a more involved approach. We first assume case \ref{item:1} holds and discuss what simplifications can be made when case $\ref{item:2}$ holds afterwards. In the first case, we use Lemma~\ref{lemma:sumint} to bound each sum from above by
		\be \ba \label{eq:ubsumint1}
		\frac{1}{n} \sum_{j_t=\ell_t+1}^n (n/j_t)^{-a_t}\frac{(a_t\log(n/j_t))^{d_t}}{d_t!}\leq \frac{1}{n} \int_{\ell_t}^n (n/x_t)^{-a_t}\frac{(a_t\log(n/x_t))^{d_t}}{d_t!}\,\d x_t+\frac{1}{n}.
		\ea\ee 
		Here, we use that the summand is at most one, since
			\be \label{eq:poiprob2}
			\Big(\frac{j_t}{n}\Big)^{a_t}\frac{(a_t\log(n/j_t))^{d_t}}{d_t!}=\Pf{\text{Poi}(a_t,j_t)=d_t}\leq 1,
			\ee 
			irrespective of $a_t\in(0,\infty)$ and $j_t\in\N$ and where $\text{Poi}(a_t,j_t)$ is a Poisson random variable with rate $a_t\log(n/j_t)$, conditionally on $W_t$. In case \ref{item:2} the summand on the left-hand side is decreasing in $j_t$, so that we arrive at an upper bound without the additional error term $1/n$. Using a substitution $y_t:=\log(n/x_t)$, we obtain 
		\be \ba\label{eq:firstsumub}
		\frac{a_t^{d_t}}{(1+a_t)^{d_t+1}}{}&\int_0^{\log(n/\ell_t)}\frac{(1+a_t)^{d_t+1}}{d_t!}y_t^{d_t}\e^{-(1+a_t)y_t}\,\d y_t+\frac1n\\
		&=\frac{a_t^{d_t}}{(1+a_t)^{d_t+1}}\Pf{Y_t<\log(n/\ell_t)}+\frac1n, 
		\ea \ee 
		where, conditionally on $W_t$, $Y_t\sim \text{Gamma}(d_t+1,1+a_t)$. We recall that we redefined $a_t:=W_t/\E W=W_t/(\theta-1)$. Since $X_t:=(1+W_t/(\theta-1))Y_t\sim\text{Gamma}(d_t+1,1)$, we obtain 
		\be \label{eq:xtintro}
		\frac{\theta-1}{\theta-1+W_t}\Big(\frac{W_t}{\theta-1+W_t}\Big)^{d_t}\Pf{X_t<\Big(1+\frac{W_t}{(\theta-1)}\Big)\log(n/\ell_t)}+\frac1n.
		\ee 
		Using this in~\eqref{eq:prodsumub}, we arrive at an upper bound of the form
		\be \label{eq:ubwitherror}
		\prod_{t=1}^k \E{\frac{\theta-1}{\theta-1+W}\Big(\frac{W}{\theta-1+W}\Big)^{d_t}\Pf{X_t<\Big(1+\frac{W}{(\theta-1)}\Big)\log(n/\ell_t)}+\frac1n}(1+o(1)),
		\ee 
		where we recall that in each term of the product, the additive term $1/n$ is present only when $d_t$ satisfies case~\ref{item:1} and can be omitted when $d_t$ satisfies case~\ref{item:2}. Moreover, we have omitted the indices of the weights as they are all i.i.d. By Lemma~\ref{lemma:littleo} in the~\hyperref[sec:appendix]{Appendix}, the term $1/n$ can be included in the $o(1)$ in the square brackets when $d_t$ satisfies case~\ref{item:1}. Thus, we finally obtain
		\be\label{eq:ubwithouterror}
		\prod_{t=1}^k \E{\frac{\theta-1}{\theta-1+W}\Big(\frac{W}{\theta-1+W}\Big)^{d_t}\Pf{X_t<\Big(1+\frac{W}{(\theta-1)}\Big)\log(n/\ell_t)}}(1+o(1)),
		\ee 
		as desired. This concludes the upper bound of the first term in~\eqref{eq:enbound}. Since we can choose $\gamma$ arbitrarily large in the second term in~\eqref{eq:enbound}, we can use the same argument as in Lemma~\ref{lemma:littleo} (\eqref{eq:neglb} through~\eqref{eq:ctbound} in particular), but now using that $d_t\leq c\log n<\theta/(\theta-1)\log n$, to obtain that the second term in~\eqref{eq:enbound} can be included in the $o(1)$ term of the final expression of the upper bound as well in both case~\ref{item:1} and~\ref{item:2}, which concludes the proof of the upper bound. 
	\end{proof}
	
	We now provide a lower bound for~\eqref{eq:degdistr}, which uses many of the definitions and steps provided in the proof for the upper bound.
	
	\begin{proof}[Proof of Proposition~\ref{prop:deglocwrt}, Equation~\eqref{eq:degdistr}, lower bound]
		We define the event 
		\be \label{eq:wten}
		E_n^{(2)}:=\Big\{\sum_{\ell=k+1 }^j W_\ell\in(\E{W}(1-\zeta_n)j,\E{W}(1+\zeta_n)j),\ \forall\  n^\eta\leq j\leq n\Big\}.
		\ee
		We then  again have~\eqref{eq:aimbound} and start by considering the identity permutation, $\pi(i)=i$ for all $i\in[k]$, as in~\eqref{eq:aimub}, by omitting the second term in~\eqref{eq:enbound}, and using the event $E_n^{(2)}$ instead of $E_n^{(1)}$. This yields the lower bound
		\be \ba 
		\frac{1}{(n)_k}{}&\sum_{j_1=\ell_1+1}^n\ldots \!\!\!\!\sum_{j_k=(\ell_k\vee j_{k-1})+1}^n\!\!\!\!\!\!\!\!\mathbb E[\Pf{\Zm_n(j_\ell)=m_\ell,\ell\in[k]}\ind_{E_n^{(2)}}]\\
		\geq{}&\frac{1}{(n)_k}\sum_{j_1=\ell_1+1}^n\ldots \!\!\!\!\sum_{j_k=(\ell_k\vee j_{k-1})+1}^n\ \, \sideset{}{^*}\sum_{\substack{j_i<m_{1,i}<\ldots<m_{d_i,i}\leq n,\\i\in[k]}}\mathbb E\Bigg[\prod_{t=1}^k\prod_{s=1}^{d_t}\frac{W_{j_t}}{\sum_{\ell=1}^{m_{s,t}-1}W_\ell}\\ &\times \prod_{u=1}^k\!\!\prod_{\substack{s=j_u+1\\s\neq m_{i,t},t\in[d_i],i\in[k]}}^{j_{u+1}}\!\!\!\!\bigg(1-\frac{\sum_{\ell=1}^u W_{j_\ell}}{\sum_{\ell=1}^{s-1}W_{\ell}}\bigg)\ind_{E_n^{(2)}}\Bigg].
		\ea \ee  
		We omit the constraint $s\neq m_{\ell,i},\ell\in[d_i],i\in[k]$ in the final product. As this introduces more multiplicative terms smaller than one, we obtain a lower bound. Then, in the two denominators, we bound the vertex-weights $W_{j_1},\ldots, W_{j_k}$ from above by one and below by zero, respectively, to obtain a lower bound \be\ba 
		\frac{1}{(n)_k}\sum_{j_1=\ell_1+1}^n\ldots \!\!\!\!\sum_{j_k=(\ell_k\vee j_{k-1})+1}^n\ \, \sideset{}{^*}\sum_{\substack{j_i<m_{1,i}<\ldots<m_{d_i,i}\leq n,\\i\in[k]}}\!\!\!\!\mathbb E\Bigg[{}&\prod_{t=1}^k\prod_{s=1}^{d_t}\frac{W_{j_t}}{\sum_{\ell=1}^{m_{s,t}-1}W_\ell\ind_{\{\ell\neq j_t,t\in[k]\}}+k}\\
		&\times \prod_{u=1}^k\prod_{s=j_u+1}^{j_{u+1}}\!\!\!\bigg(1-\frac{\sum_{\ell=1}^u W_{j_\ell}}{\sum_{\ell=1}^{s-1}W_\ell\ind_{\{\ell\neq j_t,t\in[k]\}}}\bigg)\ind_{E_n^{(2)}}\Bigg].
		\ea\ee 
		As a result, we can now swap the labels of $W_{j_t}$ and $W_t$ for each $t\in[k]$, which again does not change the expected value, but it changes the value of the two denominators to $\sum_{\ell=k+1}^{m_{s,t}}W_\ell+k$ and $\sum_{\ell=k+1}^{m_{s,t}}W_\ell$, respectively. After this we use the bounds in $E_n^{(2)}$ on these sums in the expected value to obtain a lower bound. Finally, we note that the (relabelled) weights $W_t,t\in[k],$ are independent of $E_n^{(2)}$ so that we can take the indicator out of the expected value. Combining all of the above steps, we arrive at the lower bound
		\be \ba \label{eq:lbas}
		\frac{1}{(n)_k}\sum_{j_1=\ell_1+1}^n\ldots \!\!\!\!\sum_{j_k=(\ell_k\vee j_{k-1})+1}^n\ \,{}& \sideset{}{^*}\sum_{\substack{j_i<m_{1,i}<\ldots<m_{d_i,i}\leq n,\\i\in[k]}}\mathbb E\Bigg[\prod_{t=1}^k \Big(\frac{W_t}{\E{W}}\Big)^{d_t} \prod_{s=1}^{d_t}\frac{1}{m_{s,t}(1+2\zeta_n)}\\ &\times\prod_{u=1}^k\prod_{s=j_u+1}^{j_{u+1}}\bigg(1-\frac{\sum_{\ell=1}^u W_\ell}{(s-1)\E W(1-\zeta_n)}\bigg)\Bigg]\mathbb P(E_n^{(2)}).\\
		\ea \ee 
		The $1+2\zeta_n$ in the fraction on the first line arises from the fact that, for $n$ sufficiently large, $(m_{s,t}-1)(1+\zeta_n)+k\leq m_{s,t}(1+2\zeta_n)$. It follows from Lemma~\ref{lemma:weightsumbounds} that $\mathbb P(E_n^{(2)})=1-o(n^{-\gamma})$ for any $\gamma>0$. Similar to the calculations in \eqref{eq:expbound} and using $\log(1-x)\geq -x-x^2$ for $x$ small, we obtain an almost sure lower bound for the final product for $n$ sufficiently large of the form 
		\be
		\prod_{s=j_u+1}^{j_{u+1}}\bigg(1-\frac{\sum_{\ell=1}^u W_\ell}{(s-1)\E W(1-\zeta_n)}\bigg)\geq \Big(\frac{j_{u+1}}{j_u}\Big)^{-\sum_{\ell=1}^u W_\ell/(\E W(1-\zeta_n))}(1-o(1)).
		\ee 
		Using this in~\eqref{eq:lbas} yields the lower bound
		\be \label{eq:lbstep}
		\frac{1}{(n)_k}\!\sum_{j_1=\ell_1+1}^n\ldots \!\!\!\!\!\!\sum_{j_k=(\ell_k\vee j_{k-1})+1}^n\ \, \sideset{}{^*}\sum_{\substack{j_i<m_{1,i}<\ldots<m_{d_i,i}\leq n,\\i\in[k]}}\!\!\!\!\!\!\!\!\!\!\!\!\!(1-o(1))\Big(\frac{1-\zeta_n}{1+2\zeta_n}\Big)^{\sum_{t=1}^k d_t} \mathbb E\Bigg[\prod_{t=1}^k \wt a_t^{d_t} \Big(\frac{j_t}{n}\Big)^{\wt a_t}\prod_{s=1}^{d_t}\frac{1}{m_{s,t}} \Bigg],
		\ee  
		where $\wt a_t:=W_t/(\E W(1-\zeta_n))$. Since $d_t\leq c\log n$ and $j_t\geq \ell_t\geq n^\eta$ for all $t\in[k]$, and $\zeta_n=n^{-\eta\delta}/\E W$ for some $\delta\in(0,1/2)$, we have as in~\eqref{eq:fracs}, that 
		\be 
		\Big(\frac{1-\zeta_n}{1+2\zeta_n}\Big)^{\sum_{t=1}^k d_t}=1-o(1), \quad \text{and}\quad \wt a_t^{d_t}\Big(\frac{j_t}{n}\Big)^{\wt a_t}=a_t^{d_t}\Big(\frac{j_t}{n}\Big)^{a_t}(1-o(1)), 
		\ee 
		where $a_t:=W_t/\E W$.	This yields
		\be 
		\frac{1}{(n)_k}\!\sum_{j_1=\ell_1+1}^n\ldots \!\!\!\!\!\!\sum_{j_k=(\ell_k\vee j_{k-1})+1}^n\ \, \sideset{}{^*}\sum_{\substack{j_i<m_{1,i}<\ldots<m_{d_i,i}\leq n,\\i\in[k]}}\!\!\!\!\! \mathbb E\Bigg[\prod_{t=1}^k a_t^{d_t} \Big(\frac{j_t}{n}\Big)^{ a_t}\prod_{s=1}^{d_t}\frac{1}{m_{s,t}} \Bigg](1-o(1)).
		\ee 
		We now bound the sum over the  indices $m_{s,i}$ from below. We note that the expression in the expected value is  decreasing in $m_{s,i}$ and we restrict the range of the indices to $j_i+\sum_{t=1}^k d_t<m_{1,i}<\ldots< i_{d_i,i}\leq n$ for all $i \in[k]$, but no longer constrain the indices to be distinct (so that we can drop the $*$ in the sum). In the distinct sums and the suggested lower bound, the number of values the $m_{s,i}$ take on equal
		\be 
		\prod_{i=1}^k \binom{n-(j_i-1)-\sum_{t=1}^{i-1}d_t}{d_i} \quad \text{and} \quad
		\prod_{i=1}^k \binom{n-(j_i-1)-\sum_{t=1}^k d_t}{d_i},
		\ee 
		respectively. It is straightforward to see that the former allows for more possibilities than the latter, as $\binom{b}{c}> \binom{a}{c}$ when $b> a\geq c$. As we omit the largest values of the expected value (since it decreases in $m_{s,t}$ and we omit the smallest values of $m_{s,t}$), we thus arrive at the lower bound
		\be \ba \label{eq:lbstep2}
		\frac{1}{(n)_k} \!\!\!\sum_{j_1=\ell_1+1}^{n-\sum_{t=1}^k d_t}\!\!\!\!\ldots \!\!\!\!\!\!\sum_{j_k=(\ell_k\vee j_{k-1})+1}^{n-\sum_{t=1}^k d_t}\ \, \sideset{}{^*}\sum_{\substack{j_i+\sum_{t=1}^k d_t<m_{1,i}<\ldots<m_{d_i,i}\leq n,\\i\in[k]}}\!\!\!\!\mathbb E\Bigg[\prod_{t=1}^k a_t^{d_t} \Big(\frac{j_t}{n}\Big)^{ a_t}\prod_{s=1}^{d_t}\frac{1}{m_{s,t}} \bigg](1-o(1)),
		\ea \ee 
		where we also restrict the upper range of the indices of the outer sums, as otherwise there would be a contribution of zero from these values of $j_1,\ldots,j_k$. We now use similar techniques compared to the upper bound of the proof to switch from summation to integration. However, due to the altered bounds on the range of the indices over which we sum and the fact that we require lower bounds rather than upper bound, we face some more technicalities. 
		
		For now, we omit the expected value and focus on the terms
		\be \label{eq:lbsum}
		\frac{1}{(n)_k} \!\!\!\sum_{j_1=\ell_1+1}^{n-\sum_{t=1}^k d_t}\!\!\!\!\ldots \!\!\!\!\!\!\sum_{j_k=(\ell_k\vee j_{k-1})+1}^{n-\sum_{t=1}^k d_t} \sum_{\substack{j_i+\sum_{t=1}^k d_t<m_{1,i}<\ldots<m_{d_i,i}\leq n,\\i\in[k]}}\prod_{t=1}^k  a_t^{d_t}\Big(\frac{j_t}{n}\Big)^{ a_t}\prod_{s=1}^{d_t}\frac{1}{m_{s,t}}.
		\ee 
		We start by restricting the upper bound on the $k$ outer sums to $n-2\sum_{i=1}^k d_i$. This will prove useful later. We set $h_k:=\sum_{t=1}^k d_t$ and bound the inner sum over the indices $m_{s,t}$ from below by 
		\be\ba 
		{}&\sum_{\substack{j_i+h_k<m_{1,i}<\ldots<m_{d_i,i}\leq n,\\i\in[k]}}\prod_{t=1}^k \prod_{s=1}^{d_t}\frac{1}{m_{s,t}} \geq \prod_{t=1}^k \int_{j_t+1+h_k}^{n}\int_{x_{1,t}+1}^{n}\cdots \int_{x_{d_t-1,t}+1}^{n}\prod_{s=1}^{d_t}x_{s,t}^{-1}\,\d x_{d_t,t}\ldots \d x_{1,t}.
		\ea \ee 
		Applying Lemma~\ref{lemma:logints} with $a=j_t+1+h_k$ and $b=n$, and using that $j_t\leq n-2h_k $ (recall that we restricted the upper bound on the outer sums in~\eqref{eq:lbsum} to $n-2h_k$), yields the lower bound
		\be 
		\prod_{t=1}^k \frac{ a_t^{d_t}}{d_t!}\Big(\log\Big(\frac{n}{j_t +2\sum_{i=1}^k d_i}\Big)\Big)^{d_t}.
		\ee 
		Substituting this in \eqref{eq:lbsum} with the restriction on the outer sum discussed after~\eqref{eq:lbsum} yields
		\be\label{eq:finalsum}
		\frac{1}{(n)_k} \!\!\!\sum_{j_1=\ell_1+1}^{n-2\sum_{i=1}^k d_i}\!\!\!\!\ldots \!\!\!\!\!\!\sum_{j_k=(\ell_k\vee j_{k-1})+1}^{n-2\sum_{i=1}^k d_i}\prod_{t=1}^k \Big(\frac{j_t}{n}\Big)^{ a_t}\frac{ a_t^{d_t}}{d_t!}\Big(\log\Big(\frac{n}{j_t +2\sum_{i=1}^k d_i}\Big)\Big)^{d_t} .
		\ee 
		To simplify the summation over $j_1,\ldots,j_k$, we write the summand as 
		\be
		\prod_{t=1}^k\Big(\Big(j_t+2\sum_{i=1}^k d_i\Big)/n\Big)^{ a_t}\frac{ a_t^{d_t}}{d_t!}\Big(\log\Big(\frac{n}{j_t +2\sum_{i=1}^k d_i}\Big)\Big)^{d_t}\bigg(1-\frac{2\sum_{i=1}^k d_i}{j_t +2\sum_{i=1}^k d_i}\bigg)^{ a_t}.
		\ee
		Using that $d_t\leq c \log n,j_t\geq \ell_t\geq n^\eta$ and $x^{ a_t}\geq x^{1/\E W}$ for $x\in(0,1)$ almost surely, we can write the last term as $(1-o(1))$ almost surely. We then shift the bounds on the range of the sums in~\eqref{eq:finalsum} by $2\sum_{i=1}^k d_i$ and let $\wt \ell_i:=\ell_i+2\sum_{t=1}^k d_t$ for all $ i\in[k]$, to obtain the lower bound
		\be \label{eq:sumtointstart}
		\frac{1}{(n)_k} \sum_{j_1=\wt \ell_1+1}^n\sum_{j_2=(\wt \ell_2\vee j_1)+1}^n\!\!\!\ldots \!\!\!\sum_{j_k=(\wt \ell_k\vee j_{k-1})+1}^n\!\!\!\!\!\!\!\!\!\!\!\!(1-o(1))\prod_{t=1}^k \Big(\frac{j_t}{n}\Big)^{ a_t}\frac{ 1}{d_t!}(a_t\log(n/j_t))^{d_t}.
		\ee 
		We recall that this lower bound is achieved for the permutation $\pi$ such that $\pi(i)=i$ for all $i\in[k]$. As the product is invariant to permuting the indices $t\in[k]$, we can use this in~\eqref{eq:aimbound} to obtain 
			\be\ba \label{eq:permsum}
			\frac{1}{(n)_k}{}&\sum_{\pi\in\CP_k}\sum_{j_{\pi(1)}=\ell_{\pi(1)}}^n \sum_{j_{\pi(2)}=(\ell_{\pi(2)}\vee j_{\pi(1)})+1}^n \cdots \sum_{j_{\pi(k)}=(\ell_{\pi(k)}\vee j_{\pi(k-1)})+1}^n \P{\Zm_n(j_i)=d_i, i\in[k]}\\
			\geq{}&\frac{1}{(n)_k}\sum_{j_1=\wt \ell_1+1}^n \sum_{\substack{j_2=\wt \ell_2+1\\j_2\neq j_1}}^n \cdots \sum_{\substack{j_k=\wt \ell_k+1\\ j_k\neq j_1,\ldots, j_{k-1}}}^n (1-o(1))\prod_{t=1}^k \E{\Big(\frac{j_t}{n}\Big)^{ a_t}\frac{ 1}{d_t!}(a_t\log(n/j_t))^{d_t}}.
			\ea\ee 
			We now want to allow for the indices $j_1,\ldots, j_k$ to have the same value. This way, we can more easily evaluate the sums. To do this, we distinguish between two cases in terms of the sizes of $d_1,\ldots, d_k$, namely case~\ref{item:1} and case~\ref{item:2}. In case~\ref{item:1}, we subtract all terms where two or more indices have the same value to avoid creating an upper bound. That is, we write the multiple sums as 
			\be \ba 
			\frac{1}{(n)_k}{}&\sum_{j_1=\wt \ell_1+1}^n \sum_{j_2=\wt \ell_2+1}^n \cdots \sum_{j_k=\wt \ell_k+1}^n (1-o(1))\prod_{t=1}^k \E{\Big(\frac{j_t}{n}\Big)^{ a_t}\frac{ 1}{d_t!}(a_t\log(n/j_t))^{d_t}}\\
			-\frac{1}{(n)_k}{}&\sum_{m=2}^k \sum_{\substack{S\subseteq [k]\\ |S|=m}}\sum_{i\in S}\sum_{j_i=\wt \ell_i+1}^n\  \sideset{}{^*}\sum_{\substack{j_s=\wt \ell_s+1\\ s\in [k]\backslash S}}^n\!\! \E{ \prod_{u\in S}\Big(\frac{j_i}{n}\Big)^{a_u}\frac{(a_u\log(n/j_i))^{d_u}}{d_u!} \!\!\prod_{s\in [k]\backslash S}\!\!\!\Big(\frac{j_s}{n}\Big)^{d_s}\frac{(a_s\log(n/j_s))^{d_s}}{d_s!}}.
			\ea \ee 
			Here, the $*$ in the final sum on the second line indicates that the indices $j_s$ with $s\in [k]\backslash S$ are not allowed to have the same value, nor be equal to $j_i$ for any $i\in S$. The error term on the second line can be bounded from below by bounding the multiple sums from above, which follows an equivalent approach as the proof of the upper bound. By~\eqref{eq:poiprob2} we can omit all terms $u\neq i$ in the product over $u\in S$, as they can be bounded from above by one. Furthermore, we can omit the $*$ in the final sum to obtain an upper bound, so that all indices $j_i$ and $j_s,s\in[k]\backslash S$ can be equal in value. Finally, let us write $S_i:=S\backslash \{i\}$. It then follows from~\eqref{eq:prodsumub} through~\eqref{eq:ubwithouterror} that the error term is at least 
			\be \ba
			-{}&\sum_{m=2}^k \!\frac{1+o(1)}{n^{m-1}}\! \sum_{\substack{S\subseteq [k]\\ |S|=m}}\sum_{i\in S} \prod_{t\in[k]\backslash S_i}\!\!\!\!\E{\frac{\theta-1}{\theta-1+W}\Big(\frac{W}{\theta-1+W}\Big)^{d_t}\Pf{X_t\leq \Big(1+\frac{W}{\theta-1}\Big)\log(n/\wt \ell_t)}}\\
			&\geq -C\sum_{m=2}^k \frac{1}{n^{m-1}}\!\!\!\!\!\!\!\!\!\! \sum_{\substack{S'\subseteq [k]\\ |S|=k-(m-1)}}\!\!\!\!\!\!\!\prod_{t\in S'}\!\E{\frac{\theta-1}{\theta-1+W}\Big(\frac{W}{\theta-1+W}\Big)^{d_t}\Pf{X_t\leq \Big(1+\frac{W}{\theta-1}\Big)\log(n/\wt \ell_t)}},
			\ea \ee 
			for some large constant $C>0$. It remains to take care of the main term,
			\be \ba \label{eq:mainterm}
			\frac{1}{(n)_k}{}&\sum_{j_1=\wt \ell_1+1}^n \sum_{j_2=\wt \ell_2+1}^n \cdots \sum_{j_k=\wt \ell_k+1}^n (1-o(1))\prod_{t=1}^k \E{\Big(\frac{j_t}{n}\Big)^{ a_t}\frac{ 1}{d_t!}(a_t\log(n/j_t))^{d_t}}\\
			&\geq \prod_{t=1}^k\E{\frac1n \sum_{j_t=\wt \ell_t+1}^n\Big(\frac{j_t}{n}\Big)^{ a_t}\frac{ 1}{d_t!}(a_t\log(n/j_t))^{d_t}}(1-o(1)).
			\ea \ee 
			We bound each sum from below by an integral, similar to the proof of the upper bound. We again consider the two cases used in the upper bound, case~\ref{item:1} and case~\ref{item:2}. In case~\ref{item:2}, the summand is decreasing in $j_t$ and hence we can replace the sum by an integral from $\ell_t$ to $n$. In case~\ref{item:1}, we use Lemma~\ref{lemma:sumint} and~\eqref{eq:poiprob2} to obtain the lower bound
			\be 
			\frac1n\sum_{j_t=\wt \ell_t+1}^n\Big(\frac{j_t}{n}\Big)^{ a_t}\frac{ 1}{d_t!}(a_t\log(n/j_t))^{d_t}\geq \int_{\wt \ell_t}^n \Big(\frac{x_t}{n}\Big)^{ a_t}\frac{ 1}{d_t!}(a_t\log(n/x_t))^{d_t}\,\d x_t -\frac1n.
			\ee 
			The same steps as in~\eqref{eq:firstsumub} and~\eqref{eq:xtintro} yield that this equals
			\be 
			\frac{\theta-1}{\theta-1+W_t}\Big(\frac{W_t}{\theta-1+W_t}\Big)^{d_t}\Pf{X_t<\Big(1+\frac{W_t}{(\theta-1)}\Big)\log(n/\wt \ell_t)}-\frac1n.
			\ee 
			Using this in~\eqref{eq:mainterm} and combining it with the bound for the error term, we arrive at the final lower bound
			\be \ba 
			\prod_{t=1}^k{}& \E{\frac{\theta-1}{\theta-1+W_t}\Big(\frac{W_t}{\theta-1+W_t}\Big)^{d_t}\Pf{X_t<\Big(1+\frac{W_t}{(\theta-1)}\Big)\log(n/\wt \ell_t)}-\frac1n}(1-o(1))\\
			&-C\sum_{m=2}^k \frac{1}{n^{m-1}}\!\!\!\!\!\!\!\!\!\! \sum_{\substack{S'\subseteq [k]\\ |S|=k-(m-1)}}\!\!\!\!\!\!\!\prod_{t\in S'}\!\E{\frac{\theta-1}{\theta-1+W}\Big(\frac{W}{\theta-1+W}\Big)^{d_t}\Pf{X_t\leq \Big(1+\frac{W}{\theta-1}\Big)\log(n/\wt \ell_t)}}.
			\ea \ee 
			We can replace $\wt \ell_t$ with $\ell_t$ at the cost of a $1-o(1)$ term, since $\log(n/\wt \ell_t)=\log(n/\ell_t)-o(1)$. It then follows from Lemma~\ref{lemma:littleo} that both the $1/n$ term on the first line as well as the second line can be incorporated into the $1-o(1)$ term.
			
			In case~\ref{item:2}, we know that the summand in~\eqref{eq:permsum} is decreasing in $j_t$ for all $t\in[k]$. Hence, we can omit the smallest values of $j_1, \ldots, j_k$ to obtain a lower bound. This yields 
			\be 
			\frac{1}{(n)_k}{}\sum_{j_1=\wt \ell_1+1}^n \sum_{j_2=\wt \ell_2+2}^n \cdots \sum_{j_k=\wt \ell_k+k}^n (1-o(1))\prod_{t=1}^k \E{\Big(\frac{j_t}{n}\Big)^{ a_t}\frac{ 1}{d_t!}(a_t\log(n/j_t))^{d_t}},
			\ee 
			which can be evaluated in the same manner as in case~\ref{item:1} to yield the lower bound
			\be
			\prod_{t=1}^k \E{\frac{\theta-1}{\theta-1+W_t}\Big(\frac{W_t}{\theta-1+W_t}\Big)^{d_t}\Pf{X_t<\Big(1+\frac{W_t}{(\theta-1)}\Big)\log(n/(\wt \ell_t+(t-1)))}}(1-o(1)).
			\ee 
			Again, since $\log(n/(\wt \ell_t+(t-1)))=\log(n/\ell_t)-o(1)$ for each $t\in[k]$, we can replace $\wt \ell_t+(t-1)$ with $\ell_t$ for each $t\in[k]$ at the cost of a $1-o(1)$ term. We thus conclude that 
		\be \ba 
		\mathbb P(\Zm_n{}&(v_i)= d_i, v_i>\ell_i,i\in[k])\\
		\geq{}& (1-o(1))\prod_{t=1}^k \Bigg[\E{\frac{\theta-1}{\theta-1+W}\Big(\frac{W}{\theta-1+W}\Big)^{d_t}\Pf{X_t<\Big(1+\frac{W}{\theta-1}\Big)\log(n/\ell_t)}},
		\ea \ee 
		which concludes the proof of the lower bound.
	\end{proof}
	
	We observe that the combination of the upper and lower bound proves~\eqref{eq:degdistr}. What remains is to prove~\eqref{eq:degdistrtail}.
	
	\begin{proof}[Proof of Proposition~\ref{prop:deglocwrt}, Equation~\eqref{eq:degdistrtail}]		
		We prove the two bounds in~\eqref{eq:degdistrtail} by using~\eqref{eq:degdistr}. We assume that $d_i$ diverges with $n$ and we note that, if 
		\be 
		d_i\leq c\log n\quad \text{and}\quad \ell_i\leq n\exp(-(1-\xi)(1-\theta^{-1})(d_i+1)),
		\ee 
		for any $\xi\in(0,1)$ and for all sufficiently large $n$, then for any $j\in[\lfloor d_i^{1/4}\rfloor]$, it also holds that 
		\be 
		d_i +j \leq c'\log n,\quad \text{and}\quad  \ell_i\leq n\exp(-(1-\xi)(1-\theta^{-1})(d_i+j+1)),
		\ee 
		for any $\xi\in(0,1)$ and for all sufficiently large $n$ as well, where we can choose $c'\in(c,\theta/(\theta-1))$ arbitrarily close to $c$. As a result, we can write 
		\be \ba
		\mathbb P(\Zm_n{}&(v_i)\geq d_i, v_i> \ell_i, i\in[k])\\
		\leq & \sum_{j_1=d_1}^{d_1+\lfloor d_1^{1/4}\rfloor}\cdots \sum_{j_k=d_k}^{d_k+\lfloor d_k^{1/4}\rfloor}\P{\Zm_n(v_i)=j_i, v_i> \ell_i, i\in[k]}\\
		&+\sum_{t=1}^k \P{\Zm_n(v_t)\geq d_t+\lceil d_t^{1/4}\rceil ,\Zm_n(v_i)\geq d_i, i\neq t, v_i>\ell_i, i\in[k]}. 
		\ea\ee 
		We first provide an upper bound for the multiple sums on the first line. By~\eqref{eq:degdistr}, this equals
		\be
		\sum_{j_1=d_1}^{d_1+\lfloor d_1^{1/4}\rfloor}\!\!\!\cdots\!\!\! \sum_{j_k=d_k}^{d_k+\lfloor d_k^{1/4}\rfloor}\!\!\!\!(1+o(1))\prod_{i=1}^k \E{\frac{\theta-1}{\theta-1+W}\Big(\frac{W}{\theta-1+W}\Big)^{j_i}\Pf{X_{j_i}<\Big(1+\frac{W}{\theta-1}\Big)\log(n/\ell_i)}},
		\ee
		where we write $X_{j_i}\sim\text{Gamma}(j_i+1,1)$ instead of $X_i$ to explicitly state the dependence on $j_i$. If $X_{j_i}\sim\text{Gamma}(j_i+1,1),X_{j_i'}\sim\text{Gamma}(j_i'+1,1)$, then $X_{j_i}$ stochastically dominates $X_{j_i'}$ when $j_i>j_i'$. Hence, we obtain the upper bound
		\be\ba\label{eq:finub}
		\sum_{j_1=d_1}^\infty\!\!\!{}&\ldots\!\!\!  \sum_{j_k=d_k}^\infty\!(1+o(1))\prod_{i=1}^k \mathbb E\bigg[\frac{\theta-1}{\theta-1+W}\Big(\frac{W}{\theta-1+W}\Big)^{j_i}\!\mathbb P_W\!\bigg(\!\!X_{d_i}\!<\Big(1+\frac{W}{\theta-1}\Big)\log(n/\ell_i)\!\bigg)\bigg]\\
		&=(1+o(1))\prod_{i=1}^k \E{\Big(\frac{W}{\theta-1+W}\Big)^{d_i}\Pf{X_i<\Big(1+\frac{W}{\theta-1}\Big)\log(n/\ell_i)}},
		\ea \ee 
		where we note that $X_i\equiv X_{d_i}$ by the definition of $X_i$ and $X_{d_i}$. It thus remains to show that
		\be \label{eq:negterm}
		\sum_{t=1}^k \P{\Zm_n(v_t)\geq d_t+\lceil d_t^{1/4}\rceil ,\Zm_n(v_i)\geq d_i, i\neq t, v_i>\ell_i, i\in[k]} 
		\ee 
		is negligible compared to~\eqref{eq:finub}. We show this holds for each term in the sum, and since all $d_i,i\in[k]$ diverge, it suffices to show this holds for $t=1$. The in-degrees in the WRT model are negative quadrant dependent under the conditional probability measure $\mathbb P_W$. That is, by~\cite[Lemma $7.1$]{LodOrt21}, for any indices $r_1, \ldots, r_k\in[n]$, $r_i\neq r_j$ when $i\neq j$, 
		\be 
		\Pf{\Zm_n(r_i)\geq d_i, i\in[k]}\leq \prod_{i=1}^k \Pf{\Zm_n(r_i)\geq d_i}.
		\ee 
		We can thus bound the term with $t=1$ in~\eqref{eq:negterm} from above by 
		\be \ba
		\sum_{j_1=\ell_1+1}^n{}&\sum_{\substack{j_2=\ell_2+1\\ j_2\neq j_1}}^n \cdots \sum_{\substack{j_k=\ell_k+1\\j_k\neq j_{k-1},\ldots, j_1}}^n \E{\Pf{\Zm_n(j_1)\geq d_1+\lceil d_1^{1/4}\rceil}\prod_{i=2}^k \Pf{\Zm_n(j_i)\geq d_i}} \\
		&\leq \E{\Pf{\Zm_n(v_1)\geq d_1+\lceil d_1^{1/4}\rceil, v_1>\ell_1}\prod_{i=2}^k \Pf{\Zm_n(v_i)\geq d_i, v_i>\ell_i}}, 
		\ea \ee 
		where the last step follows by allowing the indices $j_i$ to take on any value between $\ell_i+1$ and $n$, $i\in[k]$. We can now deal with each of these probabilities individually instead of with all the events at the same time, which makes obtaining an explicit bound for the probability of the event $\{\Zm_n(v_i)\geq d_i, v_i>\ell_i\}$ easier. We claim that, with a very similar approach compared to the proof of the upper bound for~\eqref{eq:degdistr} (see also steps $(5.47)$ through $(5.51)$ in the proof of~\cite[Lemma $5.11$]{EslLodOrt21} for the case $\ell_1=\ldots \ell_k=n^{1-\eps}$ for some $\eps\in(0,1)$), it can be shown that this expected value is bounded from above by 
		\be \ba
		(1+o(1)){}&\E{\Big(\frac{W}{\theta-1+W}\Big)^{d_1+\lceil d_1^{1/4}\rceil}\Pf{X_1\leq \Big(1+\frac{W}{\theta-1}\Big)\log(n/\ell_1)}}\\
		&\times \prod_{i=2}^k \E{\Big(\frac{W}{\theta-1+W}\Big)^{d_i}\Pf{X_i\leq \Big(1+\frac{W}{\theta-1}\Big)\log(n/\ell_i)}}\\
		\leq{}& (1+o(1))\theta^{-\lceil d_1^{1/4}\rceil}\prod_{i=1}^k \E{\Big(\frac{W}{\theta-1+W}\Big)^{d_i}\Pf{X_i\leq \Big(1+\frac{W}{\theta-1}\Big)\log(n/\ell_i)}}.
		\ea\ee 
		This upper bound  can be achieved for each term in~\eqref{eq:negterm} (with $\lceil d_1^{1/4}\rceil$ changed accordingly), so that~\eqref{eq:negterm} is indeed negligible compared to~\eqref{eq:finub} and hence can be included in the $o(1)$ term in~\eqref{eq:finub}. This proves the upper bound in~\eqref{eq:degdistrtail}.
		
		For a lower bound we directly obtain 
		\be 
		\P{\Zm_n(v_i)\geq d_i, v_i> \ell_i, i\in[k]}\geq \sum_{j_1=d_1}^{d_1+\lfloor d_1^{1/4}\rfloor}\cdots \sum_{j_k=d_k}^{d_k+\lfloor d_k^{1/4}\rfloor}\P{\Zm_n(v_i)=j_i, v_i> \ell_i, i\in[k]}.
		\ee 
		With a similar approach as for the upper bound we can use~\eqref{eq:degdistr} and now bound the probability from below by replacing $X_{j_i}$ with $\wt X_i\equiv X_{d_i+\lfloor d_i^{1/4}\rfloor}$ instead of $X_{d_i}$, to arrive at the lower bound  
		\be \ba
		\sum_{j_1=d_1}^{d_1+\lfloor d_1^{1/4}\rfloor}\!\!\!{}&\cdots\!\!\! \sum_{j_k=d_k}^{d_k+\lfloor d_k^{1/4}\rfloor}\!\!\!\!(1+o(1))\prod_{i=1}^k \E{\frac{\theta-1}{\theta-1+W}\Big(\frac{W}{\theta-1+W}\Big)^{j_i}\Pf{ X_{j_i}<\Big(1+\frac{W}{\theta-1}\Big)\log(n/\ell_i)}}\\
		\geq (1+{}&o(1))\prod_{i=1}^k \E{\Big(\frac{W}{\theta-1+W}\Big)^{d_i}\Big(1-\Big(\frac{W}{\theta-1+W}\Big)^{\lfloor d_i^{1/4}\rfloor}\Big)\Pf{\wt X_i<\Big(1+\frac{W}{\theta-1}\Big)\log(n/\ell_i)}}\\
		\geq(1+{}&o(1))\prod_{i=1}^k\E{\Big(\frac{W}{\theta-1+W}\Big)^{d_i}\Pf{\wt X_i<\Big(1+\frac{W}{\theta-1}\Big)\log(n/\ell_i)}},
		\ea \ee
		where in the last step we use that $1-(W/(\theta-1+W))^{\lfloor d_i^{1/4}\rfloor}\geq 1-\theta^{-\lfloor d_i^{1/4}\rfloor}=1-o(1)$ almost surely, since $d_i$ diverges for any $i\in[k]$. This concludes the proof of the lower bound in~\eqref{eq:degdistrtail} and hence of Proposition~\ref{prop:deglocwrt}.
	\end{proof}	
	
	\section{Extended results for the~\ref{ass:beta} and~\ref{ass:gamma} cases}\label{sec:examples}
	
	In this section we discuss two examples of vertex-weight distributions as provided in Assumption~\ref{ass:weights}, for which results similar to those of Theorems~\ref{thrm:conddegloc},~\ref{thrm:deglocwrt} and Proposition~\ref{prop:momentconvwrt} (where the latter two hold for the~\ref{ass:weightatom} case) can be proved.
	
	\begin{example}[\ref{ass:beta} case]
		We consider a random variable $W$ with a beta distribution, i.e.\ with a tail distribution as in~\eqref{eq:betacdf} for some $\alpha,\beta>0$. We define, for $j\in\Z,B\in\CB(\R)$,
		\be\ba\label{eq:epsnbeta}
		\wt X^{(n)}_j(B)&:=\Big|\Big\{\inn: \zni=\lfloor \log_\theta n-\beta\log_\theta\log_\theta n \rfloor +j, \frac{\log i-\mu\log n}{\sqrt{(1-\sigma^2)\log n}}\in B\Big\}\Big|,\\ 
		\wt X^{(n)}_{\geq j}(B)&:=\Big|\Big\{\inn: \zni\geq \lfloor \log_\theta n-\beta\log_\theta\log_\theta n \rfloor +j, \frac{\log i-\mu\log n}{\sqrt{(1-\sigma^2)\log n}}\in B\Big\}\Big|,\\
		\eps_n&:=(\log_\theta n-\beta \log_\theta\log_\theta n)-\lfloor\log_\theta n-\beta\log_\theta\log_\theta n \rfloor,\\
		c_{\alpha,\beta,\theta}&:=\frac{\Gamma(\alpha+\beta)}{\Gamma(\alpha)}(1-\theta^{-1})^{-\beta}.
		\ea\ee 
		Then, we can formulate the following results.
		
		\begin{theorem}\label{thrm:betappp}
			Consider the WRT model, that is, the WRG model as in Definition~\ref{def:wrg} with $m=1$, with vertex-weights $(W_i)_{i\in\N}$ which are distributed according to~\eqref{eq:betacdf} for some $\alpha,\beta>0$, and recall $\theta=1+\E W$. Let $v^1,v^2,\ldots, v^n$ be the vertices in the tree in decreasing order of their in-degree $($\!where ties are split uniformly at random$)$, let $d_n^i$ and $\ell_n^i$ denote the in-degree and label of $v^i$, respectively, and fix $\eps\in[0,1]$. Recall $\eps_n$ from~\eqref{eq:epsnbeta} and let $(n_j)_{j\in\N}$ be a positive, diverging, integer sequence such that $\eps_{n_j}\to\eps$ as $j\to\infty$. Finally, let $(P_i)_{i\in\N}$ be the points of the Poisson point process $\CP$ on $\R$ with intensity measure $\lambda(x)=c_{\alpha,\beta,\theta}\theta^{-x}\log \theta\,\d x$, ordered in decreasing order, let $(M_i)_{i\in\N}$ be a sequence of i.i.d.\ standard normal random variables and define $\mu:=1-(\theta-1)/(\theta\log \theta), \sigma^2:=1-(\theta-1)^2/(\theta^2\log \theta)$. Then, as $j\to\infty$, 
			\be
			\Big(d_{n_j}^i-\lfloor \log_\theta n_j-\beta \log_\theta \log_\theta n_j\rfloor, \frac{\log(\ell_{n_j}^i)-\mu\log n_j}{\sqrt{(1-\sigma^2)\log n_j}},i\in[n_j]\Big)\toindis (\lfloor P_i+\eps\rfloor , M_i, i\in\N).
			\ee 
		\end{theorem}
		
		\begin{proposition}\label{prop:betamoment}
			Consider the WRT model, that is, the WRG model as in Definition~\ref{def:wrg} with $m=1$, with vertex-weights $(W_i)_{\inn}$ which are distributed according to~\eqref{eq:betacdf} for some $\alpha,\beta>0$. Recall that $\theta:=1+\E W$ and that $(x)_k:=x(x-1)\cdots (x-(k-1))$ for $x\in\R,k\in\N$, and $(x)_0:=1$. Fix $K\in\N$, let $(j_k)_{k\in[K]}$ be a fixed non-decreasing sequence with $0\leq K':=\min\{k: j_{k+1}=j_K\}$, let $(B_k)_{k\in[K]}$ be a sequence of sets $B_k\in \CB(\R)$ such that $B_k\cap B_\ell=\emptyset $ when $j_k=j_\ell$ and $k\neq \ell$, and let $(c_k)_{k\in[K]}\in \N_0^K$. Recall the random variables $\wt X^{(n)}_j(B), \wt X_{\geq j}^{(n)}(B)$ and $\eps_n, c_{\alpha,\beta,\theta}$ from~\eqref{eq:epsnbeta}. Then, 
			\be \ba
			\E{\prod_{k=1}^{K'}\Big(\wt X_{j_k}^{(n)}(B_k)\Big)_{c_k}\prod_{k=K'+1}^K \Big(\wt X_{\geq j_k}^{(n)}(B_k)\Big)_{c_k}}={}&(1+o(1))\prod_{k=1}^{K'}\Big(\frac{c_{\alpha,\beta,\theta}(1-\theta^{-1})}{(1+\delta)^\beta}\theta^{-k+\eps_n}\Phi(B_k)\Big)^{c_k}\\
			&\times \prod_{k=K'+1}^{K}\Big(\frac{c_{\alpha,\beta,\theta}}{(1+\delta)^\beta}\theta^{-k+\eps_n}\Phi(B_k)\Big)^{c_k}.
			\ea\ee 
		\end{proposition}
		
		\begin{remark}
			A more general result as in Proposition~\ref{prop:momentconvwrt} holds in this particular example as well. However, as only the factorial moments of $\wt X_j^{(n)}(B),\wt X_{\geq j}^{(n)}(B)$ are of interest for Theorem~\ref{thrm:betappp}, these more general results are omitted here.
		\end{remark}
		
		We note that the beta distribution satisfies Conditions~\ref{item:c1},~\ref{item:c2}, and~\ref{item:c3} of Assumption~\ref{ass:weights}, so that this case is already captured by Theorem~\ref{thrm:conddegloc}. We hence do not need to state an analogue of this theorem here.
		
		Theorem~\ref{thrm:betappp} and Proposition~\ref{prop:betamoment} are the analogue of Theorem~\ref{thrm:deglocwrt} and Proposition~\ref{prop:momentconvwrt}. As the proof of the theorem presented here is very similar to the proof of Theorem~\ref{thrm:deglocwrt} (namely using Proposition~\ref{prop:betamoment} with a subsequence $n_j$ such that $\eps_{n_j}$, as in~\eqref{eq:epsnbeta}, converges to some $\eps\in[0,1]$, combined with the method of moments), we omit it here. The proof of the proposition is very similar to the proof of Proposition~\ref{prop:momentconvwrt} when using~\eqref{eq:asympbeta} from Corollary~\ref{cor:expasymp} in the~\hyperref[sec:appendix]{Appendix}, and is omitted, too.
	\end{example}
	
	\begin{example}[\ref{ass:gamma} case]
		We consider a random variable $W$ with a tail distribution as in~\eqref{eq:gumbex} for some $b\in \R,c_1>0,\tau\geq 1$ such that $b\leq 0$ when $\tau>1 $ and $bc_1\leq 1$ when $\tau=1$ (this condition is to ensure that the probability density function is non-negative on $[0,1)$). We define,
		\be\ba \label{eq:c}
		C_{\theta,c_1}&:=\frac{2}{\log \theta}\sqrt{\frac{1-\theta^{-1}}{c_1}},\qquad &C:=\e^{c_1^{-1}(1-\theta^{-1})/2}\sqrt{\pi}c_1^{-1/4+b/2}(1-\theta^{-1})^{1/4+b/2},&\\
		c_{\theta,c_1}&:=C\theta^{C_{\theta,c_1}^2/2},\qquad &K_{\theta,c_1,\tau}:=\frac1\theta \Big(\frac{\tau}{c_1^\tau(1-\theta^{-1})}\Big)^\gamma.&
		\ea \ee
		and, for $j\in\Z,B\in\CB(\R)$,
		\be\ba\label{eq:epsngamma}
		\wt X^{(n)}_j(B):={}&\Big|\Big\{\inn: \zni=\big\lfloor \log_\theta n-C_{\theta,c_1}\sqrt{\log_\theta n}+(b/2+1/4) \log_\theta \log_\theta n \big\rfloor +j, \\
		&\hphantom{\Big|\Big\{\inn:}\ \frac{\log i-\mu\log n}{\sqrt{(1-\sigma^2)\log n}}\in B\Big\}\Big|,\\ 
		\wt X^{(n)}_{\geq j}(B):={}&\Big|\Big\{\inn: \zni\geq \big\lfloor \log_\theta n-C_{\theta,c_1}\sqrt{\log_\theta n}+(b/2+1/4) \log_\theta \log_\theta n \big\rfloor +j,\\
		&\hphantom{\Big|\Big\{\inn:}\ \frac{\log i-\mu\log n}{\sqrt{(1-\sigma^2)\log n}}\in B\Big\}\Big|,\\
		\eps_n:={}&\big(\log_\theta n-C_{\theta,c_1}\sqrt{\log_\theta n}+(b/2+1/4) \log_\theta \log_\theta n\big )\\
		&-\big\lfloor\log_\theta n-C_{\theta,c_1}\sqrt{\log_\theta n}+(b/2+1/4) \log_\theta \log_\theta n \big \rfloor.
		\ea\ee 
		Then, we can formulate the following results.
		
		\begin{theorem}\label{thrm:gammacond}
			Consider the WRT model, that is, the WRG model as in Definition~\ref{def:wrg} with $m=1$, with vertex-weights $(W_i)_{i\in\N}$ which are distributed according to~\eqref{eq:gumbex} for some $b\in\R,c_1>0,\tau\geq 1$ such that $b\leq 0$ when $\tau>1$ and  $bc_1\leq 1$ when $\tau=1$, and let $\gamma:=1/(\tau+1)$. Fix $k\in\N, c\in(0,\theta/(\theta-1))$, let $(d_i)_{i\in[k]}$ be $k$ integer-valued sequences that diverge as $n\to\infty$ such that $d_i\leq c\log n$ for all $i\in[k]$ and let $(v_i)_{i\in[k]}$ be $k$ distinct vertices selected uniformly at random without replacement from $[n]$. For $\tau\in[1,2)$, the tuple
				\be 
				\Big(\frac{\log v_i -(\log n-(1-\theta^{-1})(d_i+K_{\theta,c_1,\tau}d_i^{1-\gamma}))}{\sqrt{(1-\theta^{-1})^2d_i }}\Big)_{i\in[k]},
				\ee  
				conditionally on the event $\Zm_n(v_i)\geq d_i$ for all $i\in[k]$, converges in distribution to $(M_i)_{i\in[k]}$, where the $M_i$ are i.i.d.\ standard normal random variables, and with $K_{\theta,c_1,\tau}$ as in~\eqref{eq:c}.
		\end{theorem}
		 
			\begin{remark}
				$(i)$ We see here that the behaviour of the labels of high-degree vertices is different compared to Theorem~\ref{thrm:conddegloc}, where the second-order term $K_{\theta,c_1,\tau}d_i^{1-\gamma}$ is not present. This is due to the exponential decay of the vertex-weight tail distribution near one, which does not satisfy Condition~\ref{item:c2}, as discussed in Remark~\ref{rem:ass}$(i)$ and $(iii)$, as well as in the heuristic arguments in Section~\ref{sec:heur}.
				
				$(ii)$ The statement of the theorem is different to that of Theorem~\ref{thrm:conddegloc}, as there is no need to distinguish between two cases. This is due to the fact that the distribution in~\eqref{eq:gumbex} satisfies Condition~\ref{item:c3} and so the two cases can be presented as one. 
				
				$(iii)$ When $\tau=1$, we observe that $d_i^{1-\gamma}=\sqrt{d_i}$ so that the tuples contain a constant term. Hence, the statement in Theorem~\ref{thrm:gammacond} for $\tau=1$ is equivalent to saying that the tuple
				\be 
				\Big(\frac{\log v_i -(\log n-(1-\theta^{-1})d_i)}{\sqrt{(1-\theta^{-1})^2d_i }}\Big)_{i\in[k]},
				\ee  
				conditionally on the event $\Zm_n(v_i)\geq d_i$ for all $i\in[k]$, converges in distribution to $(M_i')_{i\in[k]}$, where the $M_i'$ are i.i.d.\ $\cN(-K_{\theta,c_1,1},1)$ random variables.
				
				$(iv)$ When $\tau\geq 2$ we expect more higher-order terms to appear, which require a proof with even more precise and technical estimates and hence are not included here.
			\end{remark} 
		
		In the case that $\tau=1$, we have a precise asymptotic expression for $p_{\geq d}$ from Theorem~\ref{thrm:pkasymp}. This enables us to derive the following more detailed results:
		
		\begin{theorem}\label{thrm:gammappp}
			Consider the WRT model, that is, the WRG model in Definition~\ref{def:wrg} with $m=1$, with vertex-weights $(W_i)_{\inn}$ which are distributed according to~\eqref{eq:gumbex} for $\tau=1$ and some $b\in\R,c_1>0$ such that $bc_1\leq 1$ and recall $\theta=1+\E W$ and $C_{\theta,c_1},c_{\theta,c_1},$ and $K_{\theta,c_1,1}$ from~\eqref{eq:c}. Let $v^1,v^2,\ldots, v^n$ be the vertices in the tree in decreasing order of their in-degree $($\!where ties are split uniformly at random$)$, let $d_n^i$ and $\ell_n^i$ denote the in-degree and label of $v^i$, respectively, and fix $\eps\in[0,1]$. Recall $\eps_n$ from~\eqref{eq:epsngamma} and let $(n_j)_{j\in\N}$ be a positive, diverging, integer sequence such that $\eps_{n_j}\to\eps$ as $j\to\infty$. Finally, let $(P_i)_{i\in\N}$ be the points of the Poisson point process $\CP$ on $\R$ with intensity measure $\lambda(x)=c_{\theta,c_1}\theta^{-x}\log \theta\, \d x$, ordered in decreasing order, let $(M_{i,\theta,c_1})_{i\in\N}$ be a sequence of i.i.d.\ $\CN(-K_{\theta,c_1,1},1)$ random variables and define $\mu:=1-(\theta-1)/(\theta\log \theta)$, $ \sigma^2:=1-(\theta-1)^2/(\theta^2\log \theta)$. Then, as $j\to\infty$, 
			\be\ba
			\Big({}&d_{n_j}^i-\Big\lfloor \log_\theta n_j-C_{\theta,c_1}\sqrt{\log_\theta n}+ \Big(\frac b2+\frac 14\Big)\log_\theta \log_\theta n_j\Big\rfloor, \frac{\log(\ell_{n_j}^i)-\mu\log n_j}{\sqrt{(1-\sigma^2)\log n_j}},i\in[n_j]\Big)\\
			&\toindis (\lfloor P_i+\eps\rfloor , M_{i,\theta,c_1}, i\in\N).
			\ea \ee  
		\end{theorem}
		
		\begin{proposition}\label{prop:gammamoment}
			Consider the WRT model, that is, the WRG model as in Definition~\ref{def:wrg} with $m=1$, with vertex-weights $(W_i)_{\inn}$ which are distributed according to~\eqref{eq:gumbex} for some $b\in\R,c_1>0$ such that $bc_1\leq 1$. Recall that $\theta:=1+\E W$ and that $(x)_k:=x(x-1)\cdots (x-(k-1))$  for $x\in\R,k\in\N$, and $(x)_0:=1$. Fix $K\in\N$, let $(j_k)_{k\in[K]}$ be  a fixed non-decreasing sequence with $0\leq K':=\min\{k: j_{k+1}=j_K\}$, let $(B_k)_{k\in[K]}$ be a sequence of sets $B_k\in\CB(R)$ such that $B_k\cap B_\ell=\emptyset $ when $j_k=j_\ell$ and $k\neq \ell$, and let $(c_k)_{k\in[K]}\in \N_0^K$. Recall the random variables $\wt X^{(n)}_j(B),\wt X_{\geq j}^{(n)}(B)$ and the sequence $\eps_n$ from~\eqref{eq:epsngamma},  $c_{\theta,c_1}$ and $C_{\theta,c_1}$ from~\eqref{eq:c}, and let $\Phi_{\theta,c_1}$ denote the cumulative distribution function of $\CN(-1/\sqrt{c_1\theta(\theta-1)},1)$. Then, 
			\be \ba
			\mathbb E\Bigg[ \prod_{k=1}^{K'}\!\Big(\wt X_{j_k}^{(n)}(B_k)\Big)_{c_k}\!\prod_{k=K'+1}^K\!\! \Big(\wt X_{\geq j_k}^{(n)}(B_k)\Big)_{c_k}\!\Bigg]={}&(1+o(1))\prod_{k=1}^{K'}\Big(c_{\theta,c_1}(1-\theta^{-1})\theta^{-k+\eps_n}\Phi_{\theta,c_1}(B_k)\Big)^{c_k}\\
			&\times \prod_{k=K'+1}^{K}\Big(c_{\theta,c_1}\theta^{-k+\eps_n}\Phi_{\theta,c_1}(B_k)\Big)^{c_k}.
			\ea\ee 
		\end{proposition}
		
		\begin{remark}
			A more general result as in Proposition~\ref{prop:momentconvwrt} holds in this particular example as well. However, as only the factorial moments of $\wt X_j^{(n)}(B),\wt X_{\geq j}^{(n)}(B)$ are of interest for Theorem~\ref{thrm:gammappp}, these more general results are omitted here.
		\end{remark} 
		
		We observe that the behaviour of the labels of high-degree vertices in the above results is different e.g.\  Theorem~\ref{thrm:conddegloc}. Since the higher-order terms of the asymptotic expression of the degree are of the same order as the second-order rescaling of the label of the high-degree vertices, this causes a correlation between the higher-order behaviour of the degree and the location, so that more complex behaviour is observed.
		
		Theorems~\ref{thrm:gammacond} and~\ref{thrm:gammappp} and Proposition~\ref{prop:gammamoment} are the analogue of Theorems~\ref{thrm:conddegloc} and~\ref{thrm:deglocwrt} and Proposition~\ref{prop:momentconvwrt}, respectively. As proof of the theorems presented here are very similar to the proofs of Theorems~\ref{thrm:conddegloc} and~\ref{thrm:deglocwrt} (namely using~\eqref{eq:asympgamma2} rather than~\eqref{eq:genasymp} in the proof of Theorem~\ref{thrm:conddegloc} to prove Theorem~\ref{thrm:gammacond}, and using Proposition~\ref{prop:gammamoment} with a subsequence $n_j$ such that $\eps_{n_j}$, as in~\eqref{eq:epsngamma}, converges to some $\eps\in[0,1]$, combined with the method of moments to prove Theorem~\ref{thrm:gammappp}), we omit them here. The proof of the proposition is very similar to the proof of Proposition~\ref{prop:momentconvwrt} when using~\eqref{eq:asympgamma2} from Lemma~\ref{lemma:expasymp} in the~\hyperref[sec:appendix]{Appendix}, and is omitted, too.
	\end{example}
	
	\textbf{Acknowledgements}\\
	Bas Lodewijks has been supported by grant GrHyDy ANR-20-CE40-0002. He would also like to thank the anonymous referees for providing comments and suggestions that helped to substantially improve the presentation of the article as well as generalise some of the results.
	
	\bibliographystyle{abbrv}
	\bibliography{wrtbib}
	
	\appendix
	\section{}\label{sec:appendix}
	
	\begin{lemma}\label{lemma:expasymp}
		Consider the same definitions and assumptions as in Proposition~\ref{prop:deglocwrt}. We provide the asymptotic value of $\P{\Zm_n(v_1)\geq d, v_1>\ell}$ under several assumptions on the distribution of $W$ and a parametrisation of $\ell$ in terms of $d$. In all cases we let $d$ diverge as $n\to\infty$. We first set, for $x\in\R$, 
		\be \label{eq:elldef}
		\ell:=n\exp(-(1-\theta^{-1})d+x\sqrt{(1-\theta^{-1})^2d}).
		\ee
		We now distinguish between the different cases:
		
		When $W$ has a distribution that satisfies Conditions~\ref{item:c1} and~\ref{item:c2} of Assumption~\ref{ass:weights}, 
			\be \label{eq:genasymp}
			\P{\Zm_n(v_1)\geq d,v_1\geq \ell}=\E{\Big(\frac{W}{\theta-1+W}\Big)^d}(1-\Phi(x))(1+o(1))=p_{\geq d}(1-\Phi(x))(1+o(1)).
			\ee 
			Furthermore, let $W$ satisfy the~\ref{ass:gamma} case of Assumption~\ref{ass:weights} for some $b\in\R,c_1>0,\tau\in[1,2)$ such that $b\leq 0$ when $\tau>1$ and $bc_1\leq 1$ when $\tau=1$, set $\gamma:=1/(\tau+1)$, and define, for $x\in\R$ and with $K_{\theta,c_1,\tau}$ as in~\eqref{eq:c},
			\be \label{eq:elldef2}
			\ell:=n\exp(-(1-\theta^{-1})(d+K_{\theta,c_1,\tau}d^{1-\gamma})+x\sqrt{(1-\theta^{-1})^2d}).
			\ee 
			Then,
			\be \label{eq:asympgamma2}
			\P{\Zm_n(v_1)\geq d, v_1\geq \ell}= \mathbb E\bigg[\Big(\frac{W}{\theta-1+W}\Big)^{d}\bigg](1-\Phi(x))(1+o(1))=p_{\geq d}(1-\Phi(x))(1+o(1)).
			\ee 
	\end{lemma}
	
	\begin{remark}\label{rem:expasymp}
		$(i)$ For $k>1$ and with $(d_i,\ell_i)_{i\in[k]}$ satisfying the assumptions of Proposition~\ref{prop:deglocwrt}, it follows that 
		\be
		\P{\Zm_n(v_i)\geq d_i, v_i>\ell_i, i\in[k]}=(1+o(1))\prod_{i=1}^k \P{\Zm_n(v_i)\geq d_i, v_i>\ell_i}, 
		\ee  		
		so that the result of Lemma~\ref{lemma:expasymp} can immediately be extended to the case $k>1$ as well with $\ell_i=n\exp(-(1-\theta^{-1})d_i+x_i\sqrt{(1-\theta^{-1})^2d_i})$ and $(x_i)_{i\in[k]}\in\R^k$ (and a similar adaptation for~\eqref{eq:elldef2}).
		
		$(ii)$ With only minor modifications to the proof, we can show that in all cases of Lemma~\ref{lemma:expasymp}, 
		\be 
		\P{\Zm_n(v_1)=d, v_1>\ell}=(1-\theta^{-1})\P{\Zm_n(v_1)\geq d, v_1>\ell}(1+o(1)), 
		\ee  
		is satisfied. This holds in the case of $k$ vertices, as in point $(i)$, as well.
	\end{remark}
	
	A direct corollary of Lemma~\ref{lemma:expasymp} is that we can obtain several precise asymptotic expressions for $\P{\Zm_n(v_1)\geq d,v_1\geq \ell}$ for particular choices of the random variable $W$, whose distribution either satisfies Conditions~\ref{item:c1} and~\ref{item:c2}, or the~\ref{ass:gamma} case, and for which we have a precise asymptotic expression for $p_{\geq d}$. The asymptotics follow from combining Lemma~\ref{lemma:expasymp} with Theorem~\ref{thrm:pkasymp}.
		
		\begin{corollary} \label{cor:expasymp}
			When $W$ satisfies the~\ref{ass:weightatom} case for some $q_0\in(0,1]$, and with $\ell $ as in~\eqref{eq:elldef},
			\be\label{eq:asympatom}
			\P{\Zm_n(v_1)\geq d, v_1>\ell}=q_0\theta^{-d}(1-\Phi(x))(1+o(1)).
			\ee 
			When $W$ satisfies the~\ref{ass:beta} case for some $\alpha,\beta>0$, and with $\ell $ as in~\eqref{eq:elldef},  
			\be \label{eq:asympbeta}
			\P{\Zm_n(v_1)\geq d, v_1>\ell}=\frac{\Gamma(\alpha+\beta)}{\Gamma(\alpha)(1-\theta^{-1})^{\beta}}d^{-\beta}\theta^{-d}(1-\Phi(x))(1+o(1)).
			\ee 
			When $W$ satisfies the~\ref{ass:gamma} case for $\tau=1$ and some $b\in\R,c_1>0$ such that $bc_1\leq 1$, and with $\ell$ as in~\eqref{eq:elldef2},
			\be
			\P{\Zm_n(v_1)\geq d, v_1\geq \ell}= Cd^{b/2+1/4}\e^{-2\sqrt{c_1^{-1}(1-\theta^{-1})d}}\theta^{-d}(1-\Phi(x))(1+o(1)), 
			\ee 
			where $C$ is as in~\eqref{eq:c}.
		\end{corollary}
	
	\begin{remark}
		By the parametrisation of $\ell$, the event $\{v_1>\ell\}$ is equivalent to \be 
		\Big\{\frac{\log v_1-(\log n-(1-\theta^{-1})d_i)}{\sqrt{(1-\theta^{-1})^2d_i}}\in (x,\infty)\Big\}.
		\ee
		As a result, we can rewrite e.g.\ \eqref{eq:asympatom} as 
		\be 
		\P{\Zm_n(v_1)\geq d,\frac{\log v_1-(\log n-(1-\theta^{-1})d_i)}{\sqrt{(1-\theta^{-1})^2d_i}}\in (x,\infty)}=q_0\theta^{-d}\Phi((x,\infty))(1+o(1)), 
		\ee 
		and it can, in fact, be generalised to any set $A\in\CB(\R)$ rather than just $(x,\infty)$ with $x\in \R$. A similar notational change can be made in~\eqref{eq:asympbeta},~\eqref{eq:asympgamma2}, and~\eqref{eq:genasymp}.
	\end{remark}

	\begin{proof} [Proof of Lemma~\ref{lemma:expasymp}]
		We first observe that for our choice of $\ell$ (both as in~\eqref{eq:elldef} and~\eqref{eq:elldef2}), the conditions on $\ell$ in Proposition~\ref{prop:deglocwrt} are met (for $n$ sufficiently large) since $d$ diverges with $n$. By Proposition~\ref{prop:deglocwrt}, we thus have the bounds 
		\be \ba 
		\mathbb P({}&\Zm_n(v_1)\geq d, v_1>\ell)\leq(1+o(1)) \E{\Big(\frac{W}{\theta-1+W}\Big)^{d}\Pf{X<\Big(1+\frac{W}{\theta-1}\Big)\log(n/\ell)}},\\
		\mathbb P({}&\Zm_n(v_1)\geq d, v_1>\ell)\geq(1+o(1))\E{\Big(\frac{W}{\theta-1+W}\Big)^{d}\Pf{\wt X<\Big(1+\frac{W}{\theta-1}\Big)\log(n/\ell)}},
		\ea \ee 
		where $X\sim\text{Gamma}(d+1,1), \wt X\sim\text{Gamma}(d+\lfloor d^{1/4}\rfloor +1, 1)$. To prove the desired results, it suffices to provide an asymptotic expression for the expected values on the right-hand side. We do this for the expected value in the upper bound; the proof for the other expected value follows similarly. 
		
		We use the following approach to prove~\eqref{eq:genasymp}. To obtain an upper bound, we use that $W\leq 1$ almost surely in the conditional probability, which yields
			\be 
			\P{\Zm_n(v_1)\geq d,v_1>\ell}\leq (1+o(1))\E{\Big(\frac{W}{\theta-1+W}\Big)^{d}}\P{X<\frac{\theta}{\theta-1}\log(n/\ell)}, 
			\ee 
			so that it remains to prove that the probability converges to $1-\Phi(x)$. By the parametrisation of $\ell$, it follows that 
			\be \label{eq:cltprob}
			\P{X<\frac{\theta}{\theta-1}\log(n/\ell)}=\P{X<d-x\sqrt{d}}=\P{\frac{X-\E X}{\sqrt{\Var(X)}}\leq \frac{d-x\sqrt{d}-\E X}{\sqrt{\Var(X)}}}.
			\ee 
			As $X$ can be viewed as a sum of $d+1$ i.i.d.\ rate one exponential random variables, the central limit theorem can be applied to the left-hand side in the final probability. Moreover, as $\E X=d+1$ and $\Var(X)=d+1$, it follows that the limit equals $1-\Phi(x)$, as desired.
			
			To obtain a lower bound, we take some sequence $t_d\geq 1$ that tends to infinity with $d$ (and hence with $n$). We then bound
			\be \ba \label{eq:degproblb}
			\mathbb P({}&\Zm_n(v_1)\geq d, v_1>\ell)\\
			&\geq (1+o(1))\E{\Big(\frac{W}{\theta-1+W}\Big)^{d}\ind_{\{1-1/t_d\leq W\leq 1\}}}\P{X<\frac{\theta}{\theta-1}\Big(1-\frac{1}{\theta t_d}\Big)\log(n/\ell)}.
			\ea \ee 
			We can write the probability as 
			\be 
			\P{X<d-x\sqrt{d}-(d-x\sqrt{d})/(\theta t_d)}.
			\ee 
			Hence, with the same steps as in~\eqref{eq:cltprob} we arrive at the same limit $1-\Phi(x)$ whenever $\sqrt{d}/t_d=o(1)$. So, let us set $t_d=d^\beta$ for some $\beta\in(1/2,1/(1+\tau))$. We observe that this interval is non-empty since $\tau\in(0,1)$. It remains to show that for this choice of $t_d$, the expected value on the right-hand side of~\eqref{eq:degproblb} with the indicator is asymptotically equal to the same expected value when the indicator is omitted. Equivalently, we require that  
			\be \label{eq:neg}
			\E{\Big(\frac{W}{\theta-1+W}\Big)^{d}\ind_{\{ W\leq1-1/t_d\}}}=o\bigg(\E{\Big(\frac{W}{\theta-1+W}\Big)^{d}}\bigg).
			\ee 
			To prove this, we bound the expected value on the left-hand side from above and the one on the right-hand side from below. We start with the former. Since $x\mapsto x/(\theta-1+x)$ is increasing on $(0,1]$, we directly have that 
			\be \label{eq:truncexpub}
			\E{\Big(\frac{W}{\theta-1+W}\Big)^{d}\ind_{\{ W\leq1-1/t_d\}}}\leq \Big(\frac{1-1/t_d}{\theta-1/t_d}\Big)^d\leq \exp(-(1-\theta^{-1})d/t_d)\theta^{-d}.
			\ee 
			To bound the other expected value from below, we let $\wt t_d:=t^{\wt\beta}$ for some $\wt\beta>\beta$. As $x\mapsto x/(\theta-1+x)$ is increasing on $(0,1)$, we obtain the lower bound
			\be\label{eq:inttrans}
			\E{\Big(\frac{W}{\theta-1+W}\Big)^d}\geq \E{\Big(\frac{W}{\theta-1+W}\Big)^d\ind_{\{W\geq 1-1/\wt t_d\}}}\geq  \Big(\frac{1-1/\wt t_d}{\theta-1/\wt t_d}\Big)^d\P{ W\geq 1-1/\wt t_d}
			\ee 
			Now using Condition~\ref{item:c2} from Assumption~\ref{ass:weights} yields for $n$ sufficiently large the lower bound 
			\be 
			\Big(\frac{1-1/\wt t_d}{\theta-1/\wt t_d}\Big)^d a \exp\big(-c_1 \wt t_d^\tau\big).
			\ee 
			We then bound 
			\be  
			\Big(\frac{1-1/\wt t_d}{\theta-1/\wt t_d}\Big)^d
			=
			\theta^{-d}\Big(1-\frac{\theta-1}{\wt t_d\theta-1}\Big)^d= \theta^{-d}\exp(-(1-\theta^{-1})d/\wt t_d+\cO(d/\wt t_d^2)).
			\ee  
			Combined, we obtain the lower bound
			\be \label{eq:finlb}
			\E{\Big(\frac{W}{\theta-1+W}\Big)^{d}}\geq a\exp(-(1-\theta^{-1})d/\wt t_d-c_1\wt t_d^\tau+\cO(d/\wt t_d^2))\theta^{-d}.
			\ee 
			The upper bound in~\eqref{eq:truncexpub} is negligible compared to this lower bound when $d/\wt t_d=o(d/t_d)$ and $\wt t_d^\tau=o(d/t_d)$. That is, we require that $\beta<\wt \beta$ and $\wt\beta\tau<1-\beta$. Such a $\wt\beta$ can be found since $\beta<1/(1+\tau)$. As a result, the claim in~\eqref{eq:neg} follows, which results in the desired lower bound and finishes the proof of~\eqref{eq:genasymp}.
			
		Finally, we prove~\eqref{eq:asympgamma2}, that is, when $W$ satisfies~\eqref{eq:gumbex} for some $b\in\R, c_1>0$ and $\tau\in[1,2)$ such that $b\leq 0$ if $\tau>1$ and $bc_1\leq 1$ if $\tau=1$. Set $\gamma:=1/(\tau+1)$. Note that this distribution does not satisfy Condition~\ref{item:c2} in Assumption~\ref{ass:weights}. The behaviour here is different, since the main contribution to the expected value $\E{(W/(\theta-1+W))^d}$ comes from $W=1-Kd^{-\gamma}$ for $K$ a positive constant. At the same time, for $W=1-Kd^{-\gamma}$,
		\be 
		\Pf{X\leq \Big(1+\frac{W}{\theta-1}\Big)\log(n/\ell)}=\P{X\leq \frac{\theta}{\theta-1}\Big(1-\frac{K}{d^{\gamma}}\Big)\log (n/\ell)}
		\ee
		no longer converges to the tail of a standard normal distribution when $\ell$ is as in~\eqref{eq:elldef}, as the $\log(n/\ell)/d^\gamma$ term is of the same order as the variance of $X$ when $\tau=1$ and of higher order when $\tau>1$. As a result, we need $\ell$ to be as in~\eqref{eq:elldef2}.
		
		To be able to obtain the desired result, we first need a lower bound for $p_{\geq d}$ when $\tau>1$ (for $\tau=1$ this is already provided in Theorem~\ref{thrm:pkasymp}). With similar steps as in~\eqref{eq:inttrans} through~\eqref{eq:finlb} and with $t_d=(c_1^\tau(1-\theta^{-1})d/\tau)^\gamma$, we obtain for some constants $K,\wt K>0$,
		\be \ba \label{eq:improvelb}
		\E{\Big(\frac{W}{\theta-1+W}\Big)^{d}}&\geq \theta^{-d} \exp(-(1-\theta^{-1})d/t_d-(t_d/c_1)^\tau-K d/t_d^2)\\
		&=\theta^{-d}\exp\bigg(-\frac{\tau^\gamma}{1-\gamma}\Big(\frac{(1-\theta^{-1})d}{c_1}\Big)^{1-\gamma}-\wt K d^{1-2\gamma}\bigg),
		\ea \ee 
		We now aim to find an upper and lower bound for 
		\be 
		\mathbb E\bigg[\Big(\frac{W}{\theta-1+W}\Big)^{d}\Pf{X<\Big(1+\frac{W}{\theta-1}\Big)\log(n/\ell)}\bigg].
		\ee 
		We start with an upper bound. We let $\eps\in(0,1)$ fixed (when $\tau=1$) or set $\eps=\eps(d)=K_1d^{-\gamma/2}$ for some large constant $K_1$ (when $\tau>1$). We then bound 
		\be \ba\label{eq:firstbound} 
		\mathbb E\bigg[{}&\Big(\frac{W}{\theta-1+W}\Big)^{d}\Pf{X<\Big(1+\frac{W}{\theta-1}\Big)\log(n/\ell)}\bigg]\\
		\leq{}& \E{\Big(\frac{W}{\theta-1+W}\Big)^{d}\ind_{\{1-(1-\eps)/t_d<W<1\}}}\\
		&+\E{\Big(\frac{W}{\theta-1+W}\Big)^{d}}\P{X\leq \frac{\theta}{\theta-1}\Big(1-\frac{1-\eps}{\theta t_d}\Big)\log(n/\ell)}.
		\ea \ee  		
		We then show that the first expected value on the right-hand side is negligible compared to the second, and that the probability has a non-zero limit. We start with the expected value. By the distribution of $W$ as in~\eqref{eq:gumbex}, we find  
		\be \ba 
		\E{\Big(\frac{W}{\theta-1+W}\Big)^{d}\ind_{\big\{1-\frac{1-\eps}{t_d}<W<1\big\}}}&= \int_{1-(1-\eps)/t_d}^1 \!\!\!\!\!\!\!\!\!\! d(\theta-1)\frac{x^{d-1}}{(\theta-1+x)^{d+1}}(1-x)^{-b}\e^{-(x/(c_1(1-x)))^\tau}\,\d x\\
		&\leq (1+o(1))d\int_{t_d/(1-\eps)}^\infty \Big(\frac{1-1/y}{\theta-1/y}\Big)^d y^{b-2}\e^{-((y-1)/c_1)^\tau}\,\d y.
		\ea\ee 
		In the last step, we used that $x^{-1}=1+o(1)$ for $x\in(1-(1-\eps)/t_d,1)$, that $(\theta-1)/(\theta-1+x)\leq 1$, as well as a variable transformation $x=1-1/y$. We now introduce the function $f:(0,1)\to (0,1)$, with $f(\eps)=1/2+(1/2)(1+\tau\eps)(1-\eps)^\tau$. Since, for all $\eps>0$ sufficiently small, $(1+\tau\eps)(1-\eps)^\tau=1-\eps^2\tau(\tau+1)/2+o(\eps^2)<1$, this function satisfies 
		\be \label{eq:fprop}
		f(\eps)>(1+\tau\eps)(1-\eps)^\tau>(1-\eps)^{\tau+1}, \quad\text{and}\quad f(\eps)<1,\qquad \text{for all }\eps\in(0,1).
		\ee 
		We then observe that, for any $b\in\R$, we can bound $y^b\e^{-((y-1)/c_1)^\tau}\leq \e^{-f(\eps)((y-1)/c_1)^\tau}$ for all $y>t_d/(1-\eps)$ when $n$ is sufficiently large, since $f(\eps)<1$ holds (note that this upper bound holds for $\eps>0$ fixed and also for $\eps=K_1d^{-\gamma/2}$ and any constant $K_1>0$ when $\tau>1$). A bound similar to~\eqref{eq:truncexpub} also yields
		\be \label{eq:fracbound2}
		\Big(\frac{1-1/y}{\theta-1/y}\Big)^d\leq \theta^{-d}\exp\Big(-(1-\theta^{-1})\frac{d}{y-1}+(1-\theta^{-1})^2\frac{d}{(y-1)^2}\Big).
		\ee 
		Combining both bounds and using that $(1-\theta^{-1})^2d/(y-1)^2\leq Cd^{1-2\gamma}$ for $y>t_d/(1-\eps)$ and some constant $C>0$ yields the upper bound
		\be \label{eq:intub0}
		Kd\theta^{-d}\int_{t_d/(1-\eps)}^\infty y^{-2}\exp\Big(-(1-\theta^{-1})\frac{d}{y-1}-f(\eps)\Big(\frac{y-1}{c_1}\Big)^\tau+Cd^{1-2\gamma}\Big)\,\d y,
		\ee 
		where $K>0$ is a large constant. The exponential is decreasing in $y$ for all $y>1+t_df(\eps)^{-\gamma}$. By the first inequality in~\eqref{eq:fprop}, it thus follows that the exponential in the integral is maximised for $y=t_d/(1-\eps)>1+t_df(\eps)^{-\gamma}$. As a result, we obtain the upper bound
		\be \ba\label{eq:intub}
		K{}&d\theta^{-d}\exp\Big(-(1-\theta^{-1})(1-\eps)\frac{d}{t_d}-f(\eps)\Big(\frac{t_d}{c_1(1-\eps)}\Big)^\tau+C'd^{1-2\gamma}\Big)\\
		={}&Kd\theta^{-d}\exp\Bigg(-\Big(1-\eps+\frac{f(\eps)}{\tau(1-\eps)^\tau}\Big)\tau^\gamma \Big(\frac{(1-\theta^{-1})d}{c_1}\Big)^{1-\gamma}+C'd^{1-2\gamma}\Bigg).
		\ea \ee 
		Here we change the constant $C$ to a constant $C'>C$, since 
		\be \label{eq:const}
		\frac{d}{t_d/(1-\eps)-1}+f(\eps)\Big(\frac{t_d/(1-\eps)-1}{c_1}\Big)^\tau=(1-\eps)\frac{d}{t_d}+f(\eps)\Big(\frac{t_d}{c_1(1-\eps)}\Big)^\tau+\cO(d^{1-2\gamma}).
		\ee 
		We have that $1-\eps+f(\eps)/(\tau(1-\eps)^\tau)>1+1/\tau=1/(1-\gamma)$ for all $\eps\in(0,1)$ by the first inequality in~\eqref{eq:fprop}. Thus, the lower bound in~\eqref{eq:improvelb} yields that for any $\eps>0$ fixed,
		\be \label{eq:expneg}
		\E{\Big(\frac{W}{\theta-1+W}\Big)^{d}\ind_{\{1-(1-\eps)/t_d<W<1\}}}=o\bigg(\E{\Big(\frac{W}{\theta-1+W}\Big)^{d}}\bigg),
		\ee  
		Whilst this holds for all $\tau\in[1,2)$, we need a stronger statement for $\tau\in(1,2)$, namely that~\eqref{eq:expneg} is true with $\eps=K_1d^{-\gamma/2}$ (which does not hold for $\tau=1$). We stress that all the above steps also hold with this choice of $\eps$ as well. Additionally, a Taylor expansion yields that
		\be 
		1-\eps+\frac{f(\eps)}{\tau(1-\eps)^\tau}=\frac{1}{1-\gamma}+\frac{\tau+1}{4}\eps^2(1+o(1))>\frac{1}{1-\gamma}+\frac{\tau+1}{8}\eps^2, \qquad \text{as }\eps\downarrow 0.
		\ee 
		Using this in~\eqref{eq:intub}, we obtain 
		\be \ba
		\mathbb E \bigg[{}&\Big(\frac{W}{\theta-1+W}\Big)^{d}\ind_{\{1-(1-\eps)/t_d<W<1\}}\bigg]\\
		&\leq K\theta^{-d}\exp\Bigg(-\frac{\tau^\gamma}{1-\gamma} \Big(\frac{(1-\theta^{-1})d}{c_1}\Big)^{1-\gamma}+\Big(C'-K_1^2\frac{(\tau+1)\tau^\gamma}{8}\Big(\frac{(1-\theta^{-1})}{c_1}\Big)^{1-\gamma}\Big)d^{1-2\gamma}\Bigg)\\
		&=\theta^{-d}\exp\Bigg(-\frac{\tau^\gamma}{1-\gamma} \Big(\frac{(1-\theta^{-1})d}{c_1}\Big)^{1-\gamma}-\wt K_1d^{1-2\gamma}(1+o(1))\Bigg),
		\ea \ee 
		where the constant $\wt K_1$ is positive for all large $K_1$ and grows polynomially in $K_1$. Again using the lower bound in~\eqref{eq:improvelb} implies that we need to choose $K_1$ sufficiently large, so that $\wt K_1>\wt K$. This then implies that~\eqref{eq:expneg} holds for $\tau>1$ with $\eps=K_1d^{-\gamma/2}$ as well.
		
		We now determine the limit of the probability on the right-hand side of~\eqref{eq:firstbound}. We again distinguish between the two cases $\tau=1$ and $\tau>1$ and start with the former. First, observe that $d^{1-\gamma}=\sqrt{d}$ when $\tau=1$. Then, since $\E X=\Var(X)=d+1$ and $\ell$ is as in~\eqref{eq:elldef2}, for a fixed $\eps>0$, 
		\be \ba \label{eq:probeq}
		\mathbb P{}&\bigg(X\leq \frac{\theta}{\theta-1}\Big(1-\frac{1-\eps}{\theta t_d}\Big)\log\Big(\frac n\ell\Big)\bigg)\\
		&=\P{\frac{X-\E X}{\sqrt{\Var(X)}}\leq \frac{(K_{\theta,c_1,1}-x)\sqrt{d}-1}{\sqrt{d+1}}-\frac{(1-\eps)(d+(K_{\theta,c_1,1}-x)\sqrt{d})}{\theta t_d\sqrt{d+1}}}.
		\ea\ee 
		As $t_d=\sqrt{c_1(1-\theta^{-1})d}$ when $\tau=1$ and with $Z\sim \cN(0,1)$, this equals 
		\be \ba\label{eq:normlim}
		\P{Z\leq K_{\theta,c_1,1}-x-(1-\eps)K_{\theta,c_1,1}}+o(1)=1-\Phi( x-\eps K_{\theta,c_1,1})+ o(1). 
		\ea \ee 
		Combining this with~\eqref{eq:expneg} in~\eqref{eq:firstbound} yields for $\tau=1$ and any $\eps>0$ fixed,
		\be \ba\label{eq:asympub}
		\mathbb E\bigg[{}&\Big(\frac{W}{\theta-1+W}\Big)^{d}\Pf{X<\Big(1+\frac{W}{\theta-1}\Big)\log(n/\ell)}\bigg]\\
		&\leq \mathbb E\bigg[\Big(\frac{W}{\theta-1+W}\Big)^{d}\bigg](1-\Phi(x-\eps K_{\theta,c_1,1}))(1+o(1)).
		\ea \ee 
		When $\tau\in(1,2)$ we adapt~\eqref{eq:probeq} and~\eqref{eq:normlim} with $\eps=K_1d^{-\gamma/2}$ to obtain 
		\be \ba\label{eq:gammaclt0}
		\mathbb P{}&\bigg(X\leq \frac{\theta}{\theta-1}\Big(1-\frac{1-\eps}{\theta t_d}\Big)\log\Big(\frac n\ell\Big)\bigg)\\
		&=\P{\frac{X-\E X}{\sqrt{\Var(X)}}\leq \frac{K_{\theta,c_1,1}d^{1-\gamma}-x\sqrt{d}-1}{\sqrt{d+1}}-\frac{(1-K_1d^{-\gamma/2})(d+(K_{\theta,c_1,1}-x)\sqrt{d})}{\theta t_d\sqrt{d+1}}}.
		\ea\ee 
		We observe that $d/(\theta t_d)=K_{\theta,c_1,1}d^{1-\gamma}$, so that the right-hand side can be simplified as
		\be \label{eq:gammaclt}
		\P{\frac{X-\E X}{\sqrt{\Var(X)}}\leq -x+o(1)+\cO(d^{1/2-3\gamma/2})}=1-\Phi(x)+o(1).
		\ee 
		Here, the last step follows from the fact that $\cO(d^{1/2-3\gamma/2})=o(1)$ when $\tau<2$ since $1/2-3\gamma/2<0$. We also stress that this is possible \emph{only} when $\eps$ tends to zero with $d$. If $\eps$ were fixed, this would yield a limit of one rather than $1-\Phi(x)$.
		
		Combining this with~\eqref{eq:expneg} when $\tau>1$ and $\eps=K_1d^{-\gamma}$, yields
		\be \ba\label{eq:asympub2}
		\mathbb E\bigg[{}&\Big(\frac{W}{\theta-1+W}\Big)^{d}\Pf{X<\Big(1+\frac{W}{\theta-1}\Big)\log(n/\ell)}\bigg]\\
		&\leq \mathbb E\bigg[\Big(\frac{W}{\theta-1+W}\Big)^{d}\bigg](1-\Phi(x))(1+o(1)).
		\ea \ee 
		In a similar way, we construct a matching lower bound (up to error terms). Namely, for $\eps\in(0,1)$,
		\be \ba\label{eq:pknormlb}
		\mathbb E\bigg[{}&\Big(\frac{W}{\theta-1+W}\Big)^{d}\Pf{ X<\Big(1+\frac{W}{\theta-1}\Big)\log(n/\ell)}\bigg]\\
		&\geq \mathbb E\bigg[\Big(\frac{W}{\theta-1+W}\Big)^{d}\ind_{\{1-(1+\eps)/t_d<W<1\}}\bigg]\P{ X<\frac{\theta}{\theta-1}\Big(1-\frac{1+\eps}{\theta t_d}\Big)\log(n/\ell)}.
		\ea\ee 
		Again, we let $\eps$ fixed when $\tau=1$ and set $\eps=K_1d^{-\gamma/2}$ for some large constant $K_1$ when $\tau>1$. As in~\eqref{eq:probeq} and~\eqref{eq:normlim}, we have for the probability on the right-hand side that
		\be \label{eq:problblim}
		\P{ X<\frac{\theta}{\theta-1}\Big(1-\frac{1+\eps}{\theta t_d}\Big)\log(n/\ell)}=1-\Phi(x+\eps K_{\theta,c_1,1})+o(1),
		\ee 
		when $\tau=1$ and $\eps>0$ is fixed, and similar to~\eqref{eq:gammaclt0} and~\eqref{eq:gammaclt}, 
		\be \label{eq:problim2}
		\P{ X<\frac{\theta}{\theta-1}\Big(1-\frac{1+K_1d^{-\gamma/2}}{\theta t_d}\Big)\log(n/\ell)}=1-\Phi(x)+o(1), 
		\ee  
		when $\tau\in(1,2)$ and $\eps=K_1d^{-\gamma/2}$. It remains to bound the expected value on the right-hand side of~\eqref{eq:pknormlb}. We instead consider the expected value
		\be 
		\mathbb E\bigg[\Big(\frac{W}{\theta-1+W}\Big)^{d}\ind_{\big\{0<W<1-\frac{1+\eps}{t_d}\big\}}\bigg]=\int_0^{1-(1+\eps)/t_d}\frac{d(\theta-1)x^{d-1}}{(\theta-1+x)^{d+1}}(1-x)^{-b}\e^{-(x/(c_1(1-x)))^\tau}\,\d x.
		\ee 
		We first bound $(1-x)^{-b}\leq t_d^{b\vee 0}$ and $(\theta-1)/(\theta-1+x)\leq 1$, and split the integral in two parts by dividing the integration range into $(0,1-2(1+\eps)/t_d)$ and $(1-2(1+\eps)/t_d,1-(1+\eps)/t_d)$. This yields the upper bound
		\be 
		\frac{dt_d^{b\vee 0}}{\theta-1}\int_0^{1-2(1+\eps)/t_d} \Big(\frac{x}{\theta-1+x}\Big)^{d-1}\,\d x+2dt_d^{b\vee 0} \int_{1-2(1+\eps)/t_d}^{1-(1+\eps)/t_d} \Big(\frac{x}{\theta-1+x}\Big)^d \e^{-(x/(c_1(1-x)))^\tau}\, \d x.
		\ee  
		Using that $x\mapsto x/(\theta-1+x)$ is increasing on $(0,1)$ and using a variable transformation $x=1-1/y$ in the second integral, yields the upper bound
		\be 
		\frac{dt_d^{b\vee 0}(\theta+o(1))}{\theta-1}\Big(\frac{1-2(1+\eps)/t_d}{\theta-2(1+\eps)/t_d}\Big)^d+2dt_d^{b\vee 0} \int_{t_d/(2(1+\eps))}^{t_d/(1+\eps)} y^{-2}\Big(\frac{1-1/y}{\theta-1/y}\Big)^d \e^{-((y-1)/c_1)^\tau}\, \d y.
		\ee 
		We now use~\eqref{eq:fracbound2} and steps similar to those that yielded~\eqref{eq:intub0}. We can then bound this from above by 
		\be\ba 
		K{}&dt_d^{b\vee 0}\theta^{-d}\exp\Big(-(1-\theta^{-1})\frac{2(1+\eps)d}{t_d}\Big)  \\
		&+dt_d^{b\vee 0}\theta^{-d}\!\int_{t_d/(2(1+\eps))}^{t_d/(1+\eps)}y^{-2}\exp\Big(-(1-\theta^{-1})\frac{d}{y-1}-\Big(\frac{y-1}{c_1}\Big)^\tau+(1-\theta^{-1})^2\frac{d}{(y-1)^2}\Big)\,\d y,
		\ea\ee 
		for some constant $K>0$. As $2(1+\eps)>1/(1-\gamma)$ for all $\tau\geq 1$ and any $\eps>0$, it follows from the choice of $t_d$ and the lower bound in~\eqref{eq:improvelb} that the first term is negligible compared to $\E{(W/(\theta-1+W)^d}$ when $\tau>1$ and $\eps=K_1d^{-\gamma/2}$ and also when $\tau=1$ and $\eps$ is fixed. 
		
		We thus focus on the integral only from now on. We bound the final term in the second integral from above by $C_2 d^{1-2\gamma}$ for some constant $C_2>0$. The remainder in the exponent is increasing for $y<1+t_d$. With the same reasoning as in~\eqref{eq:const}, we can bound the integral from above for some $C_2'>C_2$ by 
		\be \label{eq:nofub}
		\theta^{-d}d\exp\bigg(-\Big((1+\eps)+\frac{1}{\tau(1+\eps)^\tau}\Big)\tau^\gamma \Big(\frac{(1-\theta^{-1})d}{c_1}\Big)^{1-\gamma}+C_2'd^{1-2\gamma}\bigg)
		\ee 
		Since $(1+\eps)+\tau^{-1}(1+\eps)^{-\tau}>1/(1-\gamma)$ for any $\eps>0$, it follows from the lower bound in~\eqref{eq:improvelb} that this upper bound is negligible compared to $\E{(W/(\theta-1+W)^d}$ for any $\tau\geq 1$ when $\eps$ is fixed. Combined with ~\eqref{eq:problblim} this yields, for $\tau=1$ and $\eps$ fixed,  
		\be \ba
		\mathbb E\bigg[{}&\Big(\frac{W}{\theta-1+W}\Big)^{d}\Pf{X<\Big(1+\frac{W}{\theta-1}\Big)\log(n/\ell)}\bigg]\\
		&\geq  \mathbb E\bigg[\Big(\frac{W}{\theta-1+W}\Big)^{d}\bigg](1-\Phi( x+\eps K_{\theta,c_1,1}))(1+o(1)).
		\ea\ee 
		Together with~\eqref{eq:asympub}, since $\eps$ can be taken arbitrarily small and by the continuity of $\Phi$, we finally arrive at 
		\be 
		\mathbb E\bigg[\Big(\frac{W}{\theta-1+W}\Big)^{d}\Pf{X<\Big(1+\frac{W}{\theta-1}\Big)\log(n/\ell)}\bigg]= \mathbb E\bigg[\Big(\frac{W}{\theta-1+W}\Big)^{d}\bigg](1-\Phi(x))(1+o(1)),
		\ee  
		which proves~\eqref{eq:asympgamma2} when $\tau=1$.
		
		To obtain the same result for $\tau>1$ with $\eps=K_1d^{-\gamma/2}$, we use a Taylor expansion to find that
		\be 
		(1+\eps)+\tau^{-1}(1+\eps)^{-\tau}=\frac{1}{1-\gamma}+\frac{\tau+1}{2}\eps^2(1+o(1))>\frac{1}{1-\gamma}+\frac{\tau+1}{4}\eps^2, \qquad \text{as }\eps \downarrow 0.
		\ee 
		Using this in~\eqref{eq:nofub} yields, for some constant $\wt K_1>0$,  the upper bound 
		\be 
		\theta^{-d}\exp\bigg(-\frac{\tau^\gamma}{1-\gamma} \Big(\frac{(1-\theta^{-1})d}{c_1}\Big)^{1-\gamma}-\Big(\frac{\tau+1}{4}K_1^2 \tau^\gamma\frac{(1-\theta^{-1})}{c_1}\Big)^{1-\gamma}- C_2'\Big)d^{1-2\gamma}\bigg)
		\ee  
		As in the proof of the upper bound, we conclude that~\eqref{eq:improvelb} implies that choosing $K_1$ large enough yields for $\tau>1$ and $\eps=K_1d^{-\gamma/2}$, 
		\be 
		\mathbb E\bigg[\Big(\frac{W}{\theta-1+W}\Big)^{d}\ind_{\big\{0<W<1-\frac{1+\eps}{t_d}\big\}}\bigg]=o\bigg(\mathbb E\bigg[\Big(\frac{W}{\theta-1+W}\Big)^{d}\bigg]\bigg).
		\ee 
		Combined with~\eqref{eq:problim2} in~\eqref{eq:pknormlb}, we thus arrive at
		\be
		\mathbb E\bigg[\Big(\frac{W}{\theta-1+W}\Big)^{d}\Pf{X<\Big(1+\frac{W}{\theta-1}\Big)\log(n/\ell)}\bigg]\geq  \mathbb E\bigg[\Big(\frac{W}{\theta-1+W}\Big)^{d}\bigg](1-\Phi( x)(1+o(1)),
		\ee 
		Together with~\eqref{eq:asympub2}, this completes the proof of~\eqref{eq:asympgamma2} for $\tau>1$, and concludes the proof.
	\end{proof}	
	
	\begin{lemma}\label{lemma:asymp}
		Consider the same conditions as in Lemma~\ref{lemma:maxdeg}, let $\eps\in(0\vee (c(1-\theta^{-1})-(1-\mu)),\mu)$ and $\wt X\sim\text{Gamma}(d_n+\lfloor d_n^{1/4}\rfloor +1,1)$. Then,
		\be 
		\E{\Big(\frac{W}{\theta-1+W}\Big)^{d_n}\Pf{\wt X\leq \Big(1+\frac{W}{\theta-1}\Big)\log (n^{1-\mu+\eps})}}\geq \E{\Big(\frac{W}{\theta-1+W}\Big)^{d_n}}(1-o(1)).
		\ee   
	\end{lemma}
	
	We observe that this result is of a similar nature as~\eqref{eq:genasymp} in Lemma~\ref{lemma:expasymp}. However, as $\ell=n^{\mu-\eps}$ here, rather than the a precise parametrisation in terms of $d_n$ as is the case in Lemma~\ref{lemma:expasymp}, we can make a more general statement here (though not as precise and useful) that does not require Condition~\ref{item:c2} of Assumption~\ref{ass:weights}.
	
	\begin{proof}
		Fix $\delta\in(0,(1-(\theta-1)(c/(1-\mu+\eps)-1)\wedge 1))$. It is readily checked that by the choice of $\eps$, such a $\delta$ exists. We bound the expected value from below by writing
		\be \label{eq:xhat}
		\E{\Big(\frac{W}{\theta-1+W}\Big)^{d_n}\ind_{\{1-\delta<W\leq1\}}}\P{\hat X\leq \Big(1+\frac{1-\delta}{\theta-1}\Big)\log (n^{1-\mu+\eps})},
		\ee 
		where $\hat X\sim\text{Gamma}(c\log n+\lfloor (c\log n)^{1/4}\rfloor +1,1)$, which stochastically dominates $\wt X$ as $d_n\leq c\log n$. It thus remains to prove two things: the probability converges to one, and the expected value is asymptotically equal to $\E{(W/(\theta-1+W))^{d_n}}$. Together, they prove the lemma. We start with the former. By the choice of $\delta$, it follows that 
		\be 
		c_{\delta,\theta,\eps}:=\Big(1+\frac{1-\delta}{\theta-1}\Big)\frac{1-\mu+\eps}{c} >1.
		\ee 
		Thus, as $\hat X/(c\log n)\toas 1$, the probability in~\eqref{eq:xhat} equals $1-o(1)$. It remains to prove that 
		\be 
		\E{\Big(\frac{W}{\theta-1+W}\Big)^{d_n}\ind_{\{1-\delta<W\leq1\}}}=	\E{\Big(\frac{W}{\theta-1+W}\Big)^{d_n}}(1-o(1)),
		\ee 
		which is equivalent to showing that 
		\be \label{eq:expineq}
		\E{\Big(\frac{W}{\theta-1+W}\Big)^{d_n}\ind_{\{W\leq 1-\delta\}}}=o\bigg(	\E{\Big(\frac{W}{\theta-1+W}\Big)^{d_n}}\bigg).
		\ee 
		By Theorem~\ref{thrm:pkasymp}, for any $\xi>0$ and $n$ sufficiently large, 
		\be 
		\E{\Big(\frac{W}{\theta-1+W}\Big)^{d_n}}=p_{\geq d_n}\geq (\theta+\xi)^{-d_n}.
		\ee 
		So, take $\xi\in(0,\delta(\theta-1)/(1-\delta))$. Then, as $x\mapsto x/(\theta-1+x)$ is increasing in $x$, 
		\be 
		\E{\Big(\frac{W}{\theta-1+W}\Big)^{d_n}\ind_{\{W\leq 1-\delta\}}}\leq \Big(\frac{1-\delta}{\theta-\delta}\Big)^{d_n}=\Big(\theta+\frac{\delta(\theta-1)}{1-\delta}\Big)^{-d_n}=o\big((\theta+\xi)^{-d_n}\big),
		\ee 
		so that~\eqref{eq:expineq} follows. Combined with the lower bound on the probability in~\eqref{eq:xhat}, it yields the desired lower bound.
	\end{proof}
	
	\begin{lemma}\label{lemma:littleo}
		Consider the same definitions and assumptions as in Proposition~\ref{prop:deglocwrt} $($but without indices$)$. Let $c:=\limsup_{n\to\infty} d/\log n$ and assume that $c\in[0,\theta/(\theta-1))$. Then, 
		\be
		\frac{1}{n^\gamma}=o\bigg(\E{\frac{\theta-1}{\theta-1+W}\Big(\frac{W}{\theta-1+W}\Big)^{d}\Pf{X<\Big(1+\frac{W}{(\theta-1)}\Big)\log(n/\ell)}}\bigg).
		\ee 
		holds for $\gamma=1$ when $c\in[0,1/(\theta-1)]$ and for $\gamma$ sufficiently large when $c\in(1/(\theta-1),\theta/(\theta-1))$. 
	\end{lemma}
	
	\begin{proof}
		We first consider the case $c\in[0,1/(\theta-1)]$, for which we can set $\gamma=1$. We consider two sub-cases: $(i)$ $d$ is bounded from above, and $(ii)$ $d$ diverges (but $d$ is at most $(1/(\theta-1))\log n(1+o(1))$ for all $n$ large). For $(i)$ we immediately have that 
		\be \label{eq:dtbdd}
		\Pf{X<\Big(1+\frac{W}{(\theta-1)}\Big)\log(n/\ell)}\geq \P{X<\log(n/\ell)}\geq \P{X<(1-\xi)(1-\theta^{-1})(d+1)}, 
		\ee 
		when $n$ is sufficiently large and $\xi$ small, since $\ell\leq n\exp(-(1-\xi)(1-\theta^{-1})(d+1))$ for any $\xi>0$. Since $X$ is finite almost surely for all $n\in\N$ as $d$ is bounded, the probability on the right-hand side is strictly positive. The expected value that remains is again bounded from below by a positive constant, since $d$ is bounded from above. It thus follows that $1/n$ negligible compared to the expected value.
		
		For $(ii)$, we obtain a lower bound by restricting the weight $W$ in the expected value to $(1-\delta,1]$ for some small $\delta>0$. This yields the lower bound
		\be \ba \label{eq:neglb}
		\mathbb E\Bigg[{}&\frac{\theta-1}{\theta-1+W}\Big(\frac{W}{\theta-1+W}\Big)^{d}\Pf{X<\Big(1+\frac{W}{(\theta-1)}\Big)\log(n/\ell)}\ind_{\{W\in(1-\delta,1]\}}\Bigg]\\
		&\geq (1-\theta^{-1})\Big(\frac{1-\delta}{\theta-\delta}\Big)^{d}\P{X<\frac{\theta-\delta}{\theta-1}\log(n/\ell)}\P{W\in(1-\delta,1]}.
		\ea \ee  
		Note that $\P{W\in(1-\delta,1]}$ is strictly positive for any $\delta\in(0,1)$ by Condition~\ref{item:c1}. Furthermore, since $\ell\leq n\exp(-(1-\xi)(1-\theta^{-1})(d+1))$ for any $\xi>0$,  
		\be 
		\frac{\theta-\delta}{\theta-1}\log(n/\ell)\geq (1-\delta/\theta)(1-\xi)(d+1)=:(1-\eps)(d+1).
		\ee 
		Applying this inequality to the probability on the right-hand side of~\eqref{eq:neglb} together with the equivalence between sums of exponential random variables and Poisson random variables via Poisson processes, we conclude that
		\be \label{eq:problb}
		\P{X<\frac{\theta-\delta}{\theta-1}\log\Big(\frac{n}{\ell}\Big)}\geq \P{X<(1-\eps)(d+1)}=\P{P_1\geq d+1}\geq\P{P_1=d+1},
		\ee 
		where $P_1\sim \text{Poi}((1-\eps)(d+1))$. With Stirling's formula this yields 
		\be \ba \label{eq:poiprob}
		\P{P_1=d+1}&=\e^{-(1-\eps)(d+1)}\frac{((1-\eps)(d+1))^{d+1}}{(d+1)!}\\
		&=(1+o(1))\e^{\eps(d+1)}(1-\eps)^{d+1}\frac{1}{\sqrt{2\pi d}}\\
		&=(1+o(1))\frac{(1-\eps)\e^\eps}{\sqrt{2\pi d}}\e^{d(\log(1-\eps)+\eps)},
		\ea \ee
		where we observe that the exponent is strictly negative for any $\eps\in(0,1)$. 
		Finally, combining~\eqref{eq:poiprob} with~\eqref{eq:problb} in~\eqref{eq:neglb} and since $(1-\delta)/(\theta-\delta)\geq (1-\delta)/\theta$, we arrive at the lower bound
		\be \label{eq:ctbound}
		(1+o(1))\frac{(1-\theta^{-1})\P{W\in(1-\delta,1]}(1-\eps)\e^\eps}{\sqrt{2\pi d}}\exp(d(\log(1-\eps)+\eps+\log((1-\delta)/\theta))).
		\ee
		By choosing $\delta$ and $\xi$ (used in the definition of $\eps$) sufficiently small, $\log(1-\eps)+\eps$ can be set arbitrarily close to zero (though negative), and $\log((1-\delta)/\theta)=\log(1-\delta)-\log\theta$ can be set arbitrarily close to (though smaller than) $-\log \theta$. Since $-\log \theta >-(\theta -1)$ and $c\in[0,1/(\theta-1)]$, it follows that for some small $\kappa>0$ and $\delta,\xi$ sufficiently small, that for all $n$ sufficiently large,
		\be \label{eq:1/nneg}
		\frac{1}{\sqrt{d}}\exp(d(\log(1-\eps)+\eps+\log((1-\delta)/\theta)))\geq \exp(-(1-\kappa)\log n)=n^{-(1-\kappa)},
		\ee 
		which, together with~\eqref{eq:neglb} yields the desired result
		
		For the case $c\in(1/(\theta-1),\theta/(\theta-1))$, we use the same approach but now use that $d\leq (\theta/(\theta-1))\log n$ for all $n$ large. We thus obtain the lower bound 
		\be 
		\E{\frac{\theta-1}{\theta-1+W}\Big(\frac{W}{\theta-1+W}\Big)^{d}\Pf{X<\Big(1+\frac{W}{(\theta-1)}\Big)\log(n/\ell)}}\geq \e^{-C d}\geq n^{-C\theta/(\theta-1)}, 
		\ee 
		for some large constant $C>0$. The desired result holds for $\gamma>C\theta/(\theta-1)$, which concludes the proof.
	\end{proof}
	
	\begin{lemma}\label{lemma:sumint}
		Fix $\ell,n\in \N$ such that $\ell<n$. Suppose $f:\R\to \R$ is a positive integrable function,  increasing on $[\ell,x^*]$ and decreasing on $[x^*,n]$, 
		where $x^*\in(\ell,n)$ is not necessarily an integer. Then,
		\be 
		\int_{\ell}^n f(x)\, \d x-f(x^*)\leq \sum_{j = \ell +1}^n f(j) \leq \int_{\ell}^n f(x)\, \d x + f(x^*).
		\ee
	\end{lemma}

	\begin{proof}
		As $f$ is increasing on $[\ell,\lfloor x^* \rfloor]$ and decreasing on $[\lceil x^*\rceil,n]$, we directly have that
		\be \ba 
		\sum_{j=\ell+1}^n f(j)=f(\lceil x^*\rceil)+\sum_{j=\ell+1}^{\lfloor x^*\rfloor} f(j)+\sum_{j=\lceil x^*\rceil+1}^n f(j)\leq f(x^*)+\int_\ell^{\lfloor x^*\rfloor}f(x)\,\d x+\int_{\lceil x^*\rceil}^nf(x)\,\d x.
		\ea \ee 
		The final two terms can be combined into a single integral from $\ell$ to $n$ to yield an upper bound, since $f(x)$ is positive for all $x\in\R$.
		
		For the lower bound, we use an equivalent approach and that 
		\be 
		\int_{\lfloor x^*\rfloor}^{\lceil x^*\rceil}f(x)\,\d x\leq f(x^*),
		\ee  
		to obtain the desired lower bound.
	\end{proof}

\end{document}